\documentclass{gtart}
\usepackage{amsmath, amssymb, graphicx, graphs, verbatim, epsf,amscd}
\usepackage{pinlabel}

\newtheorem{thm}{Theorem}[section]
\newtheorem{cor}[thm]{Corollary}
\newtheorem{lem}[thm]{Lemma}
\newtheorem{prop}[thm]{Proposition}

\theoremstyle{definition}
\newtheorem{defn}[thm]{Definition}

\theoremstyle{remark}
\newtheorem{rem}[thm]{Remark}
\newtheorem{remarks}[thm]{Remarks}

\numberwithin{equation}{section}

\newcommand\TransInv{\mathfrak{T}}
\newcommand\TransInvha{\widehat{\TransInv}}
\newcommand{\al}{\alpha}
\newcommand{\be}{\beta}
\newcommand{\ga}{\gamma}

\newcommand{\De}{\Delta}

\newcommand{\La}{{\mathfrak {L}}}
\newcommand{\LegInv}{\La}
\newcommand{\LegInva}{\widehat {\La}}
\newcommand{\la}{\lambda}

\newcommand{\om}{\omega}

\renewcommand{\th}{\theta}

\newcommand{\ze}{\zeta}

\newcommand{\bfz}{{\mathbb {Z}}}
\newcommand{\bfn}{{\mathbb {N}}}
\newcommand{\bfq}{{\mathbb {Q}}}
\newcommand{\bfc}{{\mathbb {C}}}

\newcommand\alphas{\mathbb{\alpha}}
\newcommand\betas{\mathbb{\beta}}
\newcommand\SpinC{{\mathrm{Spin}}^c} 
\newcommand\spinct{\mathbf t}
\newcommand\da{\widehat\partial}
\newcommand\dm{\partial^-}
\newcommand\dmK{\partial^-_K}
\newcommand\Field{\mathbb{F}}
\newcommand\twostars{\mbox{$\star\star$}}
\newcommand\threestars{\mbox{$\star\star\star$}}
\newcommand{\spinc}{\mathfrak{s}}
\newcommand\Ta{\mathbb{T}_{\alpha}}
\newcommand\Tb{\mathbb{T}_{\beta}}
\newcommand\CFa{\widehat{\mathrm{CF}}}
\newcommand\CFm{\mathrm{CF}^-}
\newcommand\HFa{\widehat{\mathrm{HF}}}
\newcommand\HFm{\mathrm{HF}^-}
\newcommand{\Laha}{\widehat {\mathfrak {L}}}
\newcommand{\CFKa}{\widehat{\mathrm{CFK}}}
\newcommand{\HFKa}{\widehat{\mathrm{HFK}}}
\newcommand{\HFKm}{\mathrm{HFK}^-}
\newcommand{\CFKm}{\mathrm{CFK}^-}

\def\co{\colon\thinspace}
\newcommand{\x}{{\bf {x}}}

\newcommand{\y}{{\bf {y}}}
\newcommand{\s}{\mathbf s}

\renewcommand{\t}{\mathbf t}

\newcommand{\Z}{\mathbb Z}

\newcommand{\Q}{\mathbb Q}

\newcommand{\bfr}{\mathbb R}

\newcommand{\del}{\partial}

\DeclareMathOperator{\tb}{tb}
\DeclareMathOperator{\TB}{TB}
\DeclareMathOperator{\rot}{rot}
\DeclareMathOperator{\lk}{lk}

\DeclareMathOperator{\PD}{PD}

\DeclareMathOperator{\Sym}{Sym}
\begin{document}

\title[{Invariants of Legendrian knots}] 
  {Heegaard Floer invariants of Legendrian knots in contact three--manifolds}

\author[Paolo Lisca]{Paolo Lisca}
\address{Dipartimento di Matematica ``L. Tonelli'', Universit\`a di Pisa \\
I-56127 Pisa, Italy} 
\email{lisca@dm.unipi.it}

\author[Peter Ozsv{\'a}th]{Peter Ozsv\'ath}
\address{Department of Mathematics, Columbia University,\\ 
New York 10027, USA}
\email{petero@math.columbia.edu}

\author[Andr{\'a}s I. Stipsicz]{Andr{\'a}s I. Stipsicz}
\address{R\'enyi Institute of Mathematics, Budapest, Hungary and\\
  Department of Mathematics, Columbia University, New York 10027, USA}
\email{stipsicz@renyi.hu} 

\author[Zolt{\'a}n Szab{\'o}]{Zolt{\'a}n Szab{\'o}} 
\address{Department of Mathematics, Princeton University,\\
Princeton, New Jersey 08544, USA}
\email{szabo@math.princeton.edu}

\begin{abstract}  
  We define invariants of null--homologous Legendrian and transverse
  knots in contact 3--manifolds. The invariants are determined by
  elements of the knot Floer homology of the underlying smooth knot.
  We compute these invariants, and show that they do not vanish for
  certain non--loose knots in overtwisted 3--spheres. Moreover, we
  apply the invariants to find transversely non--simple knot types in
  many overtwisted contact 3--manifolds.
\end{abstract} 


\maketitle

\section{Introduction}\label{s:intro}

The reformulation of knot Floer homology (introduced originally in
\cite{OSzknot, Rasmussen}) through grid diagrams \cite{MOS, MOSzT}
provided not only a combinatorial way of computing the knot Floer
homology groups of knots in $S^3$, but also showed a natural way of
defining invariants of Legendrian and transverse knots in the standard
contact 3--sphere $(S^3, \xi_{st})$, see \cite{OSzT}.  As it is shown
in \cite{NOT, VV}, such invariants can be effectively applied to study
transverse simplicity of knot types. The definition of these
invariants relies heavily on the presentation of the knot through a
grid diagram, hence does not generalize directly to Legendrian and
transverse knots in other closed contact 3--manifolds.  The aim of the
present paper is to define invariants for any null--homologous
Legendrian (and transverse) knot $L\subset (Y, \xi )$ (resp.
$T\subset (Y, \xi )$).  As demonstrated by explicit
computations, these constructions give interesting invariants for
Legendrian and transverse knots, even in cases where the ambient
contact structure is overtwisted.

Recall \cite{OSzknot} that a smooth null--homologous
knot $K\subset Y$ in a closed 3--manifold $Y$ gives rise to the
\emph{knot Floer homology groups} $\HFKa(Y,K)$ and $\HFKm(Y,K)$ (see
also Section~\ref{s:prelim} for a short review of Heegaard Floer
homology).  These groups are computed as homologies of appropriate
chain complexes, which in turn result from doubly--pointed Heegaard
diagrams of the pair $(Y,K)$.  For an isotopic pair $K_1,K_2\subset Y$
the corresponding chain complexes are quasi--isomorphic and hence the
homologies are isomorphic. In this sense the knot Floer homology
groups (as abstract groups) are invariants of the isotopy class of the
knot $K$. For the sake of simplicity, throughout this paper we use
Floer homology with coefficients in $\Field:=\Z/2\Z$.

In this paper we define the invariants $\La(L)$ and $\Laha(L)$ of an
oriented null--homologous Legendrian knot $L\subset (Y,\xi)$. For an
introduction to Legendrian and transverse knots see \cite{etnyre}.
(We will always assume that $Y$ is oriented and $\xi$ is cooriented.)
Recall that the contact invariant $c(Y, \xi )$ of a contact
3--manifold $(Y, \xi )$ --- as defined in~\cite{OSzcont} --- is an
element (up to sign) of the Floer homology of $-Y$ (rather than of $Y$).
In the same vein, the invariants $\La(L)$ and $\Laha(L)$ are
determined by elements in the Floer homology of $(-Y, L)$.  If $Y$
admits an orientation reversing diffeomorphism $\mu$ (like $S^3$
does), then the Floer homology of $(-Y, L)$ can be identified with the
Floer homology of $(Y, \mu(L))$.

More precisely, we will show that, given an oriented, null--homologous
Legendrian knot $L$ in a closed contact 3--manifold $(Y,\xi)$, one can
choose certain auxiliary data $D$ which, together with $L$, determine
a cycle $\x(L,D)$ in a complex defining the $\Field[U]$--modules
$\HFKm(-Y,L,\t_\xi)$ or $\HFKa(-Y,L,\t_\xi)$ (with trivial
$U$--action), where $\t_\xi$ is the $\SpinC$ structure on $Y$ induced
by $\xi$ (see Section~\ref{s:prelim}).  A different choice $D'$ of the
auxiliary data determines $\Field[U]$--module automorphisms sending
the class $[\x(L,D)]$ to the class $[\x(L,D')]$. Formulating this a
bit more formally, we can consider the set of pairs $(M,m)$, where $M$
is an $\Field[U]$--module and $m\in M$, and introduce an equivalence
relation by declaring two pairs $(M,m)$ and $(N,n)$ equivalent if
there is an $\Field[U]$--module isomorphism $f\co M\to N$ such that
$f(m)=n$.  Let $[M,m]$ denote the equivalence class of $(M,m)$. Then,
our main result can be stated as follows:

\begin{thm}\label{t:main}
  Let $L$ be an oriented, null--homologous Legendrian knot in the
  closed contact 3--manifold $(Y,\xi)$, and let $\t_\xi$ be the
  $\SpinC$ structure on $Y$ induced by $\xi$. Then, after choosing
  some suitable auxiliary data, it is possible to associate to $L$
  homology classes $\al_{\La}(L)\in\HFKm(-Y, L,\t_\xi)$ and
  $\al_{\Laha}(L) \in\HFKa(-Y,L,\t_\xi)$ such that
\[
\La(L):=[\HFKm(-Y, L,\t_\xi),\al_{\La}(L)]
\]
and
\[
\Laha(L):=[\HFKa(-Y, L,\t_\xi),\al_{\Laha}(L)]
\]
do not depend on the choice of the auxiliary data, and in fact only
depend on the Legendrian isotopy class of $L$.
\end{thm}

We can define multiplication by $U$ on the set of equivalence classes
$[M,m]$ by setting $U\cdot [M,m] := [M,U\cdot m]$. We will say that
$\La(L)$ is \emph{vanishing} (respectively \emph{nonvanishing}) and
write $\La(L)=0$ (respectively $\La(L)\neq 0$) if
$\La(L)=[\HFKm(-Y,L,\t_\xi),0]$ (respectively $\La(L)\neq
[\HFKm(-Y,L,\t_\xi),0]$). Similar conventions will be used for
$\Laha(L)$. Let $-L$ denote the knot $L$ with reversed orientation. It
turns out that the pairs $\La(\pm L)$ and $\Laha(\pm L)$ admit
properties similar to those of the pair of invariants $\lambda_{\pm
}(L)$ of \cite{OSzT}.

One useful feature of the invariant $\La(\pm L)$ (shared with
$\lambda_{\pm}$ from~\cite{OSzT}) is that it satisfies a nonvanishing
property, which can be formulated in terms of the contact invariant
$c(Y,\xi)$ of the ambient contact 3--manifold $(Y,\xi)$.

\begin{thm}\label{t:nonvan}
  If the contact Ozsv\'ath--Szab\'o invariant $c(Y, \xi )\in
  \HFa(-Y,\t_\xi)$ of the contact 3--manifold $(Y, \xi )$ does not
  vanish, then for any oriented Legendrian knot $L\subset (Y, \xi )$
  we have $\La(L)\neq 0$. If $c(Y,\xi)=0$ then for $d$ large enough
  $U^d \cdot \La(L)$ vanishes.
\end{thm}

As it will be explained later, such strong nonvanishing property does
not hold for $\Laha$. Since a strongly symplectically fillable contact
3--manifold has nonzero contact invariant while for an overtwisted
structure $c(Y,\xi)=0$, it follows immediately from
Theorem~\ref{t:nonvan} that

\begin{cor}
If $(Y, \xi)$ is strongly symplectically fillable then for any null--homologous 
Legendrian knot $L\subset (Y, \xi )$ the invariant $\La(L)$ is nonvanishing.  
If $(Y, \xi )$ is overtwisted, then for any Legendrian knot $L$ there 
is $d\geq 0$ such that $U^d \cdot \La(L)$ vanishes. \qed
\end{cor}

A stronger vanishing theorem holds for loose knots. Recall that a
Legendrian knot $L\subset (Y, \xi )$ is \emph{loose} if its complement
contains an overtwisted disk (and hence $(Y, \xi )$ is necessarily
overtwisted). A Legendrian knot $L\subset (Y, \xi )$ is
\emph{non--loose} (or \emph{exceptional} in the terminology of
\cite{EFr}) if $(Y, \xi )$ is overtwisted, but the complement of $L$
is tight.

\begin{thm} \label{t:loose}
If $L\subset (Y, \xi )$ is an oriented, null--homologous and loose Legendrian knot, then $\La(L)=0$.
\end{thm}

Transverse knots admit a preferred orientation, and can be
approximated, uniquely up to negative stabilization, by oriented
Legendrian knots \cite{EFM, EH}. This fact can be used to define
invariants of transverse knots:

\begin{thm}\label{t:transverse}
  Suppose that $T$ is a null--homologous transverse knot in the
  contact 3--manifold $(Y, \xi )$. Let $L$ be a (compatibly oriented)
  Legendrian approximation of $T$. Then, $\TransInv(T):= \La(L)$ and
  $\TransInvha(T) := \Laha(L)$ are invariants of the transverse knot
  type of $T$.
\end{thm}

The proof of this statement relies on the invariance of the Legendrian
invariant under negative stabilization.  After determining the
invariants of stabilized Legendrian unknots in the standard contact
3--sphere and the behaviour of the invariants under connected sum, in
fact, we will be able to determine the effect of both kinds of
stabilization on the invariant $\La$, leading us to

\begin{thm}\label{t:anystab}
Suppose that $L$ is an oriented Legendrian knot and $L^+$, resp. $L^-$
denote the oriented positive, resp. negative stabilizations of
$L$. Then, $\La(L^-)= \La(L)$ and $\La(L^+)=U\cdot\La(L)$.
\end{thm}

The invariants can be effectively used in the study of non--loose
Legendrian knots. Notice that the invariant of a non--loose Legendrian
knot is necessarily a $U$--torsion element, and can be nonvanishing
only if the knot and all its negative stabilizations are non--loose.
In Section~\ref{s:example} a family of non--loose torus knots in
overtwisted contact $S^3$'s are constructed, for which we can
determine the invariants $\La$ by direct computation.

Recall that a knot type $K$ in a contact three-manifold is said to be
{\em transversely non-simple} if there are two transversely
non-isotopic transverse knots in $Y$ in the topological type of $K$
which have the same self-linking number. (For more background on
transverse non-simplicity, see \cite{etnyre, EH}.)  In an overtwisted
contact structure however, this definition admits refinements.  (For
more on knots in overtwisted contact structures, see \cite{Dymara1,
  Dymara2, EFr, EV}.) Namely, it is not hard to find examples of pairs
of \emph{loose} and \emph{non--loose} (Legendrian or transverse) knots
with equal 'classical' invariants. By definition their complements
admit different contact structures (one is overtwisted, the other is
tight), hence the knots are clearly not Legendrian/transverse
isotopic. Examples for this phenomenon will be discussed in
Section~\ref{s:transverse}.  Finding pairs of Legendrian/transverse
knots with equal classical invariants, both loose or both non--loose,
is a much more delicate question. Non--loose pairs of Legendrian knots
not Legendrian isotopic were found in \cite{etcontsurg}; by applying
our invariants we find transverse knots with similar properties:

\begin{thm}\label{t:nonloosenonsimp}
  The knot type $T_{(2,-7)}\# T_{(2,-9)}\subset S^3$ has two
  non--loose, transversely non--isotopic transverse representatives
  with the same self--linking number with respect to the overtwisted
  contact structure $\xi_{12}$ with Hopf invariant $d_3(\xi
  _{12})=12.$
\end{thm}

In fact, by further connected sums we get a more general statement:

\begin{cor}\label{c:nonlooseeverywhere}
  Let $(Y,\xi)$ be a contact 3--manifold with $c(Y, \xi )\neq 0$.  Let
  $\ze$ be an overtwisted contact structure on $Y$ with
  $\t_\ze=\t_\xi$.  Then, in $(Y, \ze )$ there are null--homologous
  knot types which admit two non--loose, transversely non--isotopic
  transverse representatives with the same self--linking number.
\end{cor}

\begin{remarks}
{\bf {(a)}} Notice that to a Legendrian knot $L$ in the standard
contact 3--sphere $(S^3, \xi _{st})$  one can now associate several
sets of invariants: $\lambda_{\pm }(L)$ of \cite{OSzT} and $\La(\pm L)$,
$\Laha(\pm L)$ of the present work. It would be interesting to compare
these elements of the knot Floer homology groups.

\noindent {\bf {(b)}} We also note that recent work of
Honda-Kazez-Mati\'c \cite{HKMuj} provides another invariant for
Legendrian knots $L\subset (Y, \xi )$ through the sutured contact
invariant of an appropriate complement of $L$ in $Y$, using the
sutured Floer homology of Juh{\'a}sz~\cite{Juhasz}. This invariant
seems to have slightly different features than the invariants defined
in this paper; the relationships between these invariants have yet to
be understood.\footnote{Added in proof: the relation between
  $\LegInva$ and the sutured invariant has been recently worked out in
  \cite{SV}.}

\noindent {\bf {(c)}} In~\cite{etcontsurg} arbitrarily many distinct
non--loose Legendrian knots with the same classical invariants are
constructed in overtwisted contact 3--manifolds. In~\cite{EHI} similar
examples were constructed in the standard tight contact $S^3$, using
connected sums of torus knots. These constructions, however, do not 
tell us anything about transverse simplicity of the knot types.

\noindent {\bf {(d)}} Note that connected summing preserves transverse
simplicity~\cite{EHI} while torus knots are transversely simple with
respect to the standard tight contact $S^3$~\cite{EH, Me}. Therefore,
in the standard contact 3--sphere there are no examples such as those
of Theorem~\ref{t:nonloosenonsimp}.
\end{remarks}

The paper is organized as follows. In Section~\ref{s:prelim} we recall
facts about open books and contact structures, Heegaard Floer groups,
and the contact Ozsv\'ath--Szab\'o invariants. In Section~\ref{s:def}
we establish some preliminary results on Legendrian knots and open
books, we define our invariants and we prove Theorem~\ref{t:main}. In
Section~\ref{s:unknot} we compute the invariants of a certain
(stabilized) Legendrian unknot in the standard contact $S^3$.  This
calculation is both instructive and useful: it will be used in the
proof of the stabilization invariance. In Section~\ref{s:basic} we
prove Theorems~\ref{t:nonvan}, \ref{t:loose}
and~\ref{t:transverse}. In Section~\ref{s:example} we determine the
invariant $\Laha$ for some non--loose Legendrian torus knots in
certain overtwisted contact structures on $S^3$.  In
Section~\ref{s:conn} we describe the behaviour of the invariants with
respect to Legendrian connected sum, and then derive
Theorem~\ref{t:anystab}.  In Section~\ref{s:transverse} we discuss
transverse simplicity in overtwisted contact 3--manifolds, give a
refinement of the Legendrian invariant and prove
Theorem~\ref{t:nonloosenonsimp} and
Corollary~\ref{c:nonlooseeverywhere} modulo some technical results
which are deferred to the Appendix.

{\bf {Acknowledgements}}: PSO was supported by NSF grant number
DMS-0505811 and FRG-0244663.  AS acknowledges support from the Clay
Mathematics Institute.  AS and PL were partially supported by OTKA
T49449 and by Marie Curie TOK project BudAlgGeo. ZSz was supported by
NSF grant number DMS-0704053 and FRG-0244663. We would like to thank
Tolga Etg\"u for many useful comments and corrections.

\section{Preliminaries}
\label{s:prelim}

Our definition of the Legendrian knot invariant relies on the few basic
facts listed below.
\begin{itemize}
\item There is a one--to--one correspondence between (isotopy classes of)
contact structures and open book decompositions (up to positive
stabilization), as it was shown by Giroux \cite{Co, Etn, Giroux}. 
\item An idea of Giroux (cf.~\cite{Co, Etn}) allows one 
to construct an open book decomposition of $Y$ compatible with a
contact structure $\xi $, which is also adapted to a given Legendrian
knot $L\subset (Y, \xi )$. 
\item There is a contact invariant $c(Y, \xi )$ of a closed contact
3--manifold $(Y, \xi )$ originally defined in \cite{OSzcont} and
recently reformulated by Honda--Kazez--Mati\'c \cite{HKM} (cf. also
\cite{HKMuj}).
\end{itemize}

We describe these ingredients in more detail in the following
subsections. In Subsection~\ref{subsec:Generalities} we recall how
contact 3--manifold admit open book decompositions adapted to given
Legendrian knots.  In
Subsection~\ref{subsec:UniquenessOfLegendrianKnots} (which is not
logically required by the rest of this article, but which fits in
neatly at this point) we explain how, conversely, an isotopy class of
embedded curves in a page of an open book decomposition gives rise to
a unique isotopy class of Legendrian knots in the associated contact
three--manifold.  In Subsection~\ref{subsec:HF} we recall the basics
of Heegaard Floer homology, mainly to set up notation, and finally in
Subsection~\ref{subsec:HFcont} we recall the construction of the
contact invariant.

\subsection{Generalities on open books and contact structures}
\label{subsec:Generalities}
Recall that an open book decomposition of a 3--manifold $Y$ is a pair
$(B, \varphi)$ where $B\subset Y $ is a (fibered) link in $Y$ and
$\varphi \colon Y-B \to S^1$ is a locally trivial fibration such that
the closure of each fiber, $S_t ={\overline {\varphi ^{-1}(t)}}$
(a \emph{page} of the open book), is a Seifert surface for the
\emph{binding} $B$. The fibration $\varphi $ can  also be determined by
its monodromy $h_{\varphi }\colon S _{+1}\to S _{+1}$, which
gives rise to an element of the mapping class group of the page
(regarded as a surface with boundary).

An open book decomposition $(B,\varphi)$ can be modified by a
classical operation called \emph{stabilization} \cite{St}. The page
$S'$ of the resulting open book $(B',\varphi')$ is obtained from the
page $S$ of $(B,\varphi)$ by adding a 1--handle $H$, while the
monodromy of $(B', \varphi')$ is obtained by composing the monodromy
of $(B,\varphi)$ (extended trivially to $S'$) with a Dehn twist
$D_\ga$ along a simple closed curve $\gamma\subset S'$ intersecting
the cocore of $H$ transversely in a single point. Depending on
whether the Dehn twist is right-- or left--handed, the stabilization is
called \emph{positive} or \emph{negative}. A positive stabilization is
also called a \emph{Giroux stabilization}.

\begin{defn}(Giroux) A contact structure $\xi$ and an open book
  decomposition are \emph{compatible} if $\xi$, as a cooriented
  2--plane field, is the kernel of a contact one--form $\alpha$ with
  the property that $d\alpha$ is a symplectic form on each page, hence
  orients both the page and the binding, and with this orientation
  the binding $B$ is a link positively transverse to $\xi$. In this 
  situation we also say that the contact one--form $\alpha$ is compatible 
  with the open book.
\end{defn}

A theorem of Thurston and Winkelnkemper \cite{TW} can be used to
verify that each open book admits a compatible contact structure.
Moreover, Giroux \cite{Co, Etn, Giroux} proved the following:
\begin{enumerate}
\item
Each contact structure is compatible with some open book decomposition; 
\item
Two contact structures compatible with the same open book are isotopic;
\item
If a contact structure $\xi$ is compatible with an open book $(B,\varphi)$, 
then $\xi$ is isotopic to any contact structure compatible with a positive 
stabilization of $(B,\varphi)$.
\end{enumerate} 
The construction of an open book compatible with a given contact
structure rests on a \emph{contact cell decomposition} of the contact
3--manifold $(Y,\xi)$, see \cite[Section~4]{Etn} or
\cite[Subsection~3.4]{Co}.  In short, consider a $CW$--decomposition
of $Y$ such that its 1--skeleton $G$ is a Legendrian graph, each
2--cell $D$ satisfies the property that the twisting of the contact
structure along its boundary $\partial D$ (with respect to the framing
given by $D$) is $-1$ and the 3--cells are in Darboux balls of $(Y,
\xi )$.  Then, there is a compact surface $R$ (a ribbon for $G$) such
that $R$ retracts onto $G$, $T_pR=\xi _p$ for all $p\in G$ and
$T_pR\neq \xi _p$ for $p\in R-G$. The 3--manifold $Y$ admits an open
book decomposition with binding $\partial R$ and page $R$ which is
compatible with $\xi$. We also have the following (see
\cite[Theorem~4.28]{Etn} or \cite[Theorem~3.4]{Co}):

\begin{thm}(Giroux) Two open books compatible with the same contact structure
admit isotopic Giroux stabilizations.\qed
\end{thm}

The proof of this statement rests on two facts. The first fact is given by

\begin{lem}(\cite[Lemma~4.29]{Etn} or \cite[Proposition~3.7]{Co})
\label{l:cellde}
Each open book decomposition compatible with $(Y, \xi )$ admits a
sequence of Giroux stabilizations such that the resulting open book
comes from a contact cell decomposition. \qed
\end{lem}

The second fact (see the proof of \cite[Theorem~4.28]{Etn}) is that
any two contact cell decompositions can be connected by a sequence of
operations of the following types: (1) a subdivision of a 2--cell by a
Legendrian arc intersecting the dividing set (for its definition see
\cite{EH}) of the 2--cell once, (2) an addition of a 1--cell $c'$ and
a 2--cell $D$ so that $\partial D=c\cup c'$, where $c$ is part of the
original 1--skeleton and the twisting of the contact structure along
$\partial D$ (with respect to $D$) is $-1$, and (3) an addition of a
2--cell $D$ whose boundary is already in the 2--skeleton and satisfies
the above twisting requirement. It is not hard to see that operations
(1) and (2) induce positive stabilizations on the open book associated
to the cell decomposition, while (3) leaves the open book unchanged.

Thus, an invariant of open book decompositions which is constant under
positive stabilizations is, in fact, an isotopy invariant of the
compatible contact structure. Since in the construction of a contact
cell decomposition for a contact 3--manifold $(Y,\xi)$ one can choose
the $CW$--decomposition of $Y$ in such a way that a Legendrian knot
(or link) $L\subset (Y, \xi )$ is contained in its 1--skeleton $G$, we
have:

\begin{prop}(\cite[Corollary~4.23]{Etn}) Given a Legendrian knot $L$
  in a closed, contact 3--manifold $(Y,\xi)$, there is an open book
  decomposition compatible with $\xi$, containing $L$ on a page $S$
  and such that the contact framing of $L$ is equal to the framing
  induced on $L$ by $S$.  The open book can be chosen in such a way
  that $L$ is homologically essential on the page $S$. \qed
\end{prop}

Recall (see e.g.~\cite{Etn}) that any open book is obtained via a
mapping torus construction from a pair (sometimes called an
\emph{abstract open book}) $(S,\varphi)$, where $S$ is an oriented
surface with boundary and $\varphi\colon S\to S$ is an
orientation--preserving diffeomorphism which restricts as the identity
near $\del S$. Our previous observations amount to saying that any triple
$(Y,\xi,L)$, where $L\subset Y$ is a Legendrian knot (or link) is
obtained via the standard mapping torus construction from a triple
$(S,L,\varphi)$, where $L\subset S$ is a homologically essential
simple closed curve.

\begin{defn}
  Let $(S,\varphi)$ be an abstract open book, and let $L\subset S$ be
  a homologically essential simple closed curve. We say that the
  Giroux stabilization $(S', R_\ga\circ\varphi)$ is
  \emph{$L$--elementary} if after a suitable isotopy the curve $\ga$
  intersects $L$ transversely in at most one point.
\end{defn}

The construction of the Legendrian invariant rests on the following:

\begin{prop}\label{p:elementary}
  Suppose that $L\subset (Y,\xi)$ is a Legendrian knot in a contact
  3--manifold.  If the triple $(Y,\xi,L)$ is associated via the
  mapping torus construction with two different triples ${\cal
    T}_i=(S_i,L_i,\varphi_i)$, $i=1,2$, then $(Y,\xi,L)$ is also
  associated with a triple $(S,L,\varphi)$, obtained from each of
  ${\cal T}_1$ and ${\cal T}_2$ by a finite sequence of
  $L$--elementary Giroux stabilizations.
\end{prop}

\begin{proof}
  The proof of \cite[Lemma~4.29]{Etn} uses $L$--elementary
  stabilizations only (cf. \cite[Figure~12]{Etn}) and (1)--(2) after
  Lemma~\ref{l:cellde} determine $L$--elementary stabilizations as
  well, while (3) leaves the open book decomposition unchanged.
\end{proof}

\subsection{From curves on a page to Legendrian knots: uniqueness}
\label{subsec:UniquenessOfLegendrianKnots}

Giroux's results give a correspondence between Legendrian knots and
knots on a page in an open book decomposition. In fact, with a little
extra work, this correspondence can be suitably inverted.  Although
this other direction is not strictly needed for our present
applications (and hence, the impatient reader is free to skip the
present subsection), it does fit in naturally in the discussion at
this point. Specifically, in this subsection, we prove the following:

\begin{thm}\label{t:legendrianuniqueness}
  Let $(B,\varphi)$ be an open book decomposition of the closed,
  oriented 3--manifold $Y$.  Let $\xi=\ker(\al)$ be a contact
  structure, with $\al$ a contact 1--form compatible with
  $(B,\varphi)$. Suppose that $K\subset S:=\varphi^{-1}(1)$ is a
  smooth knot defining a nontrivial homology class in $H_1(S; \Z )$.
  Then the smooth isotopy class of $K$ on the page uniquely determines
  a Legendrian knot in $(Y,\xi)$ up to Legendrian isotopy.
\end{thm}

The proof of this theorem rests on two technical lemmas.

\begin{lem}\label{l:vanishingforms}
  Let $(B,\varphi)$ be an open book decomposition of the closed,
  oriented 3--manifold $Y$.  Let $K_t\subset S:=\varphi^{-1}(1)$ be a
  smooth family of knots which are homologically nontrivial on the
  page and provide an isotopy from $K_0$ to $K_1$.  Then, there exists
  a smooth family of contact 1--forms $\al_{K_t}\in \Omega^1(Y)$
  compatible with $(B,\varphi)$ and such that the restriction of
  $\al_{K_t}$ to $K_t$ vanishes.
\end{lem}

\begin{proof} 
In view of the argument given in~\cite[pp. 115--116]{Etn}, it suffices
to show that there exists a smooth family $\la_t\in\Omega^{1}(S)$ such
that, for each $t$
\begin{enumerate}
\item
$\la_t=(1+s)d\th$ near $\del S$, with coordinates $(s,\th)\in [0,1]\times S^1$ 
near each boundary component of $S$;
\item
$d\la_t$ is a volume form on $S$;
\item
$\la_t$ vanishes on $K_t$.
\end{enumerate}
To construct the family $\la_t$ we proceed as follows. For each $t$,
choose a closed collar $U$ around $K_t\subset S$, parametrized by
coordinates $(s,\th)\in [-1,1]\times S^1$, so that $ds\wedge d\th$ is
a volume form on $U$ with the orientation induced from $S$. Let
$\la_{1,t}\in\Omega^1(S)$ be of the form $(1+s)d\th$ near $\del S$ and
of the form $sd\th$ on $U$. We have
\[
\int_S d\la_{1,t} = \int_{\del S} \la_{1,t} = 2\pi |\del S|.
\]
Let $\om$ be a volume form on $S$ such that: 
\begin{itemize} 
\item
$\int_S \om = 2\pi |\del S|$;
\item
$\om=ds\wedge d\th$ near $\del S$ and on $U$.
\end{itemize}  
(The first condition can be fulfilled since $K _t\subset S$ is
nontrivial in homology, hence each component of its complement meets
$\partial S$.) Let $U'\subset U$ correspond to $[-1/2,1/2]\times
S^1\subset [-1,1]\times S^1$.  We have
\[
\int_{\overline{S\setminus U'}} \om - d\la_{1,t} = \int_S \om - d\la_{1,t} = 0.
\]
Since $\om - d\la_{1,t} =0$ near $\del(\overline{S\setminus U'})$, by de Rham's
theorem there is a compactly supported 1--form $\be_t\in\Omega^1({\overline{S\setminus U'}})$ 
such that $d\be_t = \om - d\la_{1,t}$ on ${\overline{S\setminus U'}}$. 
Let $\tilde\be_t$ be the extension of $\be_t$ 
to $S$ by zero. Then, 
\[
\la_t:=\la_{1,t}+\tilde\be_t\in\Omega^1(S)
\] 
satisfies (1), (2) and (3) above, and the dependence on $t$ can be clearly arranged 
to be smooth.  
\end{proof}

\begin{lem}\label{l:theknotL}
Let $(B,\varphi)$ be an open book decomposition of the closed,
oriented 3--manifold $Y$.  Let $\xi=\ker(\al)$ be a contact structure,
with $\al$ a contact 1--form compatible with $(B,\varphi)$. Let
$K\subset (Y,\xi)$ be a smooth knot contained in the page
$S:=\varphi^{-1}(1)$ with $K$ homologically nontrivial in $S$. Then, the
quadruple $(B,\varphi,\al,K)$ determines, uniquely up to Legendrian
isotopy, a Legendrian knot
\[
L=L(B,\varphi,\al,K)\subset (Y,\xi)
\]
smoothly isotopic to $K$. 
\end{lem} 

\begin{proof}
  Notice first that Giroux's proof that two contact structures
  compatible with the same open book are isotopic shows that the space
  $C\subset\Omega^1(Y)$ of contact 1--forms compatible with
  $(B,\varphi)$ is connected and simply connected. In fact, given
  $\al_0,\al_1\in C$, one can first deform each of them inside $C$ to
  $\al_{0,R}, \al_{1,R}\in C$, where $R\geq 0$ is a constant, so that
  when $R$ is large enough, the path $(1-s)\al_{0,R}+s\al_{1,R}$ from
  $\al_{0,R}$ to $\al_{1,R}$ is inside $C$ (see~\cite{Etn}). This
  proves that $C$ is connected. A similar argument shows that $C$ is
  simply connected. In fact, given a loop $\mathcal L=\{\al_z\}\subset
  C$, $z\in S^1$, one can deform it to $\mathcal
  L_R=\{\al_{z,R}\}\subset C$. By compactness, when $R$ is large
  enough, $\cal L_R$ can be shrunk in $C$ onto $\{\al_{1,R}\}$ by
  taking convex linear combinations.

By our assumptions $\al\in C$ and by Lemma~\ref{l:vanishingforms} there is a 1--form $\al_K\in C$ 
whose restriction to $K$ vanishes. We can choose a path $P$ in the space $C$ 
connecting $\al_K$ to $\al$. Then, setting $\xi_K:=\ker(\al_K)$, by Gray's theorem 
there is a contactomorphism 
\[
\Phi=\Phi(\al,\al_K,P)\co (Y,\xi_K)\to (Y,\xi).
\] 
We define $L:=\Phi(K)$. Since $\Phi$ is smoothly isotopic to the identity, 
$L$ is smoothly isotopic to $K$. Hence, to prove the lemma it suffices to show that 
changing our choices of $\al$, $\al_K$ or $P$ only changes $L$ by a Legendrian 
isotopy. Observe that if $P'$ is another path from $\al_K$ to $\al$, since 
$C$ is simply connected there exists a family $P_t$ of paths from $\al_K$ to $\al$ 
connecting $P$ to $P'$. This yields a family of contactomorphisms 
\[
\Phi_t\co (Y,\xi_K) \to (Y,\xi)
\]
and therefore a family of Legendrian knots $L_t:=\Phi_t(K)\subset
(Y,\xi)$ with $L_0=L$. Thus, $L$ only depends, up to
Legendrian isotopy, on the endpoints $\al_K$ and $\al$ of the path
$P$. Suppose now that we chose a different 1--form $\al'_K\in C$ whose
restriction to $K$ vanishes. Then, we claim that there is a smooth
path $\al_{K,s}\subset C$ from $\al_K$ to $\al'_K$ such that the
restriction of $\al_{K,s}$ to $K$ vanishes for every $s$. In fact,
such a path can be found by first deforming each of $\al_K$ and
$\al'_K$ to the forms $\al_{K,R}$ and $\al'_{K,R}$ obtained by adding
multiples of the standard angular 1--form suitably modified near the
binding, and then taking convex linear combinations (see~\cite{Etn}).
Neither of the two operations alters the vanishing property along $K$,
therefore this proves the claim. Now we can find a smooth family $Q_s$
of paths in $C$, such that $Q_s$ joins $\al_{K,s}$ to $\al$ for every
$s$. This produces a family of Legendrian knots
$\phi(\al,\al_{K,s},Q_s)(K)\subset (Y,\xi)$, thus proving the
independence of $L$ from $\al_K$ up to Legendrian isotopy. The
independence from $\al$ can be established similarly: if $\al'\in C$
and $\xi=\ker(\al')$ then $\al_s :=(1-s)\al+s\al'\in C$ and
$\xi=\ker(\al_s)$ for every $s$. Therefore, we can find a smooth
family of paths $R_s$ in $C$, with $R_s$ joining $\al_K$ to $\al_s$
for every $s$, and proceed as before.
\end{proof}

\begin{proof}[Proof of Theorem~\ref{t:legendrianuniqueness}] 
  It suffices to show that if $K_t\subset S$ is a family of smooth
  knots which gives a smooth isotopy from $K_0$ to $K_1$, then, the
  Legendrian knots $L_0=L(B,\varphi,\al,K_0)$ and
  $L_1=L(B,\varphi,\al,K_1)$ determined via Lemma~\ref{l:theknotL} are
  Legendrian isotopic. By Lemma~\ref{l:vanishingforms} there exists a
  smooth family of contact 1--forms $\al_{K_t}\in \Omega^1(Y)$
  compatible with $(B,\varphi)$ and such that the restriction of
  $\al_{K_t}$ to $K_t$ vanishes. Applying the construction of
  Lemma~\ref{l:theknotL} for each $t$ we obtain the required
  Legendrian isotopy
\[
L_t:=L(B,\varphi,\al, K_t)\subset (Y,\xi).
\]
\end{proof} 

\subsection{Heegaard Floer homologies}
\label{subsec:HF}
The Heegaard Floer homology groups $\HFm(Y), \HFa(Y)$ of a 3--manifold
were introduced in~\cite{OSzF} and extended in the case where $Y$ is
equipped with a null-homologous knot $K\subset Y$ to variants
$\HFKm(Y,K), \HFKa(Y,K)$ in~\cite{OSzknot, RasmussenThesis}. For the
sake of completeness here we quickly review the construction of these
groups, emphasizing the aspects most important for our present
purposes.

We start with the closed case.
An oriented 3--manifold $Y$ can be conveniently presented by a
Heegaard diagram, which is an ordered triple
$(\Sigma,\alphas,\betas)$, where $\Sigma$ is an oriented genus--$g$
surface, ${\bf {\alpha }}=\{ \alpha _1, \ldots , \alpha _g \}$ (and
similarly ${\bf {\beta }}=\{ \beta _1, \ldots , \beta _g \}$) is a
$g$--tuple of disjoint simple closed curves in $\Sigma$, linearly
independent in $H_1(\Sigma ; \bfz )$.  The $\alpha$--curves can be
viewed as belt circles of the 1--handles, while the $\beta$--curves as
attaching circles of the 2--handles in an appropriate handle
decomposition of $Y$.  We can assume that the $\alpha$-- and $\beta
$--curves intersect transversely. Consider the tori ${\mathbb
  {T}}_{\alpha }=\alpha _1 \times \ldots \times \alpha _g$, ${\mathbb
  {T}}_{\beta }=\beta _1 \times \ldots \times \beta _g$ in the
$g^{th}$ symmetric power $\Sym ^g (\Sigma)$ of $\Sigma$ and define
$\CFm (Y)$ as the free $\Field [U]$--module generated by the elements
of the transverse intersection ${\mathbb {T}}_{\alpha }\cap
{\mathbb{T}}_{\beta }$.  (Recall that in this paper we assume
$\Field=\Z/2\Z$.  The constructions admit a sign refinement to
$\Z[U]$, but we do not need this for our current applications.)  For
appropriate symplectic, and compatible almost complex structures
$(\omega , J)$ on $\Sym ^g(\Sigma )$, $\x , \y\in {\mathbb {T}}_{\alpha
}\cap {\mathbb {T}}_{\beta }$, and relative homology class $\phi \in
\pi _2 (\x ,\y)$, we define ${\mathfrak {M}}(\phi )$ as the moduli
space of holomorphic maps from the unit disk $D\subset \bfc$ to $(\Sym
^g(\Sigma ), J)$ with the appropriate boundary conditions (cf.
\cite{OSzF}).  Take $\mu(\phi)$ to be the formal dimension of
${\mathfrak {M}}(\phi )$ and ${\widehat {\mathfrak {M}}}(\phi )
={\mathfrak {M}}(\phi )/{\bfr }$ the quotient of the moduli space by
the translation action of $\bfr$.

An equivalence class of nowhere zero vector fields (under homotopy
away from a ball) on a closed 3--manifold is called a \emph{Spin$^c$
  structure}. It is easy to see that a cooriented
contact structure $\xi$ on a closed 3--manifold naturally induces a
$\SpinC$ structure: this is the equivalence class of the oriented unit
normal vector field of the 2--plane field $\xi$.

Fix a point $w\in \Sigma - {\bf {\alpha }}-{\bf {\beta }}$, and for
$\phi \in \pi _2 (\x , \y )$ denote the algebraic intersection number
$\# (\phi\cap\{w\}\times \Sym ^{g-1}\Sigma ) $ by $n_w(\phi )$.  With
these definitions in place, the differential $\partial ^- \colon \CFm
(Y)\to \CFm (Y)$ is defined as
\[
\partial ^- \x =\sum _{\y \in {\mathbb {T}}_\alpha  
\cap {\mathbb {T}}_\beta} \sum _{\phi \in \pi _2 (\x, \y ), \mu(\phi)=1}
\# {\widehat {\mathfrak {M}}}(\phi)\cdot U^{n_w(\phi )}\cdot \y.
\]
With the aid of $w$ the elements of  
${\mathbb {T}}_\alpha  \cap {\mathbb {T}}_\beta $ can be partitioned according 
to the $\SpinC$ structures of $Y$, resulting a decomposition
\[
\CFm (Y)=\oplus _{\t \in \SpinC (Y)}\CFm (Y, \t ),
\] 
and the map $\partial ^-$ respects this decomposition.  If the
technical condition of \emph{strong admissibility} (cf.~\cite{OSzF})
is satisfied for the Heegaard diagram $(\Sigma , \alpha , \beta , w)$
of $(Y, \t )$, the chain complex $(\CFm (Y, \t ), \partial ^-)$ results
in a group $\HFm (Y, \t )$ which is an invariant of the $\SpinC$
3--manifold $(Y, \t )$. Strong admissibility of $(\Sigma , \alpha ,
\beta , w)$ for $(Y, \t )$ can be achieved as follows: consider a
collection $\{ \gamma _1, \ldots , \gamma _n \}$ of curves in $\Sigma
$ generating $H_1(Y; \bfq )$.  It is easy to see that such $\gamma _i$
($i=1, \ldots , n $) can be found for any Heegaard decomposition.
Then, by applying sufficiently many times a specific isotopy of the
$\beta$--curves along each $\gamma_i$ (called `spinning', the exact
amount depending on the value of $c_1 (\t )$ on a basis of $H_2(Y;
\bfq )$, cf. \cite{OSzF}) one can arrange the diagram to be strongly
admissible.

By specializing the $\Field [U]$--module $(\CFm (Y, \t ), \partial
^-)$ to $U=0$ we get a new chain complex $(\CFa(Y), \da)$, resulting
in an invariant $\HFa(Y, \t )$ of the $\SpinC$ 3--manifold $(Y, \t )$.
In addition, if $c_1(\t )$ is a torsion class, then the homology
groups $\HFm (Y, \t )$ and $ \HFa(Y, \t)$ come with a $\bfq $--grading,
cf.~\cite{AbsGraded}, and hence split as
\[
\HFm (Y, \t )=\oplus_d \HFm _d (Y, \t), \ \ \ \HFa(Y, \t ) =\oplus_d \HFa_d (Y,
\t).
\]

Since $\CFa(Y)$ is generated over $\Field $ by the elements of
${\mathbb {T}}_{\alpha }\cap {\mathbb {T}}_{\beta}$, the Floer
homology group $\HFa(Y)$ and hence also each $\HFa(Y, \t) $ are
finitely generated $\Field$--modules. There is a long exact sequence
\[
\ldots \longrightarrow \HFm_d(Y, \t ) \stackrel{\cdot
U}{\longrightarrow} \HFm _{d-2}(Y, \t) \longrightarrow \HFa_{d-2}(Y,
\t) \longrightarrow \ldots
\]
which establishes a connection between the two versions of the theory.

Suppose now that we fix two distinct points
$$w,z\in \Sigma-\alpha_1-...-\alpha_g-\beta_1-...-\beta_g,$$
where
$(\Sigma,\alpha,\beta, w)$ is a Heegaard diagram for $Y$.  The ordered
pair of points $(w,z)$ determines an oriented knot $K$ in $Y$ by the
following convention. We consider an embedded oriented arc $\zeta$ in
$\Sigma$ from $z$ to $w$ in the complement of the $\alpha$--arcs, and
let $\eta$ be an analogous arc from $w$ to $z$ in the complement of
the $\beta$--arcs. Pushing $\zeta$ and $\eta$ into the $\alpha$-- and
$\beta$--handlebodies we obtain a pair of oriented arcs $\zeta '$ and
$\eta'$ which meet $\Sigma$ at $w$ and $z$.  Their union now is an
oriented knot $K\subset Y$. We call the tuple
$(\Sigma,\alpha,\beta,w,z)$ a Heegaard diagram compatible with the
oriented knot $K\subset Y$. (This is the orientation convention
from~\cite{OSzknot}; it is opposite to the one from~\cite{MOSzT}.)

We have a corresponding differential, defined by
\[ 
\dmK\x =\sum
_{\y \in {\mathbb {T}}_\alpha \cap {\mathbb {T}}_\beta} \sum _{\phi \in \pi _2
(\x, \y ), \mu(\phi )=1, n_z(\phi )=0} 
\# {\widehat {\mathfrak {M}}}(\phi )\cdot
U^{n_w(\phi )}\cdot \y .
\] 
Using this map we get a chain complex $(\CFKm (Y), \partial^-_K)$.
This group has an additional grading, which can be formulated in terms
of {\em relative Spin$^c$ structures}, which are possible extensions
of $\t\vert_{Y-\nu K}$ to the zero--surgery $Y_0(K)$ along the
null--homologous knot $K$. Given $K$, there are infinitely many
relative $\SpinC$ structures with a fixed background $\SpinC$
structure on the 3--manifold $Y$.  By choosing a Seifert surface $F$
for the null--homologous knot $K$, we can extract a numerical
invariant for relative $\SpinC$ structures, gotten by half the value
of $c_1(\s)$ on $F$ (which is defined as the integral of the first
Chern class of the corresponding $\SpinC$ structure of the 0--surgery
on the surface ${\hat {F}}$ we get by capping off the surface $F$).
Note that the sign of the result depends on the fixed orientation of
$K$. The induced $\Z$-grading on the knot Floer complex is called its
\emph{Alexander grading}.  When $b_1(Y)=0$, this integer, together
with the background $\SpinC$ structure $\t$, uniquely specifies the
relative $\SpinC$ structure; moreover, the choice of the Seifert
surface becomes irrelevant, except for the overall induced orientation
on $K$.

For a null--homologous knot $K\subset Y$ the homology group
$\HFKm(Y,K, \s)$ of the above chain complex (with relative $\SpinC$
structure $\s$) is an invariant of $(Y, K, \s)$, and is called the
\emph{knot Floer homology} of $K$. The specialization $U=0$ of the
complex defines again a new complex $(\CFKa(Y,K,\s), \da _K)$ with
homology denoted by $\HFKa(Y,K,\s)$. The homology groups
$\HFKa(Y,K,\s)$ and $\HFKm(Y,K,\s )$ are both finitely generated
vector spaces over $\Field$. 

An alternative way to view this construction is the following. Using one 
basepoint $w$ one can define the chain complex $(\CFm (Y,\t), \dm)$ as 
before, and with the aid of the other basepoint $z$ one can equip this 
chain complex with a filtration. As it was shown in \cite{OSzknot}, 
the filtered chain homotopy type of the
resulting complex is an invariant of the knot, and the Floer homology
groups can be defined as the homology of the associated graded object.

The restriction of a relative $\SpinC$ structure $\s$ to the
complement of $K$ extends to a unique $\SpinC$ structure $\t$ on $Y$.
The map induced by multiplication by $U$ changes the relative $\SpinC$
structure, but it preserves the background $\SpinC$ structure
$\t\in\SpinC(Y)$. Thus, we can view
\[
\CFKm(Y,K,\t):= \oplus _{\s \ {\mbox{{\small {restricts to}}}} \ \t}\CFKm(Y,K,\s),
\]
\[
\HFKm(Y,K,\t) := \oplus _{\s \ {\mbox{{\small {restricts to}}}} \ \t}\HFKm(Y,K,\s )
\]
and 
\[
\HFKa(Y,K,\t) := \oplus _{\s \ {\mbox{{\small {restricts to}}}} \ \t}\HFKa(Y,K,\s )
\]
as modules over $\Field[U]$ (where the $U$--action on $\HFKa(Y,K,\t)$
is trivial).  If $\t$ is torsion, then there is an absolute
$\Q$--grading on these modules, as in the case of closed 3--manifolds.
As before, the two versions of knot Floer homologies are connected by
the long exact sequence
\[
\cdots \longrightarrow \HFKm_d(Y, K, \t ) \stackrel{\cdot
U}{\longrightarrow} \HFKm_{d-2}(Y, K, \t) \longrightarrow 
\HFKa_{d-2}(Y, K, \t) \longrightarrow \cdots
\]

A map $F \colon \HFKm(Y,K,\t)\to \HFa(Y, \t)$ can be defined,
which is induced by the map 
\[
f\colon \CFKm (Y,K,\t) \to \CFa(Y, \t)
\]
by setting $U=1$ and taking $z$ to play the role of $w$ in the complex
$\CFa(Y, \t)$.  According to the definition, this specialization
simply disregards the role of the basepoint $w$. Indeed, the fact
that a generator of $\CFKm (Y,K)$ belongs to the summand $\CFKm
(Y,K,\t)$ is determined only by the point $z$. On the chain level this
map fits into the short exact sequence
\begin{equation}\label{e:short}
0\to \CFKm (Y,K,\t) \stackrel{U-1}{\longrightarrow} \CFKm (Y,K,\t)
\stackrel{f}{\longrightarrow} \CFa(Y,\t )\to 0 ,
\end{equation}
since $f$ is obviously surjective. 

\begin{lem}\label{l:kernel}
Let $F\colon \HFKm(Y,K,\spinct)\longrightarrow \HFa(Y,\spinct)$ 
be map induced on homology by $f$. The kernel of $F$ consists of 
elements $x$ of the form $(U-1)y$. Moreover, an element $x\in\HFKm(Y,K,\spinct)$ 
is both in the kernel of $F$ and homogeneous, i.e.~contained in the 
summand determined by a fixed relative $\SpinC$ structure, if and only if 
$x$ satisfies $U^nx=0$ for some $n\geq 0$.
\end{lem}

\begin{proof}
  The long exact sequence associated to Exact Sequence \eqref{e:short}
  identifies $\ker F$ with $(U-1)\cdot \HFKm(Y,K,\t)$. The only
  statement left to be proved is the characterization of homogeneous
  elements in $\ker F$. If $U^n x=0$ then
\[
x=(1-U^n)x =(1-U)(x+Ux+ \ldots + U^{n-1}x),
\]
hence $x$ is of the form $x=(1-U)y$. Conversely, if $x = (1-U)y$ we
may assume without loss that $y = \sum _{i=0}^{n-1} y_i$, where each
$y_i$ is homogeneous, $y_0=x$ and, for $i=0,\ldots,n-1$, $Uy_i$
belongs to the same relative $\SpinC$ structure as $y_{i+1}$. A simple
cancellation argument shows that $y_1=Uy_0, y_2=U y_1,\ldots $ hence
by the finiteness of the sum we get $U^n x=0$.
\end{proof}

\subsection{Contact Ozsv\'ath--Szab\'o invariants}
\label{subsec:HFcont}

Next we turn to the description of the contact Ozsv\'ath--Szab\'o invariant of
a closed contact 3--manifold as it is given in \cite{HKM}.  (See
\cite{OSzcont} for the original definition of these invariants.)  Suppose 
that $(B, \varphi )$ is an open book decomposition of the 3--manifold $Y$
compatible with the given contact structure $\xi$. Consider a \emph{basis}
 $\{ a_1, \ldots , a_n\}$ 
of the page $S _{+1}$, that is, take a collection of disjoint properly
embedded arcs $\{ a_1, \ldots , a_n\}$ such that $S _{+1}-\cup _{i=1}^n
a_i$ is connected and simply--connected (therefore it is homeomorphic to a
disk). Let $b_i$ be a properly embedded arc obtained by a small isotopy of $a_i$
so that the endpoints of $a_i$ are isotoped along $\partial S
_{+1}$ in the direction given by the boundary orientation, and $a_i$
intersects $ b_i$ in a unique transverse point in int$S _{+1}$, cf.
Figure~\ref{f:deform} and \cite[Figure~2]{HKM}.
\begin{figure}[ht!]
\labellist
\small\hair 2pt
\pinlabel $S$ at 94 94
\pinlabel $a_i$ at 165 24
\pinlabel $b_i$ at 231 24
\endlabellist
\centering
\includegraphics[scale=0.4]{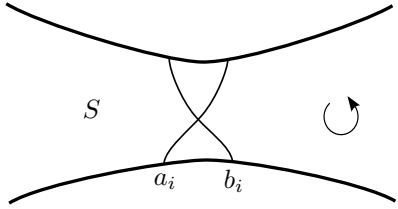}
\caption{{\bf The arcs $a_i $ and $b_i$.}}
\label{f:deform}
\end{figure}
Considering ${\overline {a_i}}=a_i$ and ${\overline {b_i}}=h_{\varphi
} (b_i)$ in $S _{-1}={\overline {\varphi ^{-1}(-1)}}$ (where
$h_{\varphi }$ denotes a diffeomorphism representing the monodromy of
the given open book), it is shown in \cite{HKM} that the triple
\[
( S _{+1}\cup (-S _{-1}), \{ a_i \cup {\overline {a_i}} \}_{i=1}^n, \{
b_i \cup {\overline {b_i}}\} _{i=1} ^n)=(S_{+1}\cup (-S_{-1}),\alpha
,\beta )
\]
is a Heegaard diagram for $Y$. (Notice that we have the freedom of
choosing $h_{\varphi}$ within its isotopy class; this freedom will be
used later, cf. Proposition~\ref{p:strongadm}.)  For technical
purposes, however, we consider the triple
\begin{equation}\label{e:ford}
( S _{+1}\cup (-S _{-1}), \{ b_i \cup {\overline {b_i}}\} _{i=1} ^n,
\{ a_i \cup {\overline {a_i}} \}_{i=1}^n)=(S_{+1}\cup (-S_{-1}),\beta
, \alpha),
\end{equation}
which is now a Heegaard diagram for $-Y$. With this choice, and the
careful placement of the basepoint $w$ we can achieve that the
proposed chain in the chain complex for defining the Heegaard Floer
group is, in fact, a cycle.  More formally, put the basepoint $w$ in
the disk $S _{+1}-\cup _{i=1}^n a_i$ outside of the small strips
between the $a_i$'s and the $b_i$'s and consider the element $\x (B,
\varphi )= \{a_i\cap b_i\}$ in the chain complex corresponding to the
Heegard diagram \eqref{e:ford} above (pointed by $w$). It is not hard
to see (cf. \cite{HKM}) that with these choices the Heegaard diagram
is weakly admissible. In the following we always have to keep in mind
the reversal of the $\alpha$-- and $\beta$--curves when working with
the contact (or Legendrian) invariants.  

\begin{thm}[\cite{HKM}, cf. also \cite{HKMuj}] \label{t:contel} The chain $\x
  (B, \varphi )$ defined above is closed when regarded as an element
  of the Heegaard Floer chain complex $\CFa(-Y)$. The homology class
  $[\x (B, \varphi )]\in\HFa (-Y)$ defined by $\x (B, \varphi
  )$ is independent (up to sign) of the chosen basis and compatible
  open book decomposition. Therefore the homology class $[\x (B,
  \varphi )]$ is an invariant of the contact structure $(Y, \xi)$.
\end{thm}

The original definition of this homology class is given
in~\cite{OSzcont}, which leads to a different Heegaard diagram.  It
can be shown that the two invariants are identified after a sequence
of handleslides; though one can work directly with the above
definition, as in~\cite{HKM}. We adopt this point of view, supplying
an alternative proof of the invariance of the contact class (which
will assist us in the definition of the Legendrian invariant).

\begin{proof}[Alternative proof of Theorem~\ref{t:contel}.]
  In computing $\partial \x (B,\varphi )$ we need to encounter
  holomorphic disks which avoid $w$ but start at $\x (B, \varphi )$.
  By the chosen order of $a_i$ and $b_i$ (resulting in a Heegaard
  diagram of $-Y$ rather than $Y$) we get that such holomorphic disk
  does not exist. In fact, there are no Whitney disks
  $\phi\in\pi_2(\x,\y)$ for any $\y$ with $n_w(\phi)=0$ and whose
  local multiplicities are all non--negative,
  cf.~\cite[Section~3]{HKM}. Thus, the intersection point represents a
  cycle in the chain complex. The independence of $[\x (B, \varphi )]$
  from the basis is given in \cite[Proposition~3.3]{HKM}.  The
  argument relies on the observation that two bases of $S _{+1}$ can
  be connected by a sequence of arc slides \cite[Section~3.3]{HKM},
  inducing handle slides on the corresponding Heegaard diagrams.  (We
  will discuss a sharper version of this argument in
  Proposition~\ref{p:cutsyst}.)  As it is verified in
  \cite[Lemma~3.4]{HKM}, these handle slides map the corresponding
  $\x$'s into each other.

  In order to show independence of the chosen open book decomposition, we only
  need to verify that if $(B', \varphi ') $ is the result of a Giroux
  stabilization of $(B, \varphi)$, then there are appropriate bases,
  giving Heegaard decompositions, for which the  map
\[
\Phi \colon \CFa (B, \varphi ) \to \CFa (B', \varphi ') 
\]
induced by the stabilization satisfies $\Phi (\x (B, \varphi ))=\x (B',
\varphi ')$.  Let us assume that the right--handed Dehn twist of the Giroux
stabilization is equal to $D_{\gamma}$, where $\gamma $ is a simple closed
curve in the page $S '$ of $(B', \varphi ')$, and $\gamma _1$ is the
portion of it inside the page $S$ of $(B, \varphi )$. Our aim is to find
a basis $\{ a_1, \ldots , a_n \}$ for $S$ which is disjoint from
$\gamma _1$.  If $S - \gamma _1$ is connected, then choose $a_1$ to be (a
little push--off of) $\gamma _1$, and extend $\{ a_1 \}$ to a basis for
$S$. If $S - \gamma _1$ is disconnected then the union of any 
bases of the components (possibly after isotoping some endpoints along
$\gamma _1$) will be appropriate.

Now consider the basis for $S '$ obtained by extending the above basis
for $S$ with the cocore $a_{n+1}$ of the new 1--handle. Since $\gamma$
is disjoint from all $a_i$ ($i\leq n$), we clearly get that $\alpha
_{n+1}=a_{n+1}\cup {\overline {a_{n+1}}}$ will be intersected only
(and in a single point $y_{n+1}$) by $\beta _{n+1}=b_{n+1}\cup
{\overline {h_{\varphi} (b_{n+1})}}$. Therefore the map $\Phi$ sending a
  generator $(y_1, \ldots , y_n)$ of $\CFa (B, \varphi , \{
  a_1, \ldots , a_n\} )$ to $(y_1, \ldots ,y _n , y_{n+1})$ (with the
  last coordinate being the unique intersection $\alpha _{n+1}\cap
  \beta _{n+1}=\{ y_{n+1}\}$) establishes an isomorphism
\[
\Phi \colon \CFa (B, \varphi , \{ a_1, \ldots , a_n\} ) \to \CFa (B',
\varphi ', \{ a_1, \ldots , a_n, a_{n+1}\} )
\]
between the underlying Abelian groups. Since $\alpha _{n+1}$ contains
a unique intersection point with all the $\beta$--curves, a
holomorphic disk encountered in the boundary map must be constant at
$y_{n+1}$, hence $\Phi$ is a chain map.  Since it maps $\x (B, \varphi
) $ to $\x (B', \varphi ') $, the proof is complete. (See also
\cite[Section~3]{HKMuj}.)
\end{proof}
\begin{rem}
{\rm The basic properties (such as the vanishing for overtwisted and
nonvanishing for Stein fillable structures, and the transformation
under contact $(+1)$--surgery) can be directly verified for the above
construction, cf.~\cite{HKM}.  }
\end{rem}

In our later arguments we will need that the Heegaard diagram can
be chosen to be strongly admissible, hence we address this issue presently, 
using an argument which was first used in~\cite{olga}.

\begin{prop}\label{p:strongadm}
For any $\SpinC$ structure $\t$ the monodromy map $h_{\varphi }$ of
the given open book decomposition can be chosen in its isotopy class
in such a way that the Heegaard diagram defined before
Theorem~\ref{t:contel} is strongly admissible for $\t$.
\end{prop}
\begin{proof}
  Recall that strong admissibility of a Heegaard diagram for a given
  $\SpinC$ structure $\t$ can be achieved by isotoping the
  $\beta$--curves through spinning them around a set of curves $\{
  \gamma _1, \ldots , \gamma _n\} $ in the Heegaard surface $\Sigma$
  representing a basis of $H_1(Y)$.  It can be shown that for a
  Heegaard diagram coming from an open book decomposition, we can
  choose the curves $\gamma _i$ all in the same page, hence all
  $\gamma _i$ can be chosen to be in $S _{-1}$. Since the spinnings
  are simply isotopies in this page, we can change a fixed monodromy
  $h _{\varphi }$ within its isotopy class to realize the required
  spinnings.  In this way we get a strongly admissible Heegaard
  diagram for $(Y, \t )$.
\end{proof}

\section{Invariants of Legendrian knots}
\label{s:def}

Suppose now that $L\subset (Y, \xi )$ is a given Legendrian knot, and
consider an open book decomposition $(B, \varphi )$ compatible with
$\xi$ containing $L$ on a page. To define our Legendrian
knot invariants we need to analyze the dependence from the choice of an appropriate basis
and from the open book decomposition as in Section~\ref{s:prelim}, but now in the presence of the Legendrian
knot.

\subsection*{Legendrian knots and bases}
Suppose that $S $ is a surface with $\partial S \neq \emptyset$ and
$\{ a_1, \ldots , a_n \}$ is a basis in $S$. If (after possibly
reordering the $a_i$'s) the two arcs $a_1$ and $a_2$ have adjacent endpoints on
some component of $\partial S$, that is, there is an arc $\tau \subset
\partial S$ with endpoints on $a_1$ and $a_2$ and otherwise disjoint
from all $a_i$'s, then define $a_1+a_2$ as the isotopy class (rel
endpoints) of the union $a_1\cup \tau \cup a_2$. The modification
\[
\{ a_1 , a_2, \ldots , a_n \} \mapsto \{ a_1+a_2, a_2, \ldots ,a_n\}
\]
is called an {\em arc slide}, cf.~\cite{HKM}.  Suppose that $L\subset
S$ is a homologically essential simple closed curve. The basis $\{
a_1, \ldots , a_n \}$ of $S $ is \emph{adapted to $L$} if $L\cap
a_i=\emptyset $ for $i\geq 2$ and $L$ intersects $a_1$ in a unique
transverse point.

\begin{lem}
  For any surface $S$ and homologically essential knot $L\subset S$
  there is an adapted basis.
\end{lem}

\begin{proof}
  The statement follows easily from the fact that $L$ represents a
  nontrivial class in $H_1(S , \partial S )$.  
\end{proof}

Suppose now that $\{ a_1, \ldots , a_n \}$ is an adapted basis for
$(S, L)$. An arc slide $\{ a_i, a_j
\}\mapsto \{ a_i+a_j , a_j\}$ is called \emph{admissible} if the arc
$a_i $ is not slid over the distinguished arc $a_1$.
The aim of this subsection is to prove the following 

\begin{prop}\label{p:cutsyst}
  If $\{ a_1, \ldots , a_n\} $ and $\{ A_1, \ldots , A_n\}$ are two
  adapted bases for $(S, L)$ then there is a sequence of admissible
  arc slides which trasform $\{ a_1, \ldots , a_n\} $ into $\{ A_1,
  \ldots , A_n\}$.
\end{prop}

\begin{proof}
As a first step in proving the statement on arc slides, we want to
show that, up to applying a sequence of admissible arc slides to the
$a_i$'s, we may assume $(a_1\cup \ldots \cup a_n )\cap (A_1\cup \ldots
\cup A_n)=\emptyset$.  We start with showing that $a_1\cap
A_1=\emptyset$ can be assumed. Suppose that $a_1\cap A_1\neq
\emptyset$; we will find arc slides reducing $\vert a_1\cap A_1\vert$.
To this end, consider the disk $D^2$ obtained by cutting $S $ along
the $a_i$'s.  Then, $A_1\cap D^2$ is a collection of arcs, and (at
least) one component intersects $a_1$. This component of $A_1$ divides
$D^2$ into two components $D_1$ and $D_2$, and one of them, say $D_1$,
contains $a_1^{-1}$ (or $a_1$). Sliding $a_1$ over all the $a_i$'s
contained in the boundary semicircle of $D_2$
(cf. Figure~\ref{f:slides}(a)) we reduce $\vert a_1\cap A_1 \vert$ by
one, so ultimately we can assume that $a_1\cap A_1=\emptyset$.
\begin{figure}[ht!]
\labellist
\small\hair 2pt
\pinlabel $a_1$ at 9 406
\pinlabel $A_1$ at 77 371
\pinlabel $D_1$ at 71 313
\pinlabel $D_2$ at 94 406
\pinlabel $a_1^{-1}$ [lt] at 134 294

\pinlabel $A_1$ at 352 343
\pinlabel $L$ at 359 399
\pinlabel $a_1$ at 286 408
\pinlabel $a_1^{-1}$ at 431 416
\pinlabel $a_2$ at 264 354

\pinlabel $a_1$ [lb] at 126 196
\pinlabel $a_1^{-1}$ [lt] at 130 72
\pinlabel $A_1$ at 40 121
\pinlabel $L$ at 90 176

\pinlabel $A_2$ at 319 121
\pinlabel $L$ at 360 179
\pinlabel $a_2$ [lt] at 391 59
\pinlabel $a_2^{-1}$ [b] at 367 211
\pinlabel $a_1$ [rb] at 300 180
\pinlabel $a_1^{-1}$ [lb] at 422 185
\pinlabel $(a)$ at 79 254
\pinlabel $(b)$ at 359 254
\pinlabel $(c)$ at 79 35
\pinlabel $(d)$ at 363 35
\endlabellist
\centering
\includegraphics[scale=0.5]{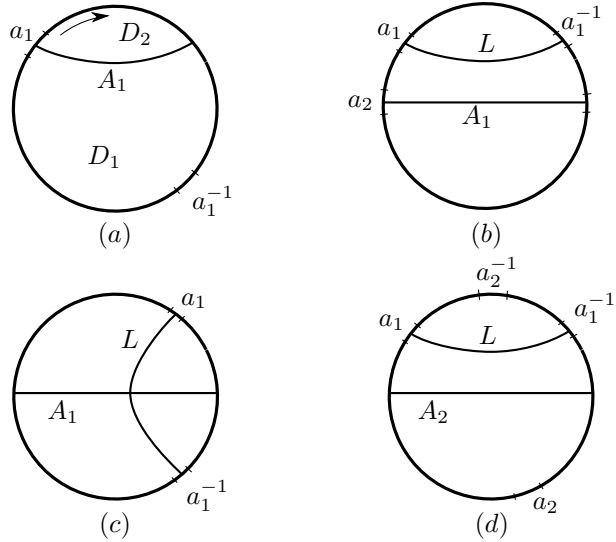}
\caption{{\bf Arc slides.}}
\label{f:slides}
\end{figure}
Next we apply further arc slides to achieve $a_i \cap A_1=\emptyset $
($i\geq 2$). For this, let us assume that $a_2$ is the first arc
intersecting $A_1$ when traversing $A_1$ starting from $\partial S$.
(Since $A_1$ intersects $L$ exactly once, after possibly starting at
the other end of $A_1$ we can assume that it first meets $a_2$ and
then $L$.)  As before, we can find a segment of $A_1$ intersecting
$a_2$ and dividing $D^2=S - \cup a_i$ into two components, one of
which contains $a_1, a_1^{-1}$ (since these are connected by $L$, and
the segment we chose is disjoint from $L$), cf.
Figure~\ref{f:slides}(b). If $a_2^{-1}$ is in the same semicircle as
$a_1$ (and so $a_1^{-1}$) then we can slide $a_2$ along the other
semicircle to eliminate one intersection point from $a_2\cap A_1$. If
$a_2^{-1}$ is on the other semicircle, then we cannot proceed in such
a simple way (since $a_2$ is not allowed to be slid over neither $a_1$
nor $a_2^{-1}$).  Now consider the continuation of $A_1$, which comes
out from $a_2^{-1}$ and stays in the same component. It might go back
to $a_2$, and repeat this spiraling some more times, but eventually it
will go to another part of the the boundary of $D^2$, producing an
arc, which starts from $a_2$ (or $a_2^{-1}$) and divides $D^2$ in such
a way that $a_1$ and $a_2^{-1}$ (or $a_2$) are on the same side of it.
Then, $a_2$ can be slid across the opposite side so as to reduce the
intersection number in question.

After these slides we can assume that $A_1\cap (\cup _{i=1}^n
a_i)=\emptyset$. This however allows us to slide $a_1$ until it becomes
isotopic to $A_1$, cf. Figure~\ref{f:slides}(c).  Consider now $A_2$. If $A_2$
intersects some $a_i$, then the segment of $A_2$ connecting $\partial S$ to
the first such intersection (with, say, $a_2$) divides $D^2$ into two
components, and one of them contains both $a_1$ and $a_1^{-1}$ (since $A_2$ is
disjoint from $L$). If $a_2^{-1}$ is on the same semicircle as $a_1$, then
sliding over the other semicircle reduces the number of intersections. The
other possibility can be handled exactly as before.

Finally, we get to the position when $a_1$ is isotopic to $A_1$ and
$(\cup _{i=1}^n a_i)\cap (\cup _{i=1}^n A_i)=\emptyset$. Now we argue
as follows: consider $A_2$ and choose $a_i$ such that $a_i$ and
$a_i^{-1}$ are in different components of $D^2-A_2$. Such $a_i$ exists
because $A_2$ is nonseparating. Suppose without loss of generality
that $i=2$. Then on the side of $A_2$ not containing $a_1$ (and
$a_1^{-1}$) we can slide $a_2$ until it becomes isotopic to $A_2$, see
Figure~\ref{f:slides}(d). Repeating this procedure for each $A_i$ the
proof is complete.
\end{proof}

\subsection*{Invariants of Legendrian knots}

Consider now a null--homologous Legendrian knot $L\subset (Y, \xi)$
and fix an open book decomposition $(B, \varphi)$ compatible with
$\xi$ and containing $L$ on the page $S_{+1}:={\overline {\varphi
    ^{-1}(1)}}$.  Pick a basis $A=\{ a_1, \ldots , a_n\}\subset
S_{+1}$ such that $a_1\cap L$ is a unique point and $a_i\cap
L=\emptyset$ ($i\geq 2$).  Under the above conditions we will say that
the triple $(B,\varphi,A)$ is \emph{compatible} with the triple
$(Y,\xi,L)$.

Place the basepoint $w$ as before. Putting the other basepoint $z$
between the curves $a_1$ and $b_1$ we recover a knot in $S_{+1}$
smoothly isotopic to $L$: connect $z$ and $w$ in the complement of
$a_i$, and then $w$ and $z$ in the complement of $b_i$ within the page
$S_{+1}$.  This procedure (hence the ordered pair $(w,z)$) equips $L$
with an orientation. Moreover, if the point $z$ is moved from one
domain between $a_1$ and $b_1$ to the other, the orientation induced
on $L$ gets reversed.  Thus, if $L$ is already oriented then there is
only one compatible choice of position for $z$.  Notice that $z$ and
$w$ chosen as above determine a knot in $S_{+1}$, unique up to
isotopy in $S_{+1}$. In turn, by Theorem~\ref{t:legendrianuniqueness}
the open book decomposition together with such a knot uniquely
determines a Legendrian knot (up to Legendrian isotopy) in the
corresponding contact structure. In short, $(B, \varphi, S_{+1}, A, z,
w )$ determines the triple $(Y, \xi , L)$.

Recall that when defining the chain complex $\CFa$ containing the
contact invariant, we reverse the roles of the $\alpha$-- and
$\beta$--curves, resulting in a Heegaard diagram for $-Y$ rather than
for $Y$. For the same reason, we do the switch between the $\alpha$--
and $\beta$--curves here as well. According to our conventions, this
change would reverse the orientation of the knot $L$ as well; to keep
the fixed orientation on $L$, switch the position of the basepoints
$w$ and $z$. The two possible locations of $z$ and $w$ we use in the
definition of $\CFa(-Y,L)$ are illustrated in Figure~\ref{f:basepoint};
the orientation of $L$ specifies the location of $w$.
\begin{figure}[ht!]
\labellist
\small\hair 2pt
\pinlabel $b_1$ [b] at 76 69
\pinlabel $a_1$ [b] at 100 69
\pinlabel $b_1$ [b] at 298 69
\pinlabel $a_1$ [b] at 322 69
\pinlabel $L$ at 36 47
\pinlabel $w$ at 87 38
\pinlabel $w$ at 310 56
\pinlabel $L$ at 358 55
\pinlabel $z$ at 147 32
\pinlabel $z$ at 371 32
\pinlabel {or} at 198 51
\endlabellist
\centering
\includegraphics[scale=0.85]{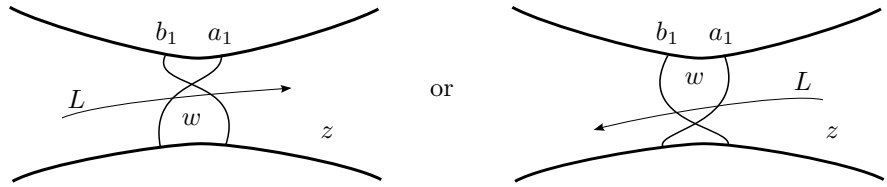}
\caption{{\bf There are two regions between $a_1$ and $b_1$ in $S_{+1}$;
  the placement of $w$ is determined by the orientation of $L$.}}
\label{f:basepoint}
\end{figure}
With $\x =\x (B, \varphi )=( a_i \cap b_i) $ as before, it is easy to
see that there are no nonnegative Whitney disks $\phi\in\pi_2(\x,\y)$
with $n_z(\phi)=0$ for any $\y$, hence (as in \cite[Section~3]{HKM})
the intersection point $\x$ can be viewed as a cycle in both
$\CFKm(-Y,L,\t)$ and $\CFKa(-Y,L,\t)$ for some $\SpinC$ structure $\t$
on $Y$. As observed in~\cite[Subsection~2.3]{OSzknot}, the $\SpinC$
structure $\t$ is determined by the point $z$ only, being equal to
$s_z(\x)$ (see~\cite{OSzknot} for notation). This shows that the
contact invariant $c(Y,\xi)$ lives in the summand of $\HFa(-Y)$
corresponding to $\t$. On the other hand, it is known~\cite{OSzcont}
that $c(Y,\xi)\in\HFa(-Y,\t_\xi)$, therefore we conclude $\t=\t_\xi$.
Since $L$ is null--homologous, one needs to make sure that the
Heegaard diagram is strongly admissible for $(Y,\t_\xi)$; this fact
follows from Proposition~\ref{p:strongadm}. Next we will address
invariance properties of the knot Floer homology class represented by
$\x (B, \varphi )$.

\begin{prop}\label{p:Lambda}
  Let $L\subset (Y,\xi)$ be a Legendrian knot. Let $(B,\varphi,A, w,z
  )$ and $(B',\varphi',A',w',z')$ be two open books compatible with
  $(Y,\xi,L)$ and endowed with bases and basepoints adapted to $L$.
  Then, there are isomorphisms of $\Field[U]$--modules
\[
\Phi^-\colon \HFKm(-Y,L,\t_\xi) \to \HFKm(-Y,L,\t_\xi)
\]
and 
\[
\widehat\Phi\colon \HFKa(-Y,L,\t_\xi) \to \HFKa(-Y,L,\t_\xi)
\]
such that 
\[
\Phi^-([\x(B,\varphi,A)])=[\x(B',\varphi',A')]
\] 
and 
\[
\widehat\Phi([\x(B,\varphi,A)])=[\x(B',\varphi',A')].
\]
\end{prop}

Recall from Section~\ref{s:intro} that, if $M$ and $N$ are
$\Field[U]$--modules, we consider $(M,m)$ and $(N,n)$ to be
equivalent, and we write $[M,m]=[N,n]$, if there is an isomorphism of
$\Field[U]$--modules $f\co M\to N$ such that $f(m)=n$.

In view of Proposition~\ref{p:Lambda}, we introduce the following 

\begin{defn}\label{d:Lambda}
  Let $L\subset (Y,\xi)$ be an oriented Legendrian knot,
  $(B,\varphi,A,w,z)$ an open book decomposition compatible with
  $(Y,\xi,L)$ with an adapted basis and basepoints.  Then, we define
\[
\La(L):=\left[\HFKm(-Y, L,\t_\xi), [\x(B,\varphi,A)]\right].
\]
Similarly, we define 
\[
\widehat\La(L):=\left[\HFKa(-Y, L,\t_\xi), [\x(B,\varphi,A)]\right].
\]
\end{defn}

\begin{proof}[Proof of Proposition~\ref{p:Lambda}]
Suppose for the moment that $(B',\varphi')=(B,\varphi)$. Then, 
Proposition~\ref{p:cutsyst} provides a sequence of arc slides
transforming $A'$ to the chosen $A$. As it is
explained in \cite{HKM}, arc slides induce handle slides on the
associated Heegaard diagrams, and the invariance of the Floer
homology element under these handle slides is verified in
\cite[Lemma~3.4]{HKM}. Notice that since we do not slide over the
arc intersecting the Legendrian knot $L$ (and hence over the second
basepoint defined by this arc), actually all the handle slides
induce $\Field[U]$--module isomorphisms on the knot Floer groups, 
cf.~\cite{OSzknot}. This argument proves the statement in the special case 
$(B',\varphi')=(B,\varphi)$. 

We now consider the special case when $(B',\varphi')$ is an
$L$--elementary stabilization of $(B,\varphi)$. Suppose that the Dehn
twist of the stabilization is along $\gamma \subset S '$, with $\gamma
_1\subset S$. Since we work with an $L$--elementary stabilization,
$\vert \gamma _1 \cap L \vert \leq 1$. If $\gamma _1$ is separating
then $\gamma _1 \cap L =\emptyset$ and we choose the bases of $S$ and
$S'$ as in the proof of Theorem~\ref{t:contel}. If $\gamma _1$ is
nonseparating and $\gamma _1\cap L=\emptyset$ then choose $a_2=\gamma
_1$ and extend it to an appropriate basis. For $\gamma _1\cap L =\{
{\mbox{pt.}} \}$ we take $a_1=\gamma _1$ and extend it further. Denote
by $T$ and $T'$ the resulting bases of $S$ and $S'$ respectively. 
The proof of Theorem~\ref{t:contel} now applies verbatim to show the
existence of automorphisms of $\HFKm(-Y,L,\t_\xi)$ and $\HFKa(-Y,
L,\xi)$ mapping $[\x(B,\varphi,T)]$ to $[\x(B',\varphi',T')]$. On the
other hand, by the first part of the proof we know that there are
other automorphisms sending $[\x(B,\varphi,A)]$ to $[\x(B,\varphi,T)]$
and $[\x(B',\varphi',T')]$ to $[\x(B',\varphi',A')]$. This proves the
statement when $(B,\varphi')$ is an $L$--elementary stabilization of
$(B,\varphi)$.

In the general case, since $(B, \varphi)$ and $(B', \varphi ')$ are two
open books compatible with $(Y,\xi, L)$, by Proposition~\ref{p:elementary} 
we know that there is a sequence of $L$--elementary stabilizations which turns 
each of them into the same stabilization $(B'', \varphi'')$. Thus, 
applying the previous special case the required number of times,  
the proof is complete. 
\end{proof}

\begin{rem}\label{r:U=0}
  Let $L\subset (Y,\xi)$ be a Legendrian knot and $(B,\varphi,A, w,
  z)$ a compatible open book with adapted basis and basepoints. Then,
  it follows immediately from the definitions that the map $f$ from
  $\HFKm(-Y,L,\s)$ to $\HFKa(-Y,L,\s)$ induced by setting $U=0$ sends
  the class of $\x(B,\varphi,A)$ in the first group to the class of
  $\x(B,\varphi,A)$ in the second group.  Moreover, the chain map
  inducing $f$ can be viewed as the canonical map from the complex
  $\CFKm (-Y,L,\s)$ onto its quotient complex $\CFKa(-Y,L,\s)$.  As
  such, it is natural with respect to the transformations of the two
  complexes induced by the isotopies, stabilizations and arc slides
  used in the proof of Proposition~\ref{p:Lambda}. Thus, it makes
  sense to write $f(\La(L))=\widehat\La(L)$. Therefore
  $\widehat\La(L)\neq 0$ readily implies $\La(L)\neq 0$, although the
  converse does not necessarily hold: a nonvanishing invariant
  $\La(L)$ determined by a class which is in the image of the $U$--map
  gives rise to vanishing $\widehat\La(L)$. As we will see, such
  examples do exist.
\end{rem}

\begin{cor}\label{c:l-invariance}
Let $L_1, L_2\subset (Y,\xi)$ be oriented Legendrian knots. Suppose 
that there exists an isotopy of oriented Legendrian knots from $L_1$ 
to $L_2$. Then, $\La(L_1) = \La(L_2)$ and 
$\widehat\La(L_1) = \widehat\La(L_2)$.
\end{cor}

\begin{proof}
  Let $(B,\varphi,A)$ be an open book compatible with $(Y,\xi,L_1)$
  with an adapted basis, and let $f_1$ be the time--1 map of the
  isotopy.  Then, the triple $(f_1(B), \varphi \circ f_1 ^{-1},
  f_1(A))$ is compatible with and adapted to $(Y,\xi,L_2)$. The
  induced map on the chain complexes maps $\x(B,\varphi , A)$ to
  $\x(f_1(B),\varphi \circ f_1 ^{-1},f_1(A))$, verifying the last
  statement.
\end{proof}

\begin{rem} {\rm In fact, we only used the fact that $f_1\colon (Y,
    \xi )\to (Y, \xi )$ is a contactomorphism mapping $L_1$ into $L_2$
    (respecting their orientation). In conclusion, Legendrian knots
    admitting such an identification have the same Legendrian
    invariants. The existence of $f_1$ with these properties and the
    isotopy of the two knots is equivalent in the standard contact
    3--sphere, but the two conditions are different in general.}
\end{rem}

\begin{proof}[Proof of Theorem~\ref{t:main}]
The theorem follows immediately from Proposition~\ref{p:Lambda} and
Corollary~\ref{c:l-invariance}.
\end{proof}

\section{An example}\label{s:unknot}

Suppose that $L\subset (S^3, \xi _{st})$ is the Legendrian unknot with
Thurston--Bennequin invariant $-1$ in the standard tight contact
3--sphere. It is easy to see that the positive Hopf link defines an
open book on $S^3$ which is compatible with $\xi_{st}$ and it contains
$L$ on a page. A basis in this case consists of a single arc cutting
the annulus.  The corresponding genus--1 Heegaard diagram has now a
single intersection point, which gives the generator of $\HFKm(-S^3,
L)$ (the 3--sphere $S^3$ has a unique $\SpinC$ structure, so we omit
it from the notation). The two possible choices $w_1$ and $w_2$ for
the position of $w$, corresponding to the two choices of an
orientation for $L$, give the same class defining $\La(L)$, because in
this case $w_1$ and $w_2$ are in the same domain,
cf. Figure~\ref{f:unknot}$(i)$. Let $L'$ denote the stabilization of
$L$. The knot $L'$ then can be put on the page of the once stabilized
open book, depicted (together with the monodromies) by
Figure~\ref{f:unknot}$(ii)$, where the unknot is represented by the
curve with $L'$ next to it, and the thin circles represent curves
along which Dehn twists are to be performed to get the monodromy map.
\begin{figure}[ht!]
\labellist
\small\hair 2pt
\pinlabel $L$ [lb] at 12 85
\pinlabel $L'$ [lb] at 244 113
\pinlabel $b$ [rt] at  96 106
\pinlabel $a$ [lt] at 63 106
\pinlabel $w_1$ [B] at 79 116
\pinlabel $w_2$ [B] at 79 146
\pinlabel $z$ [lb] at 126 113
\pinlabel $b_1$ [t] at 326 88
\pinlabel $a_1$ [t] at 300 88
\pinlabel $z$ at 358 110
\pinlabel $b_2$ [t] at 329 165
\pinlabel $a_2$ [t] at 295 164
\pinlabel $w_1$ [B] at 311 188
\pinlabel $w_2$ [B] at 309 212
\pinlabel $(i)$ at 79 0
\pinlabel $(ii)$ at 311 0
\endlabellist
\centering
\includegraphics[scale=0.85]{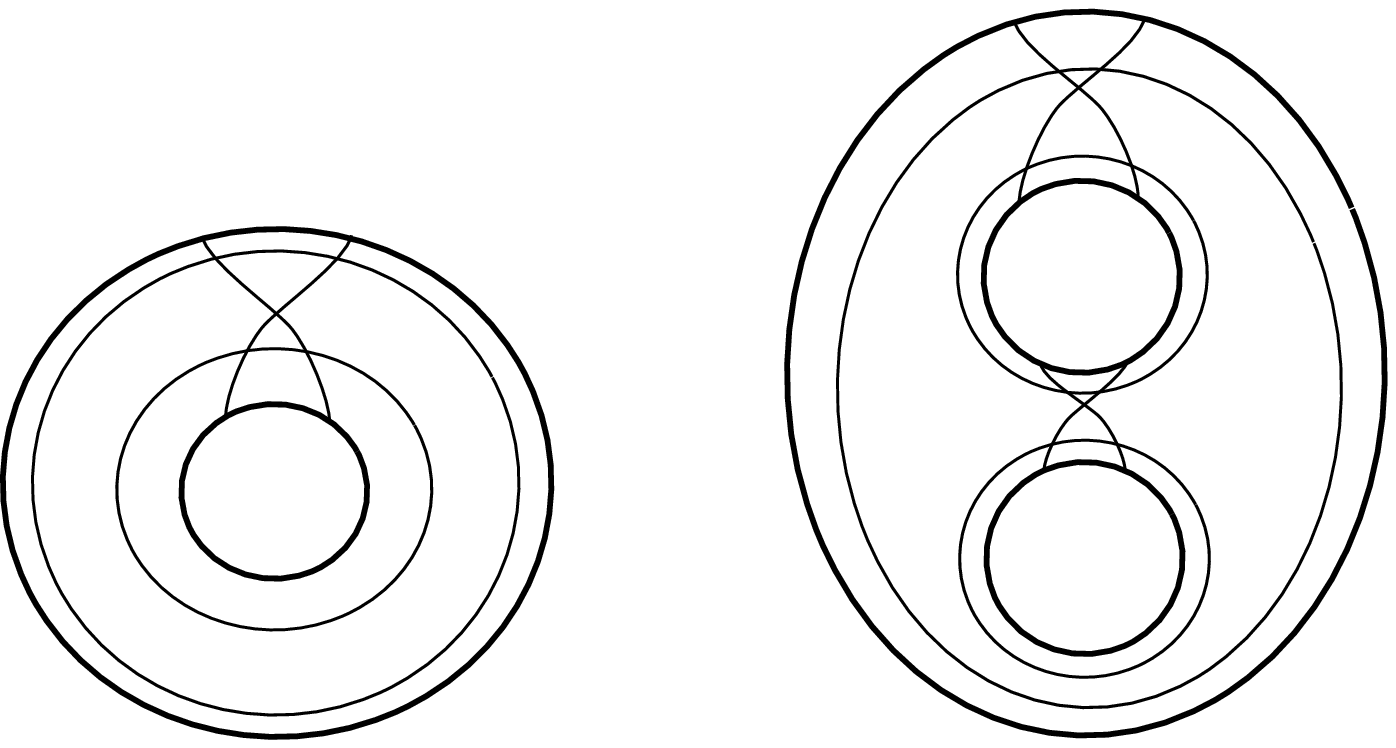}
\caption{{\bf The page of the open book and the basis for the
stabilized unknot.}}
\label{f:unknot}
\end{figure}
The page on the left represents the open book decomposition for the
Hopf band, while the page on the right is its stabilization. The
corresponding Heegaard diagram (with the use of the adapted basis of
Figure~\ref{f:unknot}$(ii)$) is shown in Figure~\ref{f:hdunknot}.
\begin{figure}[ht!]
\labellist
\small\hair 2pt
\pinlabel $P$ [t] at 98 104
\pinlabel $w_1$ [B] at 66 127
\pinlabel $V$ [ll] at 225 139
\pinlabel $T$ [ll] at 261 142
\pinlabel $X_1$ [tl] at 334 99
\pinlabel $X_2$ [B] at 353 178
\pinlabel $z$ at 262 199
\pinlabel $w_2$ at 140 155
\endlabellist
\centering
\includegraphics[scale=0.7]{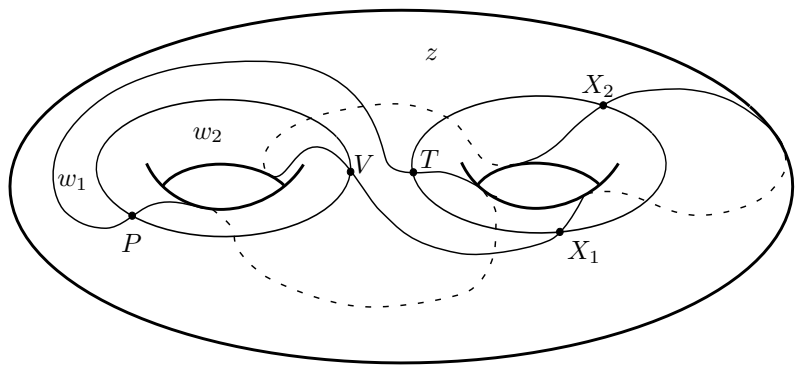}
\caption{{\bf Heegaard diagram for the unknot.}}
\label{f:hdunknot}
\end{figure}
We record both possible choices of $w$ by putting a $w_1$ and a $w_2$
on the diagram --- the two choices correspond to the two possible
orientations of $L$. Incidentally, these two choices also correspond
to the two possible stabilizations with respect to a given
orientation, since positive stabilization for an orientation is
exactly the negative stabilization for the reversed orientation.  It
still needs to be determined whether $w_2$ gives a positive or
negative stabilization.
\begin{lem}
  \label{lem:StabilizationPicture}
  The oriented knot determined by the pair $z$ and $w_2$ represents
  the negative stabilization of the oriented unknot (i.e.~the stabilization
  with $\rot=-1$).
\end{lem}

\begin{proof}
In order to compute the rotation number of the stabilization given
by $w_2$, first we construct a Seifert surface for $L'$. To this end,
consider the loop $A_1$ given by the upper part of $L'$ together with the
dashed line $C_1$ of Figure~\ref{f:last}.
\begin{figure}[ht!]
\labellist
\small\hair 2pt
\pinlabel $L'$ at 33 149
\pinlabel $R$ at 30 104 
\pinlabel $C_1$ at 138 133
\pinlabel $C_2$ at 138 75
\endlabellist
\centering
\includegraphics[scale=0.7]{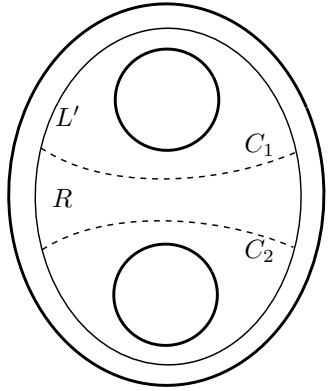}
\caption{{\bf Computation of the rotation of $L'$.}}
\label{f:last}
\end{figure}
This loop bounds a disk $D_1$ in the 3--manifold given by the open
book decomposition, and the tangent vector field along $A_1$ obviously
extends as a nonzero section of $\xi $ to $D_1$, since $D_1$ can be
regarded as an appropriate Seifert surface for the unknot $L$ before
stabilization (cf. Figure~\ref{f:unknot}). Define $A_2$ and $D_2$
similarly, now using the lower part of $L'$. A Seifert surface for
$L'$ then can be given by the union of $D_1,D_2$ and the region $R$ of
Figure~\ref{f:last}. Extending the tangent vector field along $L'$ to
a section of $\xi$ over $C_1,C_2$ first, the above observation shows
that the rotation number of $L'$ is the same as the obstruction to
extending the above vector field to $R$ as a section of $\xi$. Notice
that along $R$ we have that $\xi = TR$.  The region $R$ (with the
given vector field on its boundary) can be embedded into the disk with
the tangent vector field along its boundary, hence a simple Euler
characteristic computation shows that the obstruction we need to
determine is equal to $-1$, concluding the proof.
\end{proof}

Notice that the monodromy
is pictured on the page $S_{+1}$ (which also contains the knot),
but its effect is taken into account on $S_{-1}$, which comes in the
Heegaard surface with its orientation reversed. Therefore when
determining the $\alpha$-- and $\beta$--curves in the Heegaard
diagrams, right--handed Dehn twists of the monodromy induce
left--handed Dehn twists on the diagram and vice versa.

The three generators of the chain complex corresponding to the
Heegaard diagram of Figure~\ref{f:hdunknot} in the second symmetric
product of the genus--2 surface are the pairs $PX_1, PX_2$ and $TV$.
It is easy to see that there are holomorphic disks connecting $TV$ and
$PX_2$ (passing through the basepoint $w_1$) and $TV$ and $PX_1$
(passing through $w_2$); and these are the only two possible
holomorphic disks not containing $z$.

When we use $w_1$ as our second basepoint, 
the two disks out of $TV$ show that the class represented by $PX_1$,
and the class of $PX_2$ when muliplied by $U$ are homologous:
\[
\dm (TV)=PX_1 + U\cdot PX_2.
\]
In conclusion, in this case $[PX_2]$ generates $\HFKm(-S^3, L')$ and
the invariant $[PX_1]$ is determined by the class of $U$--times the
generator.

When using $w_2$ as the second basepoint, we see that the class
$[PX_2]$ will be equal to $U\cdot [PX_1]$, hence in this case $[PX_1]$
is the generator.  Since $PX_1$ represents the Legendrian invariant,
we conclude that in this case the equivalence class modulo automorphisms of 
the generator of $\HFKm(-S^3, L')$ (over $\Field[U]$) is equal to the 
Legendrian invariant of $L'$. In summary, we get
\begin{cor}
Suppose that $L\subset (S^3, \xi )$ is an oriented Legendrian unknot
with $\tb (L)=-1$. Then $\LegInv (L)$ is represented by the generator
of the $\Field [U]$--module $\HFKm (L) = \Field [U]$. If $L^-$ is the
negative stabilization of $L$ then $\LegInv (L^-)=\LegInv (L)$, while
for the positive stabilization $L^+$ we have $\LegInv (L^+)=U\cdot
\LegInv (L)$. \qed
\end{cor}

\section{Basic properties of the invariants}
\label{s:basic}

\subsection*{Nonvanishing and vanishing results}

The invariant $\La$ admits a nonvanishing property provided the
contact invariant $c(Y,\xi)$ of the ambient 3--manifold is nonzero
(which holds, for example, when the ambient contact structure is
strongly fillable). When $c(Y, \xi )=0$ (e.g., if $(Y, \xi )$ is
overtwisted) then $\LegInv (L)$ is a $U$--torsion element.

\begin{proof}[Proof of Theorem~\ref{t:nonvan}]
Consider the natural chain map 
\[
f\colon  \CFKm (-Y, L) \to \CFa(-Y)
\]
given by setting $U=1$, cf.~Lemma~\ref{l:kernel}. Let $(B,\varphi,A,
w, z)$ be an open book compatible with $(Y,\xi,L)$ with an adapted
basis and basepoints. Since the map $F$ induced by $f$ on the
homologies sends $[\x(B,\varphi,A)]$ to $c(Y, \xi )$, the nonvanishing
of $\La(L)$ when $c(Y, \xi )\neq 0$ obviously follows.  If the above
map sends $[\x(B,\varphi,A)]$ to zero, then by Lemma~\ref{l:kernel}
(and the fact that $[\x(B,\varphi,A)]$ is homogeneous) we get that
$U^d \cdot \La(L)=0$ for some $d\geq 0$, verifying
Theorem~\ref{t:nonvan}.
\end{proof}

A vanishing theorem can be proved for a \emph{loose} knot, that is,
for a Legendrian knot in a contact 3--manifolds with overtwisted
complement.  Before this result we need a preparatory lemma from
contact topology:

\begin{lem}\label{l:loose}
  Suppose that $L\subset (Y, \xi )$ is a Legendrian knot such that
  $(Y,\xi)$ contains an overtwisted disk in the complement of $L$.
  Then, the complement $(Y-\nu L , \xi \vert _{Y- \nu L} )$ admits a
  connected sum decomposition $(Y-\nu L, \xi _1)\# (S^3 , \xi _2)$
  with the property that $\xi _1$ coincides with $\xi \vert _{Y-
    \nu L}$ near $\partial (Y- \nu L)$ and $\xi _2$ is overtwisted.
\end{lem}

\begin{proof}
  Let us fix an overtwisted disk $D$ disjoint from the knot $L$ and
  consider a neighbourhood $V$ (diffeomorphic to $D^3$) of $D$, with
  the property that $V$ is still disjoint from $L$. By the
  classification of overtwisted contact structures on $D^3$ with a
  fixed characteristic foliation on the boundary
  \cite[Theorem~3.1.1]{Eliashb}, we can take $\xi _1$ on $Y-\nu L$ and
  $\xi _2$ on $S^3$ such that $\xi _1\# \xi _2= \xi\vert _{Y-
    \nu L}$ and $\xi _1$ is equal to $\xi$ near $\partial (Y-\nu L)$, while $\xi _2$ is
  overtwisted.  The statement of the lemma then follows at once.
\end{proof}

\begin{proof}[Proof of Theorem~\ref{t:loose}]
  Let us fix a decomposition of $(Y, \xi )$ as before, that is, $(Y,
  \xi )=(Y, \xi _1) \# (S^3, \xi _2)$ with the properties that
  $L\subset (Y, \xi _1)$ and $\xi _2$ is overtwisted on $S^3$.
  Consider open book decompositions compatible with $(Y, \xi _1)$ and
  $(S^3 , \xi _2)$. Assume furthermore that the first open book has a
  basis adapted to $L$, while the second open book has a basis
  containing an arc which is displaced to the left by the monodromy.
  (The existence of such a basis is shown in the proof of
  \cite[Lemma~3.2]{HKM}.)  The Murasugi sum of the two open books and
  the union of the two bases provides an open book decomposition for
  $(Y,\xi )$, adapted to the knot $L$, together with an arc disjoint
  from $L$ which is displaced to the left by the monodromy. Since the
  basepoint $w$ in the Heegaard diagram is in the strip determined by
  the arc intersecting the knot, the holomorphic disk appearing in the
  proof of \cite[Lemma~3.2]{HKM} avoids both basepoints and shows the
  vanishing of $\La(L)$.
\end{proof}

\subsection*{Transverse knots}
\label{s:trans}

Next we turn to the verification of the formula relating the
invariants of a negatively stabilized Legendrian knot to the
invariants of the original knot.  We will spell out the details for
$\La$ only. Then we will discuss the implication of the stabilization
result regarding invariants of transverse knots. The effect of more
general stabilizations on the invariant will be addressed later using
slightly more complicated techniques.

\begin{prop}\label{p:stab} 
Suppose that $L$ is an oriented Legendrian knot and $L^-$
denotes the oriented negative stabilization of $L$. Then, 
$\La(L^-)=\La(L)$.
\end{prop}

\begin{proof}
  The proof relies on the choice of a convenient open book
  decomposition. To this end, fix an open book decomposition
  $(B,\varphi, A)$ compatible with $(Y, \xi )$ and with a basis
  adapted to $L$. Place $w$ according to the given orientation. As
  shown in \cite{Etn}, after an appropriate Giroux stabilization
  the open book also accomodates the stabilization of $L$. The new
  open book with adapted basis $(B^-,\varphi^-,A^-)$ together with the
  new choice of $w$ (denoted by $w^-$) is illustrated by
  Figure~\ref{f:negstab}.
\begin{figure}[ht!]
\labellist
\small\hair 2pt
\pinlabel $a$ [t] at 84 22
\pinlabel $b$ [t] at 119 24
\pinlabel $a$ [t] at 328 22
\pinlabel $b$ [t] at 363 24
\pinlabel $w$ [b] at 346 48
\pinlabel $w$ at 101 56
\pinlabel $z$ at 150 47
\pinlabel $z$ at 404 49
\pinlabel $w^-$ [b] at 348 113
\pinlabel $a'$ [b] at 361 127
\pinlabel $b'$ [b] at 329 127
\pinlabel $c$ [tl] at 372 71
\endlabellist
\centering
\includegraphics[scale=0.8]{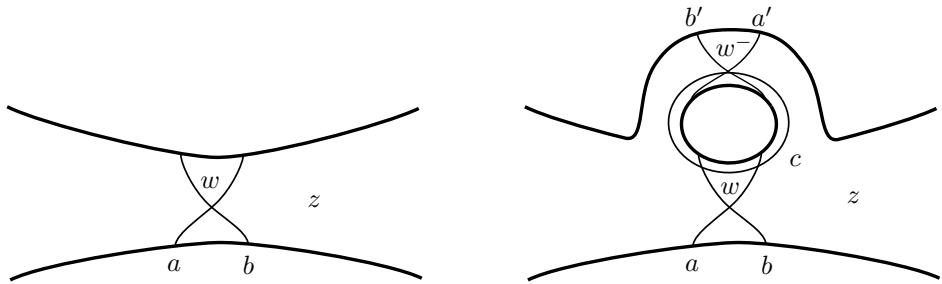}
\caption{{\bf Change of the open book after negative stabilization of $L$.}
The monodromy changes by a right--handed Dehn twist along the curve $c$.}
\label{f:negstab}
\end{figure}
As in Lemma~\ref{lem:StabilizationPicture}, we can easily see that this
choice provides the negative stabilization $L^-$. 
(Recall that the stabilization changed the monodromy of the open book
by multiplying it with the right--handed Dehn twist $D_{c}$.)  In the
new page the stabilization of $L$ is determined up to isotopy simply
by changing the basepoint from $w$ to $w^-$. (Notice that by placing
$w^-$ in the other domain in the strip between $a'$ and $b'$ the
orientation of the stabilized knot would be incorrect.)  Now the
corresponding portion of the Heegaard diagram has the form shown by
Figure~\ref{f:stabhd}.
\begin{figure}[ht!]
\labellist
\small\hair 2pt
\pinlabel $a$ [t] at 174 22
\pinlabel $b$ [t] at 217 24
\pinlabel $w$ at 194 72
\pinlabel $a'$ [tl] at 219 197
\pinlabel $w^-$ at 194 182
\pinlabel $b'$ [tr] at 173 197
\pinlabel $C_1$ [l] at 219 105 
\pinlabel $C_2$ [l] at 202 255
\pinlabel $\be$ [l] at 251 280
\pinlabel $z$ at 302 116
\endlabellist
\centering
\includegraphics[scale=0.7]{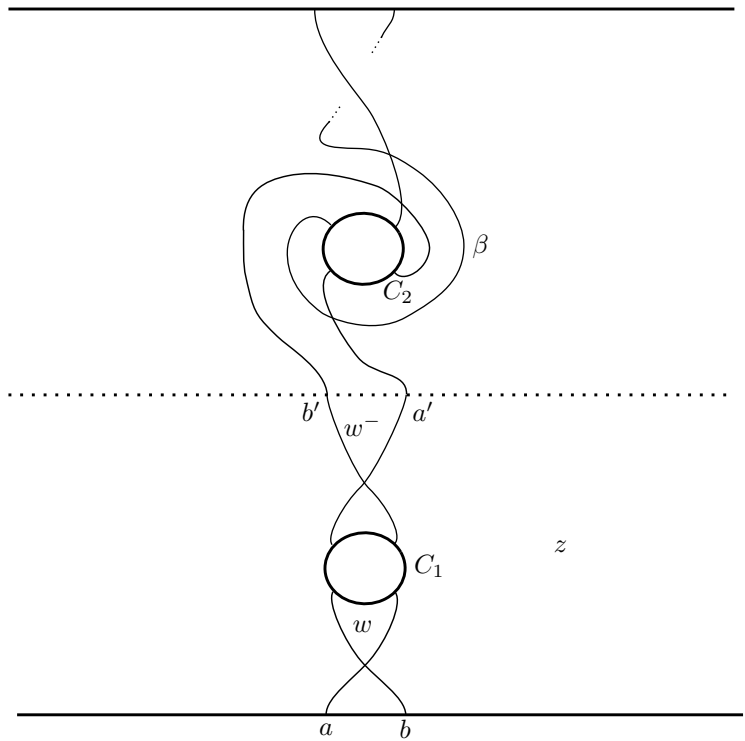}
\caption{{\bf Stabilization in the Heegaard diagram.} The top and
  bottom boundary components of the surface and the circles $C_1$ and
  $C_2$ should be thought of as identified via a reflection across the
  middle (dotted) line of the picture.  }
\label{f:stabhd}
\end{figure}
In the picture, the top and bottom boundary components of the surface
and the circles $C_1$ and $C_2$ should be thought of as identified via
a reflection across the middle (dotted) line of the picture.
Moreover, the curve $\be$ is only partially represented, due to the
action of the monodromy. As before, in the diagram $w$ gives rise to
$L$ while $w^-$ to $L^-$ (together with the common $z$).  It is
straightforward from the picture that $w$ and $w^-$ are in the same
domain, hence the statement follows.
\end{proof}

It follows from Proposition~\ref{p:stab} that the invariant $\La$ of a 
Legendrian approximation provides an invariant for a transverse knot.

\begin{proof}[Proof of Theorem~\ref{t:transverse}]
  Fix a transverse knot $T$ and consider a Legendrian approximation
  $L$ of $T$.  By \cite{EFM, EH}, up to negative stabilizations the
  Legendrian knot $L$ only depends on the transverse isotopy class of
  $T$. Therefore by Proposition~\ref{p:stab} the equivalence classes
  $\TransInv (T)=\LegInv (L)$ and $\TransInvha (T)=\LegInva (L)$ are
  invariants of the transverse isotopy class of the knot $T$, and
  hence the theorem follows.
\end{proof}

\section{Non--loose torus knots in $S^3$}
\label{s:example}

In this section we describe some examples where the invariants defined
in the paper are explicitly determined. Some interesting consequences
of these computations will be drawn in the next section.  For the sake
of simplicity, we will work with the invariant $\Laha$.

\subsection*{Positive Legendrian torus knots $T_{(2,2n+1)}$ in overtwisted 
contact $S^3$'s}

Let us consider the Legendrian knot $L(n)$ given by the surgery diagram
of Figure~\ref{f:torus-contact}. 
\begin{figure}[ht!]
\labellist
\small\hair 2pt
\pinlabel $L(n)$ [l] at 453 55
\pinlabel $+1$ [l] at 458 82
\pinlabel $+1$ [l] at 454 107
\pinlabel $-1$ [l] at 453 131
\pinlabel $-1$ [l] at 514 154
\pinlabel $-1$ [l] at 513 201
\pinlabel $n$ [r] at 55 177
\endlabellist
\centering
\includegraphics[scale=0.4]{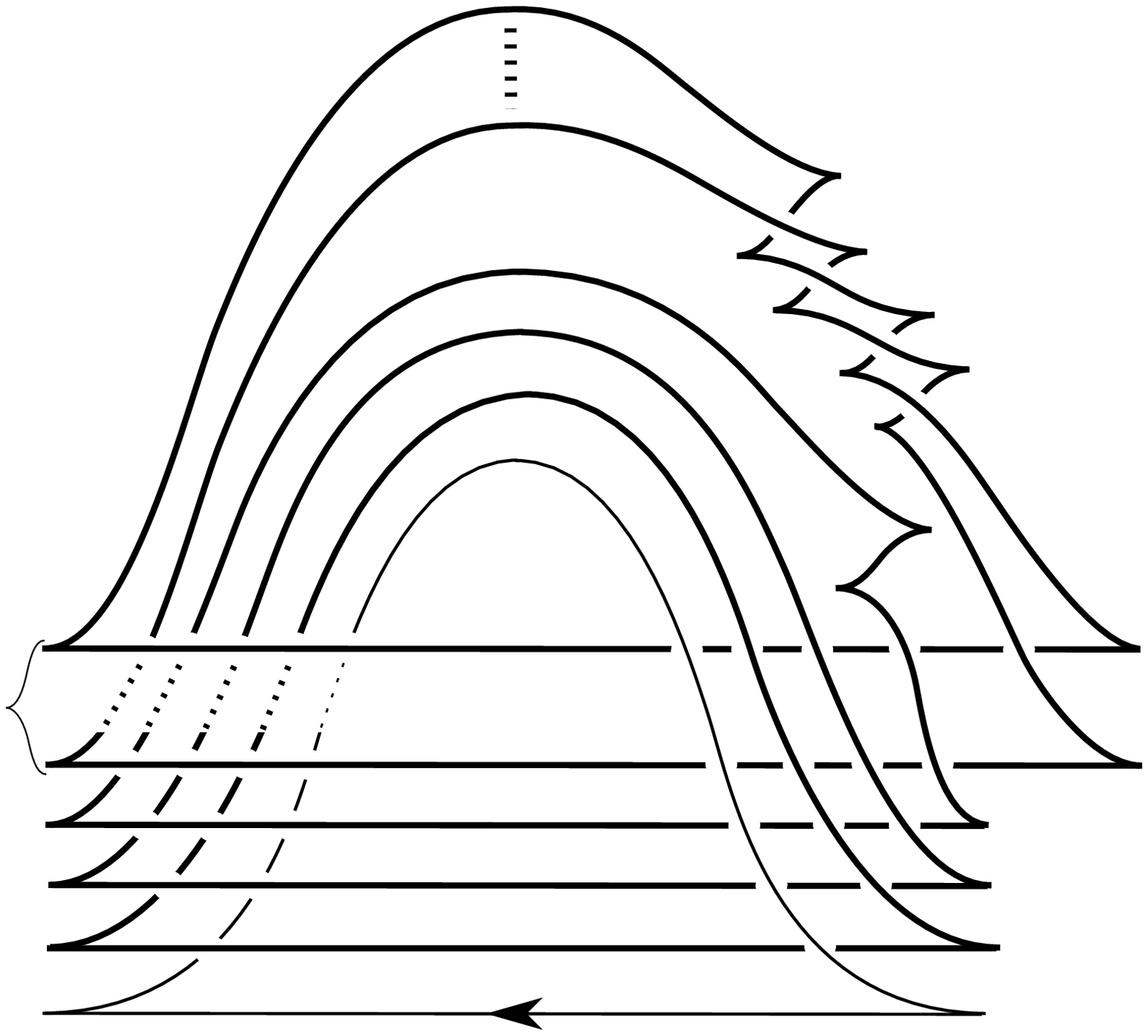}
\caption{{\bf Non--loose Legendrian torus knot $T_{(2,2n+1)}$ in $S^3$
    ($n\geq 1$).}}
\label{f:torus-contact}
\end{figure}
The meaning of the picture is that we perform contact 
$(\pm 1)$--surgeries along the given Legendrian knots, the result being a 
contact 3--manifold containing the unframed knot $L(n)$ as a Legendrian knot. 
(For contact $(\pm 1)$--surgeries and surgery presentations see \cite{DG, OS}.)

\begin{lem}\label{l:surgd3}
The contact structure $\xi _n$ defined by the surgery diagram of
Figure~\ref{f:torus-contact} is the overtwisted contact structure on $S^3$
with Hopf invariant $d_3(\xi _n)=1-2n$. The Legendrian knot $L(n)$ is
smoothly isotopic to the torus knot $T_{2,2n+1}$ and is non--loose.
\end{lem}

\begin{proof}
  Figure~\ref{f:torus-smooth} gives a smooth surgery diagram
  corresponding to Figure~\ref{f:torus-contact}.
\begin{figure}[ht!]
\labellist
\small\hair 2pt
\pinlabel $-2$ at 28 52
\pinlabel $-1$ at 139 53
\pinlabel $-3$ [l] at 75 55
\pinlabel $-4$ [lt] at 51 80
\pinlabel $-4$ [rb] at 35 96
\pinlabel $0$ [r] at 203 53
\pinlabel $0$ [r] at 232 53
\pinlabel $L(n)$ [l] at 266 56
\pinlabel $(n)$ at 37 81
\endlabellist
\centering
\includegraphics[scale=0.8]{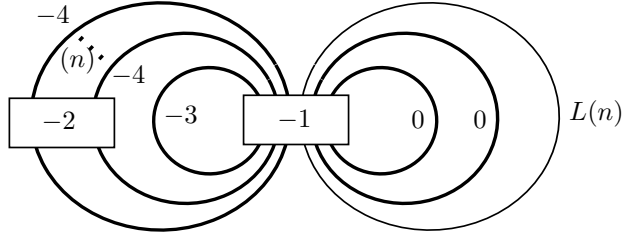}
\caption{{\bf A smooth version of Figure~\ref{f:torus-contact}.}}
\label{f:torus-smooth}
\end{figure}
The knot type of $L(n)$ and the underlying 3--manifold can be easily
identified.  The Kirby calculus moves of Figure~\ref{f:torus-moves}
show that Figure~\ref{f:torus-smooth} is equivalent to the left--hand
side of Figure~\ref{f:torus-final}.
\begin{figure}[ht!]
\labellist
\small\hair 2pt
\pinlabel $-1$ at 97 214
\pinlabel $-2$ at 20 213
\tiny
\pinlabel $-4$ at 44 233
\pinlabel $-4$ at 22 252
\pinlabel $-3$ at 58 214
\pinlabel $0$ at 139 214
\pinlabel $0$ at 159 214
\pinlabel $(n)$ at 27 234
\pinlabel $L(n)$ at 190 239
\pinlabel $+1$ at 397 178
\pinlabel $L(n)$ at 434 225
\pinlabel $(n)$ at 263 236
\pinlabel $+1$ at 392 243
\pinlabel $+1$ at 379 229
\pinlabel $-1$ at 284 235
\pinlabel $-1$ at 261 255
\pinlabel $+1$ at 236 232
\pinlabel $+1$ at 234 216
\pinlabel $-2$ at 308 231
\pinlabel $-1$ at 367 27
\pinlabel $L(n)$ at 428 83
\pinlabel $+1$ at 235 85
\pinlabel $+1$ at 231 66
\pinlabel $(n)$ at 263 89
\pinlabel $-1$ at 288 91
\pinlabel $-1$ at 267 111
\pinlabel $-2$ at 371 88
\pinlabel $-2$ at 39 71
\pinlabel $-3$ at 82 68
\pinlabel $-2$ at 83 84
\pinlabel $-2$ at 81 150
\pinlabel $n-1$ at 152 114
\pinlabel $L(n)$ at 168 70
\pinlabel $-1$ at 97 14
\endlabellist
\centering
\includegraphics[scale=0.8]{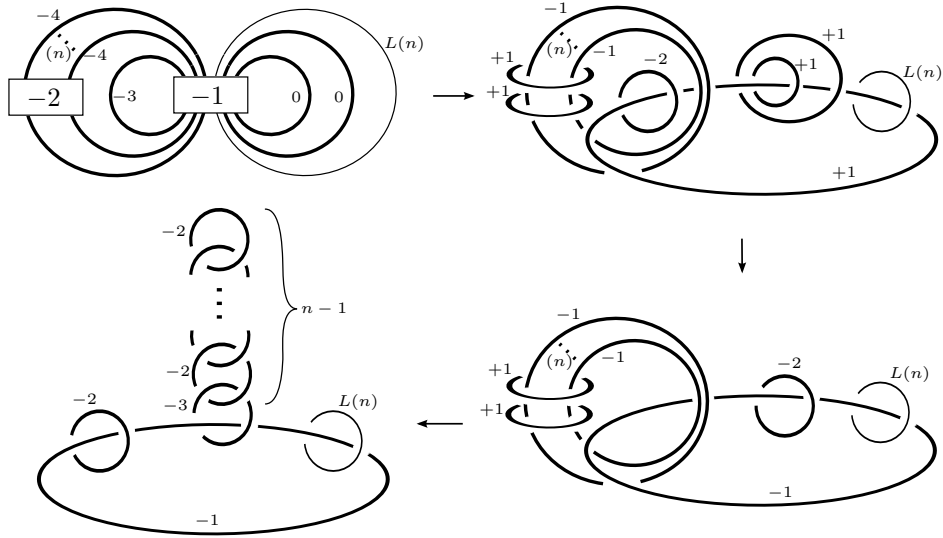}
\caption{{\bf Kirby moves on Figure~\ref{f:torus-smooth}.}}
\label{f:torus-moves}
\end{figure}
\begin{figure}[ht!]
\labellist
\small\hair 2pt
\pinlabel $-1$ [b] at 76 6
\pinlabel $-2$ [rb] at 29 61
\pinlabel $-3$ [r] at 70 61
\pinlabel $-2$ [r] at 70 77
\pinlabel $-2$ [r] at 70 141
\pinlabel $n-1$ [l] at 125 107
\pinlabel $L(n)$ [b] at 139 61
\pinlabel $n$ at 285 55
\pinlabel $L(n)$ [b] at 286 102
\endlabellist
\centering
\includegraphics[scale=0.8]{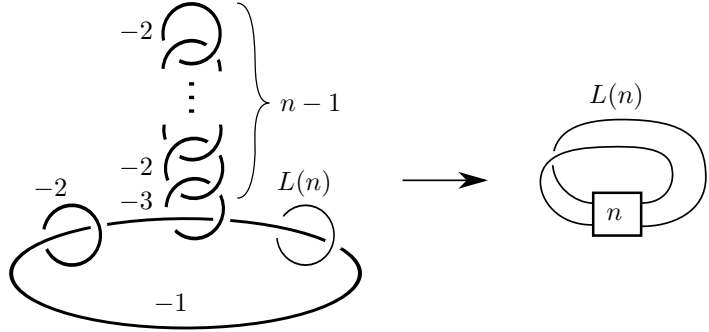}
\caption{{\bf $L(n)$ is the torus knot $T_{2,2n+1}$.}}
\label{f:torus-final}
\end{figure}
Applying a number of ``blow--downs'' yields the right--hand side of
Figure~\ref{f:torus-final}, veryfing that $L(n)$ is isotopic to the
positive torus knot $T_{(2,2n+1)}$ in $S^3$.

  Figures~\ref{f:torus-contact} and~\ref{f:torus-smooth} can be used
  to determine the signature $\sigma (X)$ and the Euler characteristic
  $\chi (X)$ of the 4--manifold obtained by viewing the integral
  surgeries as 4--dimensional 2--handle attachments to $S^3\times
  [0,1]$.  (As it is customary in Heegaard Floer theory, the
  4-manifold $X$ denotes the cobordism between $S^3$ and the
  3--manifold we get after performing the prescribed surgeries.) In
  addition, the rotation numbers define a second cohomology class $c
  \in H^2(X; \bfz )$, and a simple computation shows
\[
\sigma (X)=-n-1, \ \ \chi (X)=n+3, \ \ c^2=-9n-1.
\]
Since we apply two $(+1)$--surgeries, the formula
\[
d_3(\xi )=\frac{1}{4}(c^2-3\sigma (X)-2\chi (X))+q
\]
(with $q$ denoting the number of contact $(+1)$--surgeries) computes
$d_3(\xi _n)$ of the contact structure, providing $1-2n<0$ for all $n
\in \bfn$.  Since the unique tight contact structure $\xi _{st}$ on
$S^3$ has vanishing Hopf invariant $d_3(\xi _{st} )$, we get that
$\xi_n$ is overtwisted. Applying contact $(-1)$--surgery along $L(n)$
we get a tight contact structure, since this $(-1)$--surgery cancels
one of the $(+1)$--surgeries, and a single contact $(+1)$--surgery
along the Legendrian unknot provides the Stein fillable contact
structure on $S^1\times S^2$. Therefore there is no overtwisted disk
in the complement of $L(n)$ (since such a disk would persist after the
surgery), consequently $L(n)$ is non--loose.
\end{proof}
As it is explained in \cite[Section~6]{OSzIII}, the Legendrian link
underlying the surgery diagram for $\xi_n$ (together with the
Legendrian knot $L(n)$) can be put on a page of an open book
decomposition with planar pages, which is compatible with the standard
contact structure $\xi _{st}$ on $S^3$.  This can be seen by
considering the annular open book decomposition containing the
Legendrian unknot (and its Legendrian push--offs), and then applying
the stabilization method described in \cite{Etn} for the stabilized
knots. The monodromy of this open book decomposition can be computed
from the Dehn twists resulting from the stabilizations, together with
the Dehn twists (right--handed for $(-1)$ and left--handed for $(+1)$)
defined by the surgery curves. Notice that one of the left--handed
Dehn twists is cancelled by the monodromy of the annular open book
decomposition we started our procedure with.  This procedure results
in the monodromies given by the curves of Figure~\ref{f:monotoro-a}.
We perform right--handed Dehn twists along solid curves and a
left--handed Dehn twist along the dashed one.
\begin{figure}[ht!]
\labellist
\small\hair 2pt
\pinlabel $L(n)$ [l] at 187 159
\tiny
\pinlabel $(n)$ at 95 243
\endlabellist
\centering
\includegraphics[scale=0.8]{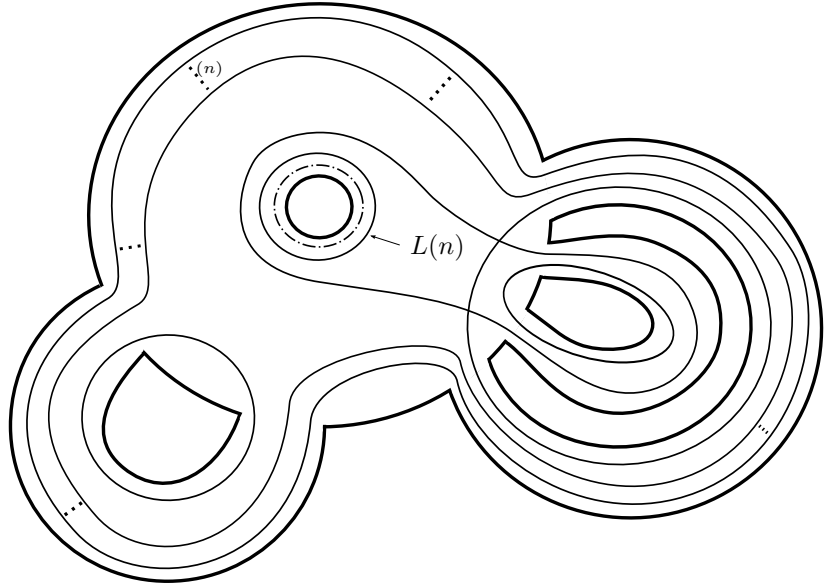}
\caption{{\bf Monodromy of the open book decomposition compatible with
    $L(n)$.}  Solid curves represent right--handed Dehn twists, while
  the dashed one (parallel to $L(n)$) represents a left--handed Dehn
  twist.}
\label{f:monotoro-a}
\end{figure}
The application of the lantern relation simplifies the monodromy 
factorization to the one shown in Figure~\ref{f:monotoro-b}. 
\begin{figure}[ht!]
\labellist
\small\hair 2pt
\pinlabel $L(n)$ at 149 137
\tiny
\pinlabel $(n)$ [b] at 30 118
\endlabellist
\centering
\includegraphics[scale=0.8]{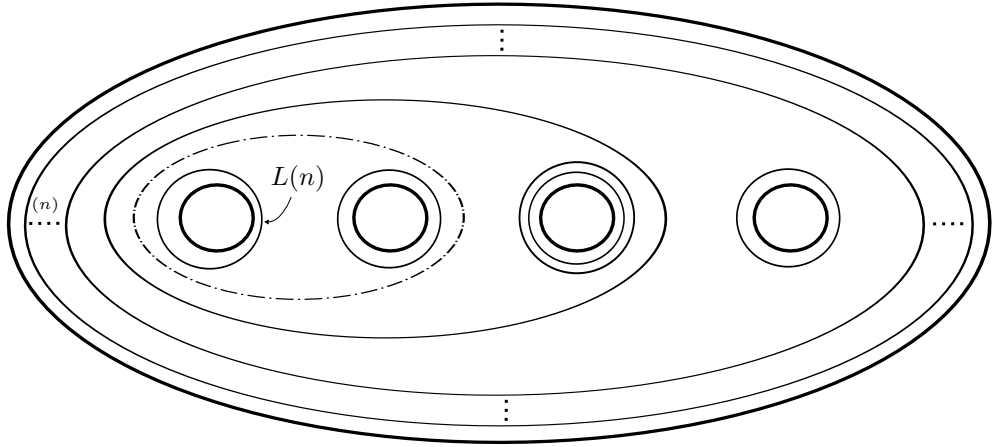}
\caption{{\bf Simplified monodromy.}}
\label{f:monotoro-b}
\end{figure}
Notice that in the monodromy factorization given by
Figure~\ref{f:monotoro-a} there are Dehn twists along intersecting
curves, hence these elements of the mapping class group do not
commute. Therefore, strictly speaking, an order of the Dehn twists
should be specified. Observe, however, that although the elements do
not commute, the fact that there is only one such pair of intersecting
curves implies that the two possible products are conjugate, and
therefore give the same open book decomposition, allowing us to
suppress the specification of the order.

Figure~\ref{f:monotoro-c} helps to visualize the curves on `half' of 
the Heegaard surface, and also indicates the chosen basis.  
\begin{figure}[ht!]
\labellist
\small\hair 2pt
\pinlabel $L(n)$ at 58 77
\pinlabel $a_4$ at 75 90
\pinlabel $a_3$ at 174 90
\pinlabel $a_2$ at 270 90
\pinlabel $a_1$ at 370 90
\tiny
\pinlabel $(n)$ at 408 102
\endlabellist
\centering
\includegraphics[scale=0.8]{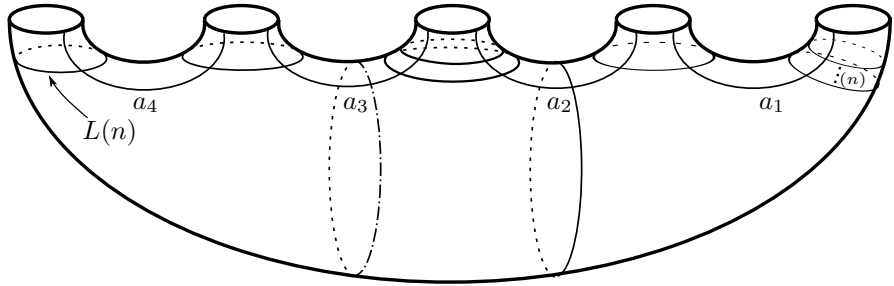}
\caption{{\bf Another view of the monodromy.}  The diagram also
  indicates the chosen basis. (Notice that here $a_4$ intersects $L(n)$
  in a unique point and not $a_1$.)}
\label{f:monotoro-c}
\end{figure}
The open book decomposition found above equips $S^3$ with a Heegaard
decomposition compatible with $L(n)$. The $\alpha$-- and
$\beta$--curves of this decomposition are given in
Figure~\ref{f:hegtoro}.  
\begin{figure}[ht!]
\labellist
\small\hair 0pt
\pinlabel $L(n)$ at 79 39 
\pinlabel $\al_4$ at 227 58
\pinlabel $\al_3$ at 218 193
\pinlabel $\al_2$ at 222 236
\pinlabel $\al_1$ at 223 372
\pinlabel $\be_1$ at 109 406
\pinlabel $\be_2$ at 90 299
\pinlabel $\be_3$ at 93 202
\pinlabel $\be_4$ at 88 100
\pinlabel $n$ at 161 434
\pinlabel $w$ at 128 18
\pinlabel $z$ at 50 211
\tiny
\pinlabel $L$ at 141 116
\pinlabel $M$ at 155 98
\pinlabel $B_3$ at 205 211
\pinlabel $X_1$ at 137 200
\pinlabel $C_2$ at 157 199
\pinlabel $X_2$ at 194 199
\pinlabel $Y_2$ at 167 218
\pinlabel $B_4$ at 149 218
\pinlabel $Y_1$ at 121 217
\pinlabel $B_1$ at 86 257
\pinlabel $P$ at 166 303
\pinlabel $B_2$ at 218 277
\pinlabel $Q$ at 124 311
\pinlabel $A_1$ at 88 381
\pinlabel $A_2$ at 133 392
\pinlabel $A_n$ at 176 405
\pinlabel $A_{n+1}$ [l] at 218 329
\pinlabel $C_1$ at 87 167
\pinlabel $D=D_1$ [r] at 91 68
\endlabellist
\centering
\includegraphics[scale=1]{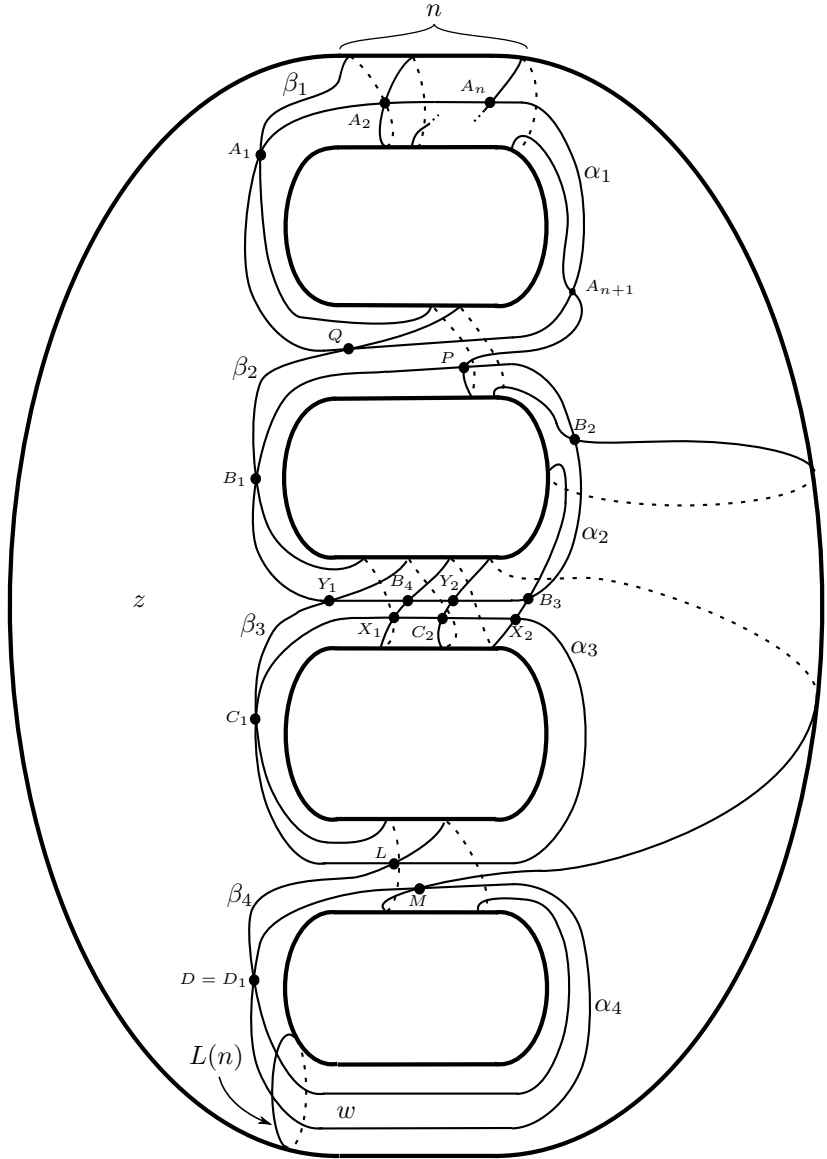}
\caption{{\bf The Heegaard decomposition compatible with $L(n)$.}}
\label{f:hegtoro}
\end{figure}
Recall that we get the $b_i$ arcs by the usual perturbation of the
$a_i$'s and the action of the monodromy yields a Heegaard
decomposition for $S^3$ with the distinguished point $\x$ in ${\mathbb
{T}}_{\alpha }\cap {\mathbb {T}}_{\beta}$ determining the Legendrian
invariant $\Laha(L(n))$. A warning about orientations is in place. 
We illustrated the monodromy curves on the page $S_{+1}$ containing 
the Legendrian knot $L(n)$, but their action must be taken into 
account on the page $S_{-1}$, which comes in the Heegaard surface 
with its opposite orientation; hence right--handed
Dehn twists take curves to the left on the Heegaard surface and vice
versa. The basepoint $z$ is placed in the `large' region of the page
$S_{+1}$, while the point $w$ is in the strip between $a_1$ and $b_1$
shown in Figure~\ref{f:hegtoro}. This choice determines the orientation 
on $L(n)$. 

Recall that since $L(n)$ is isotopic to the positive torus knot
$T_{(2, 2n+1)}$, the Legendrian invariant is given by an element of
$\HFKa (-S^3, T_{(2,2n+1)})$, which is isomorphic to $\HFKa (S^3,
T_{(2,-(2n+1))})$.  This latter group is determined readily from the
Alexander polynomial and the signature of the knot (since it is
alternating), cf.~\cite{OSzalt, Rasmussen}. In fact, we have
\begin{equation}
  \label{eq:NegTorusKnots}
  \HFKa_{n+s}(S^3, T_{(2,-(2n+1))},s)\cong
\left\{\begin{array}{ll}
    \Field & {\text{if $|s|\leq n$}} \\
    0 & {\text{otherwise.}}
    \end{array}
    \right.
    \end{equation}

After these preparations we are ready to determine the
invariants of the Legendrian knots discussed above. 

\begin{prop}
\label{prop:NontrivialLn}
The homology class $\Laha(L(n))$ is determined by the unique nontrivial 
homology class in Alexander grading $1-n$ in $\HFKa(S^3,T_{(2,-2n-1)})$.
\end{prop}

\begin{proof}
  We claim that in the Heegaard diagram of 
  Figure~\ref{f:hegtoro} the point $A_1B_1C_1D_1$ (representing the
  Legendrian invariant $\Laha(L(n))$) is the only 
  intersection point in Alexander grading  $1-n$. 

  We sketch the argument establishing this.
  We can orient every $\alpha$-- and
  $\beta$--curve in the diagram so that their
  intersection matrix (whose $(i,j)^{th}$ entry is the algebraic
  intersection of $\alpha_i$ with $\beta_j$) is
  \[
  M=\left(
    \begin{array}{rrrr}
      n+1 & -1 & 0 & 0 \\
      -1 & 4 & -2 & 0 \\
      0 & -2 & 2 &-1 \\
      0 & 0 & -1 & 1
    \end{array}
  \right).
  \]
  Note that for this particular diagram, the absolute values of the
  algebraic intersection number and the geometric intersection numbers
  coincide.
  
  A simple calculation shows that there are $16n+19$ intersecton
  points in ${\mathbb {T}}_{\alpha}\cap {\mathbb {T}}_{\beta}$.  To
  calculate the relative Alexander gradings of intersection points, it
  is convenient to organize them into types: specifically, the
  intersection points of $\Ta$ and $\Tb$ correspond to permutations
  $\sigma$ and then $4$-tuples in $\prod
  \alpha_i\cap\beta_{\sigma(i)}$, which in terms of the notation of
  Figure~\ref{f:hegtoro} are given by quadruples of letters;
  specifically, intersection points all have one of the following five
  types
  \[
  \begin{array}{lllll}
    ABCD&PQCD&AXYD&ABLM&PQLM. 
  \end{array}\]
  
  We begin by calculating relative Alexander gradings of intersection
  points of the same type. Consider first the relative gradings
  between points of the form $A_i\threestars$ and
  $A_{i+1}\threestars$, where now $\threestars$ is a fixed triple of
  intersection points. Start from an initial path $\delta$ which
  travels from $A_i$ to $A_{i+1}$ along $\alpha_1$, then back from
  $A_{i+1}$ to $A_i$ along $\beta_1$. We add to this cycle copies of
  $\alpha_i$ and $\beta_j$ to obtain a null-homologous cycle. This can
  be done since the $\alpha_i$'s and $\beta_j$'s span $H_1(\Sigma)$,
  which follows from the fact that the ambient 3--manifold is $S^3$.
  Concretely, we need to solve the expression
  \[
  \sum n_i [\alpha _i]+\sum m_j [\beta _j] +[\delta ]=0
  \]
  for $n_i$ and $m_j$.  In fact, all we are concerned about is the
  difference in local multiplicities at $z$ and $w$ in the
  null-homology of the above expression. Since $z$ and $w$ are
  separated by $\alpha_4$, and our initial curve $\delta$ is supported
  on $\alpha_1$ and $\beta_1$, this difference is given by the
  multiplicity $n_4$ of $\alpha_4$ in the above expression. In view of
  the fact that $\#\beta_i\cap\beta_j=0$, $n_4$ can be obtained as the
  last coefficient of $M^{-1} v$, where $v$ is the vector whose
  $i^{th}$ coordinate is $\#\delta\cap \beta_j$ and $M$ is the
  incidence matrix. We can find paths of the above type $\delta $
  connecting $A_i$ and $A_{i+1}$ whose intersection numbers with the
  $\beta_j$'s are given by the vector $(1,0,0,0)$, hence giving that
  the Alexander grading of $A_i\threestars$ is two smaller than
  the Alexander grading of $A_{i+1}\threestars$. We express this by saying
  that the relative gradings of the $A_i$ are given by $2i$. Repeating
  this procedure for the other curves, we find that relative gradings
  of $B_1,...,B_4$ (completed by a fixed triple to quadruples of
  intersection points) are given by $0$, $2n$, $4n+2$, and $2n+1$
  respectively; for $C_1$ and $C_2$ the relative gradings are $0$ and
  $2n+1$ respectively, for $X_1$ and $X_2$ they are $0$ and $2n+1$,
  and they are also $0$ and $2n+1$ for $Y_1$ and $Y_2$.  Putting these together,
  we can calculate the relative Alexander gradings of any
  two intersection points of the same type.
  
  To calculate the relative Alexander gradings of points of different
  types, we use three rectangles. Specifically, intersection points of
  the form $PQ\twostars$ and $A_{n+1}B_1\twostars$ (where here
  $\twostars$ denotes some fixed pair of intersection points) have
  equal Alexander gradings, shown by the rectangle with vertices
  $PB_1QA_{n+1}$, avoiding both $w$ and $z$.  Similar argument relates
  points of type $X_1Y_1\twostars$ to the ones of type
  $B_4C_1\twostars$. The rectangle $LDMC_1$ contains $w$ once, showing 
  that the Alexander gradings of elements of the form 
  $LM\twostars$ are one less than the Alexander gradings of elements of
  the form $C_1D\twostars$.  This allows us to calculate
  the relative Alexander gradings of any two intersection points.
  
  Given this information, it is now straightforward to see that there
  are no other intersection points in the same Alexander grading as
  $A_1 B_1 C_1 D_1$, and hence that it represents a homologically
  nontrivial cycle in $\HFKa(S^3, T_{(2,-(2n+1))})$. Thus, we have
  shown that for all $n\in\bfn$, the class $\Laha(L(n))$ is a
  nontrivial generator in $\HFKa(S^3, T_{(2,-(2n+1))})$.
  
  In fact, the absolute Alexander grading of generators is pinned down
  by the following symmetry property: the Alexander grading is normalized
  so that the parity of the number of points of
  Alexander grading $i$ coincides with the parity of the number of
  points of Alexander grading $-i$.  Using this property, one finds
  that $A_1 B_1 C_1 D_1$ is supported in Alexander grading $1-n$.
\end{proof}

\begin{rem}
{\rm 
Recall (cf.~\cite{OSzalt, Rasmussen}) that 
\begin{equation}
  \label{eq:NegKnots}
  \HFKm (S^3, T_{(2,-(2n+1))})\cong
  \Field ^{n}\oplus \Field[U],
\end{equation}
where the top generator of the free $\Field [U]$--module is at $(A=n,
M=2n)$ while the $n$ generators of the $\Field^{n}$ summand are of
bi--degrees 
\[
(A=n-1-2i, \quad M=2n-1-2i),\quad i=0, \ldots , n-1. 
\]
It is not hard to show that the above computation implies that
$\La(L(n))$ is determined by the unique nonzero $U$--torsion element
of $\HFKm(S^3, T_{(2,-(2n+1))})$ in Alexander grading $1-n$.}
\end{rem}

\subsection*{Negative Legendrian torus knots $T_{(2,-(2n-1))}$ in overtwisted 
contact $S^3$'s}

For $k,l\geq 0$ let us consider the knot $L_{k,l}$ in the contact
3--manifold $(Y_{k,l}, \xi _{k,l})$ given by the surgery presentation
of Figure~\ref{f:ot}. Let $n=k+l+2$.
\begin{figure}[ht!]
\labellist
\small\hair 2pt
\pinlabel $L_{k,l}$ [l] at 453 55
\pinlabel $+1$ [l] at 458 82
\pinlabel $+1$ [l] at 454 107
\pinlabel $-1$ [l] at 453 131
\pinlabel $-1$ [l] at 514 154
\pinlabel $-1$ [l] at 513 201
\pinlabel $l$ at 469 391
\pinlabel $k$ at 46 343
\endlabellist
\centering
\includegraphics[scale=0.4]{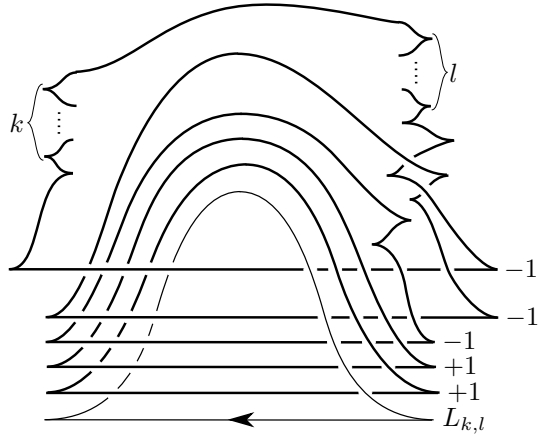}
\caption{{\bf Non--loose Legendrian $T_{(2,-(2n-1))}$ torus knot in $S^3$.}}
\label{f:ot}
\end{figure}

\begin{prop}\label{p:surgd3}
The contact 3--manifold specified by the surgery diagram of Figure~\ref{f:ot} 
is $(S^3, \xi_{k,l})$ with $d_3(\xi _{k,l})=2l+2$, hence $\xi_{k,l}$ is 
overtwisted.  The Legendrian knot $L_{k,l}$ as a smooth knot is isotopic to
the negative $(2,2n-1)$ torus knot $T_{(2,-(2n-1))}$, and it is
non--loose in $(S^3, \xi _{k,l})$.
\end{prop}

\begin{proof}
The simple Kirby calculus argument illustrated in Figure~\ref{f:kirby}
\begin{figure}[ht!]
\labellist
\small\hair 2pt
\pinlabel $L_{k,l}$ at 187 255
\pinlabel $-1$ at 106 227
\pinlabel $0$ at 142 226
\pinlabel $0$ at 161 226
\pinlabel $-3$ at 15 200
\pinlabel $-3$ at 38 209
\pinlabel $-n$ at 29 232

\pinlabel $-n$ at 230 258
\pinlabel $-2$ at 265 253
\pinlabel $L_{k,l}$ at 399 249
\pinlabel $-2$ at 295 214
\pinlabel $1$ at 323 196
\pinlabel $1$ at 361 258
\pinlabel $1$ at 343 253

\pinlabel $-n$ at 38 158
\pinlabel $-2$ at 86 167
\pinlabel $-1$ at 115 107
\pinlabel $L_{k,l}$ at 194 157
\pinlabel $-2$ at 133 168

\pinlabel $-n$ at 257 150
\pinlabel $-1$ at 294 108
\pinlabel $1$ at 328 125
\pinlabel $L_{k,l}$ at 381 147
\pinlabel $-1$ at 375 110

\pinlabel $-n$ at 84 68
\pinlabel $0$ at 63 25
\pinlabel $L_{k,l}$ at 182 68
\pinlabel {$T_{2, -(2n-1)}\subset S^3$ (after two slides)} [l] at 233 39

\endlabellist
\centering
\includegraphics[scale=0.8]{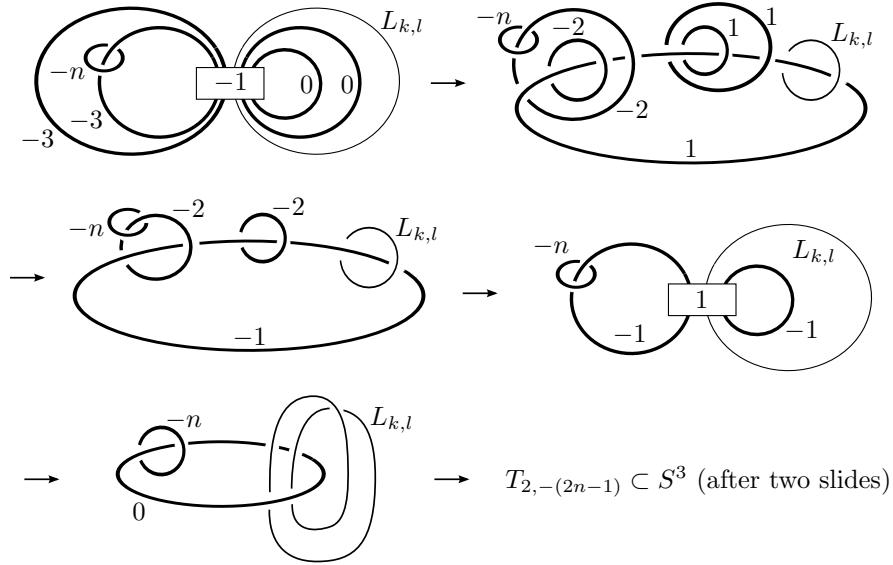}
\caption{{\bf Kirby moves on the diagrams of Figure~\ref{f:ot}.}}
\label{f:kirby}
\end{figure}
shows that the 3--manifold $Y_{k,l}$ is $S^3$, while the formula 
\[
d_3 (\xi _{k,l})=\frac{1}{4}(c^2-3\sigma -2 b_2)+q
\] 
of \cite{DGS} computes the Hopf invariant of $\xi _{k,l}$. As always,
$\sigma $ and $b_2$ denote the signature and the second Betti number
of the 4--manifold $X$ specified by the underlying smooth surgery
diagram, while $c\in H^2 (X; \bfz )$ is specified by the rotation
numbers of the contact surgery curves, and $q$ is the number of
$(+1)$--surgeries. Simple algebra verifies that
\[
d_3(\xi _{k,l})=2l+2\geq 2,
\]
and since the unique tight structure on $S^3$ has vanishing Hopf
invariant, we get that $\xi _{k,l}$ is overtwisted. Following the
Kirby moves of Figure~\ref{f:kirby} with the knot $L_{k,l}$ we arrive
to the last surgery picture of Figure~\ref{f:kirby}, and by sliding $L_{k,l}$ twice over the
$(-n)$--framed unknot and cancelling the pair we see that $L_{k,l}$
is, in fact, isotopic to the negative $(2,2n-1)$ torus knot
$T_{(2,-(2n-1))}\subset S^3$.

If contact $(-1)$--surgery on $L_{k,l}$ provides a tight contact
structure, then $L_{k,l}$ is obviously non--loose, since any
overtwisted disk in its complement would give an overtwisted disk in
the surgered manifold. In our case, however, contact $(-1)$--surgery
simply cancels one of the $(+1)$--surgeries defining $\xi _{k,l}$, and
since a single contact $(+1)$--surgery on the Legendrian unknot
provides the contact boundary of the Stein 1--handle
(cf. \cite{LSSeif}), we get that $(-1)$--surgery along $L$ provides a
Stein fillable contact structure.
\end{proof}

\begin{rem} {\rm In fact, by stabilizing $L_{k,l}$ on the \emph{left}
    and then performing a contact $(-1)$--surgery we still get a tight
    contact 3--manifold: it will be Stein fillable if we perform only
    one stabilization, and not Stein fillable but tight for more
    stabilizations. The tightness of the result of these latter
    surgeries were verified in \cite{GLS2} by computing the contact
    Ozsv\'ath--Szab\'o invariants of the resulting contact structures.
    Notice that this observation implies that after arbitrary many
    left stabilizations $L_{k,l}$ remains non--loose, which, in view
    of Proposition~\ref{p:stab}, is a necessary condition for
    $\La(L_{k,l})$ to be nonvanishing. (Finally note that performing
    contact $(-1)$--surgery on $L_{k,l}$ after a single \emph{right}
    stabilization provides an overtwisted contact structure (cf. now
    \cite[Section~5]{GLS2})). In contrast, for the non--loose knots
    $L(n)$ of the previous subsection (also having nontrivial
    $\La$--invariants) the same intuitive argument does not work,
    since some negative surgery on the knot $L(n)$ will produce a
    contact structure on the 3--manifold $S^3_{2n-1}(T_{2,2n+1})$ and
    since this 3--manifold does not admit any tight contact structure
    \cite{OSzII}, the result of the surgery will be overtwisted
    independently on the chosen stabilizations. Nevertheless, the
    overtwisted disk in the 3--manifold obtained by only negative
    stabilizations cannot be in the complement of the knot $L(n)$,
    since such stabilizations are still non--loose (shown by the
    nonvanishing of the invariant $\LegInva (L(n))$).}
\end{rem}

Next we want to determine the classical invariants of $L_{k,l}$.
Considering the problem slightly more generally, and let $L\subset Y$
be a null--homologous Legendrian knot in a contact 3--manifold $(Y,\xi
)$.  Assume furthermore that $Y$ is a \emph{rational homology
  3--sphere}.  The knot $L$ has two ``classical invariants'': the
Thurston--Bennequin and the rotation numbers $\tb(L), \rot(L)$.
Recall that the Thurston--Bennequin invariant of a Legendrian knot is
defined as the framing induced by the contact 2--plane field
distribution on the knot, hence, given an orientation for the knot, it
naturally gives rise to a homology class $\TB\in H_1(Y-L; \Z)$
(supported in a tubular neighborhood of $L$).  Fixing a Seifert
surface $F$ for $L$, the oriented intersection number of $\TB$ with
$F$ provides a numerical invariant, called the
\emph{Thurston--Bennequin number} $\tb(L)\in \Z$. This intersection
number is independent of the choice of the Seifert surface; moreover,
the Thurston--Bennequin number is independent of the orientation of
$L$ since the orientation of $\TB$ and the orientation of $F$ both
depend on this choice.  Analogously, the rotation $\rot$ of $L$ is the
relative Euler class of $\xi $ when restricted to $Y-L$, with the
trivialization given by the tangents of $L$. Again, a Seifert surface
for $L$ can be used to turn this class uniquely into an integer, also
denoted by $\rot (L)$. Note that unlike the Thurston--Bennequin
number, the sign of the rotation number does depend on the orientation
of $L$.

Suppose that ${\mathbb {S}} ={\mathbb {S}}_+\cup {\mathbb
  {S}}_-\subset (S^3, \xi _{st})$ is a contact $(\pm 1)$--surgery
presentation of the contact 3--manifold $(Y, \xi )$ and $L$ is a
Legendrian knot in $(S^3 , \xi _{st})$ disjoint from ${\mathbb {S}}$,
null--homologous in $Y$.  The Thurston--Bennequin and rotation numbers
of $L$ in $(Y,\xi)$ can be obtained from the Thurston--Bennequin and
rotation numbers of the individual components of ${\mathbb {S}}$ and
$L$ through the following data. Let $\tb_0$ denote the
Thurston--Bennequin number of $L$ as a knot in the standard contact
3--sphere (which, in terms of a front projection, is equal to the
writhe minus half the number of cusps).  Writing
${\mathbb{S}}=\cup_{i=1}^n L_i$, let $a_i$ be the integral surgery
coefficient on the link component $L_i$; i.e. $a_i=\tb(L_i)\pm 1$ if
$L_i\in {\mathbb {S}}_{\pm}$.  Define the {\em linking matrix}
\[
M(a_0,a_1,...,a_n)=\left(m_{i,j}\right)_{i,j=0}^n
\]
where
\[
m_{i,j}:=
\left\{\begin{array}{ll}
    a_i & {\text{if $i=j$}} \\
    {\mathrm{lk}}(L_i,L_j) & {\text{if $i\neq j$,}}
\end{array}
\right.
\]
with the convention that $L=L_0$ and ${\mathbb {S}}=\cup _{i=1}^n L_i$. 
Similarly, let
$M(a_1,...,a_n)$ denote the matrix $\left(m_{i,j}\right)_{i,j=1}^n$.
Consider the
integral rotation numbers $r_0,...,r_n$ obtained from our Legendrian
knot $L$ and Legendrian presentation as a link in $S^3$. 

\begin{lem}
  Suppose that ${\mathbb {S}} ={\mathbb {S}}_+\cup {\mathbb
    {S}}_-\subset (S^3, \xi _{st})$ is a contact $(\pm 1)$--surgery
  presentation of the contact 3--manifold $(Y, \xi )$ and $L$ is a
  Legendrian knot in $(S^3, \xi _{st})$ disjoint from ${\mathbb {S}}$,
  null--homologous in $Y$. Then the Thurston-Bennequin and rotation
  numbers $\tb (L)$ and $\rot (L)$ can be extracted from the above
  data by the formulae:
  \begin{equation}\label{e:tbformula}
    \tb (L)=\tb _0 + \frac{\det(M(0,a_1,...,a_n))}{\det(M(a_1,...,a_n))}
  \end{equation}
  and
  \begin{equation}\label{e:rotformula}
    \rot (L)=r_0 - \langle \left(\begin{array}{c}r_1 \\ \vdots \\ r_n 
      \end{array}\right), M^{-1}\cdot 
    \left(\begin{array}{c}\lk(L_0,L_1) \\ \vdots \\ \lk(L_0,L_n),
      \end{array}\right)\rangle ,
  \end{equation}
  where $M=M(a_1,...,a_n)$.
\end{lem}

\begin{proof}
  We turn to the verification of Equation~\eqref{e:rotformula}, after
  a few preliminary observations.  Let $\mu_i\subset S^3-(L\cup
  {\mathbb{S}})$ be a meridian for $L_i\subset L\cup {\mathbb{S}}$,
  and $\lambda_i$ be its corresponding longitude.  Recall that
  $H_1(S^3-(L\cup {\mathbb{S}}))$ is a free $\Z$-module, generated by
  the meridians $\mu_i$ (we continue with the convention that the
  $i=0$ component of the link $L\cup {\mathbb{S}}$ is $L$). We can
  express the homology class of $\lambda_i$ in terms of the other
  meridians by the expression
  $$\lambda_i=\sum_{j\neq i} \lk(L_i,L_j)\cdot \mu_j.$$
  The homology groups of the surgered manifold are gotten from the
  homology groups of the link complement by dividing out by the
  relations $a_i\mu_i+\lambda_i=0$ for $i=1,...,n$; more precisely,
  $H_1(S^3-(L\cup {\mathbb {S}}))$ is freely generated by
  $\mu_0,...,\mu_n$, while $H_1(Y-L)\cong \Z$ is obtained from this
  free group by dividing out the $n$ relations
  \[
  a_i\cdot \mu_i+\sum_{j\neq i} \lk(L_i,L_j)\cdot \mu_j=0,
  \quad i=1,\ldots, n.
  \]
  
  The rotation numbers $r_i$ can be thought of as follows. Let
  $e(\xi,L\cup {\mathbb S})\in H^2(S^3,L\cup {\mathbb S})$ denote the
  relative Euler class of $\xi$ relative to the trivialization it
  inherits along $L\cup {\mathbb S}$.  Then, the rotation numbers are
  the coefficients in the expansion of the Poincar{\'e} dual
  $\PD[e(\xi,L\cup {\mathbb S})]\in H_1(S^3-(L\cup {\mathbb {S}}))$ in
  terms of the basis of meridians $\PD[e(\xi,L\cup {\mathbb
    {S}})]=\sum_{i=0}^n r_i\cdot \mu_i$.  Similarly, $H_1(Y-L)\cong\Z$
  is generated by $\mu_0$, the meridian of $L$, and the rotation
  number $\rot (L)$ is calculated by $\PD(e(\xi,L))=\rot (L) \cdot
  \mu_0$.  Note also that $\PD[e(\xi,L)]$ is the image of
  $\PD[e(\xi,L\cup {\mathbb S})]$ under the inclusion $j\colon
  S^3-(L\cup {\mathbb S}) \subset Y-L$. Thus, to find $\rot (L)$, it
  suffices to express $j_*(\mu_i)$ in terms of $\mu_0$. Write
  \begin{eqnarray*}
    \Lambda=\left(\begin{array}{c}\lk(L_0,L_1) \\ \vdots \\ \lk(L_0,L_n),
      \end{array}\right)
    &{\text{and}}&
    R=\left(\begin{array}{c}r_1 \\ \vdots \\ r_n
      \end{array}\right).
  \end{eqnarray*}
  In view of our presentation for $H_1(S^3-L)$, we see that for all
  $i>0$ we have $\mu_i=c_i\cdot \mu_0$, where $c_i$ is the $i^{th}$
  entry in the vector $-M^{-1}\cdot \Lambda$.  It follows that
  $$\sum_{i=0}^n r_i\cdot \mu_i =(r_0 - \langle
  R, M^{-1}\cdot \Lambda\rangle)\cdot \mu_0,$$
  estabilishing Equation~\eqref{e:rotformula}.

  We turn now to Equation~\eqref{e:tbformula}.
  Choose $a_0$ so that the curve $\lambda'=a_0\cdot \mu_0 + \lambda_0$
  is null--homologous in $Y-L$.  Letting $\TB$ denote the
  Thurston--Bennequin framing curve of $L=L_0$, we have that
  $$\TB=\tb_0\cdot \mu_0+\lambda_0=\tb(L)\cdot \mu_0 + \lambda';$$
  thus  $\tb(L)-\tb_0=-a_0$. In view of our presentation
  for this first homology group, we see that $a_0$ is determined by
  the condition that
  \begin{eqnarray*}
    0&=&\left|
      \begin{array}{llll}
        a_0 & \lk(L_0,L_1) & ...& \lk(L_0, L_n) \\
        \lk(L_0,L_1) & a_1 & ... & \lk(L_0,L_1) \\
        \vdots & \vdots &\ddots & \vdots \\
        \lk(L_n,L_0) & \lk(L_n,L1) & \hdots & a_n
      \end{array}\right|  \\ \\
    &=& a_0 \cdot \det(M(a_1,...,a_n)) + \det(M(0,a_1,...,a_n)).
  \end{eqnarray*}
  Equation~\eqref{e:tbformula} follows.
\end{proof}

\begin{lem}\label{l:tbrot}
With the orientation given by Figure~\ref{f:ot}, 
the Thurston--Bennequin and rotation numbers of the knot $L_{k,l}$ are
given by 
\[
\tb (L_{k,l})=-4(k+l)-6 \quad {\mbox {and}} \quad \rot (L_{k,l})=-6l-2k-7.
\]
\end{lem}

\begin{proof}
  Applying Formulae~\eqref{e:tbformula} and~\eqref{e:rotformula} to
  the surgery diagrams defining $L_{k,l}$, simple algebra gives the
  statement of the lemma.
\end{proof} 

As it is explained in the previous subsection (resting on observations
from \cite[Section~6]{OSzIII}), the surgery diagram for $\xi _{k,l}$
(together with the Legendrian knot $L_{k,l}$) can be put on a page of
an open book decomposition with planar pages, which is compatible with
the standard contact structure $\xi _{st}$ on $S^3$. The diagram for
all choices of $k$ and $l$ is less apparent, hence we restrict our
attention first to $k=0$.  In this case we get the monodromies defined
by the curves of Figure~\ref{f:mono-a}.  
\begin{figure}[ht!]
\labellist
\small\hair 2pt
\pinlabel $L_{0,l}$ at 195 158
\pinlabel $l$ at 97 104
\endlabellist
\centering
\includegraphics[scale=0.7]{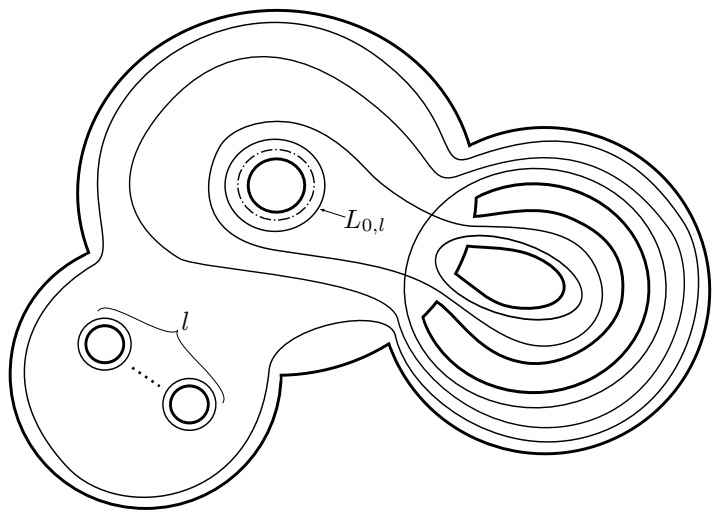}
\caption{{\bf Monodromy of the open book decomposition compatible with
    $L_{0,l}$.}}
\label{f:mono-a}
\end{figure}
The application of the lantern relation simplifies the monodromy factorization to the one
illustrated in Figure~\ref{f:mono-b}. 
\begin{figure}[ht!]
\labellist
\small\hair 2pt
\pinlabel $L_{0,l}$ at 155 140
\pinlabel $l$ at 389 155
\endlabellist
\centering
\includegraphics[scale=0.7]{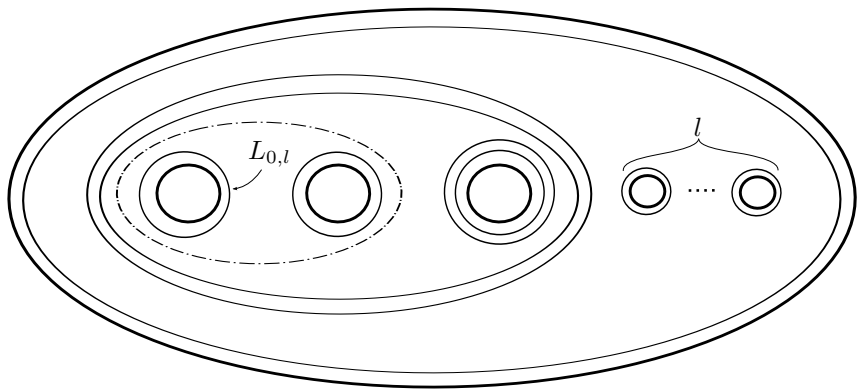}
\caption{{\bf The simplified monodromy of the open book for $L_{0,l}$.}}
\label{f:mono-b}
\end{figure}
Figure~\ref{f:mono-c} helps in visualizing the curves on `half' of the 
Heegaard surface. The diagram also indicates the chosen basis.
\begin{figure}[ht!]
\labellist
\small\hair 2pt
\pinlabel $L_{0,l}$ at 47 70
\pinlabel $l+1$ at 349 154
\endlabellist
\centering
\includegraphics[scale=0.8]{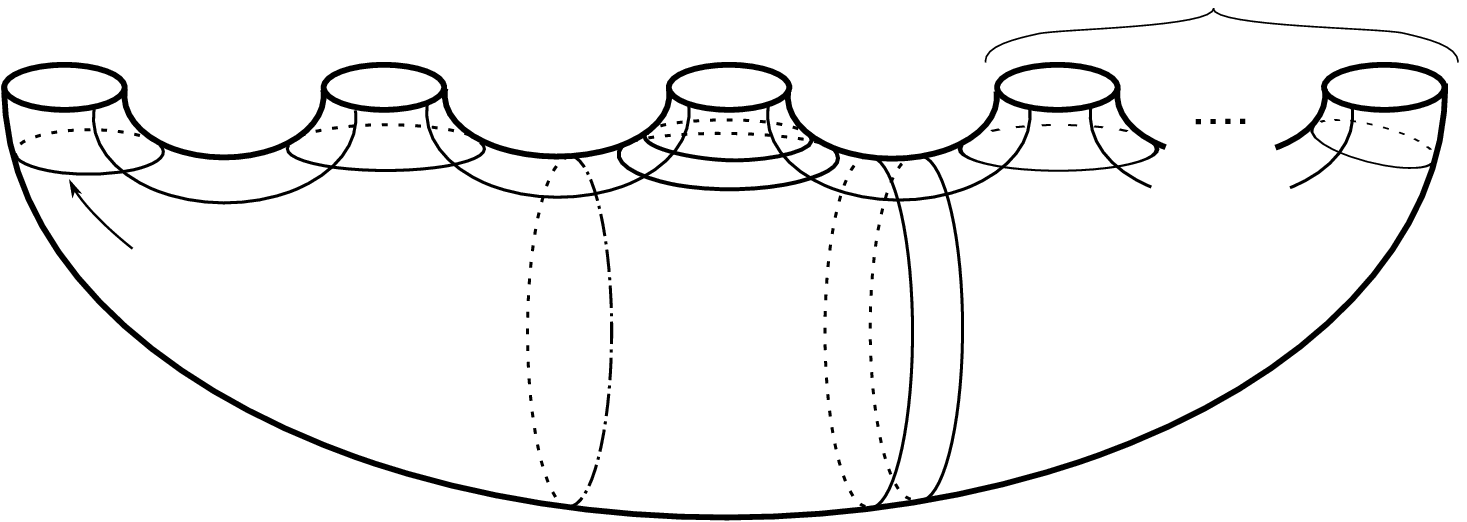}
\caption{{\bf Another view of the open book for $L_{0,l}$.}}
\label{f:mono-c}
\end{figure}
Notice that in the monodromy illustrated in Figure~\ref{f:mono-a}
there are Dehn twists with intersecting curves, hence these elements
of the mapping class group do not commute. As before, the two products
are conjugate, hence there is no need to record the order.

In a similar spirit the general diagrams could be deduced --- since we
will not use them in our computations, we will be content with working
on a further special case when $k=1$ and $l\geq 1$; the result is
given by Figure~\ref{f:mononext}. 
\begin{figure}[ht!]
\labellist
\tiny\hair 2pt
\pinlabel $L(1,l)$ [lb] at 110 152
\pinlabel $l-1$ at 378 136
\endlabellist
\centering
\includegraphics[scale=0.6]{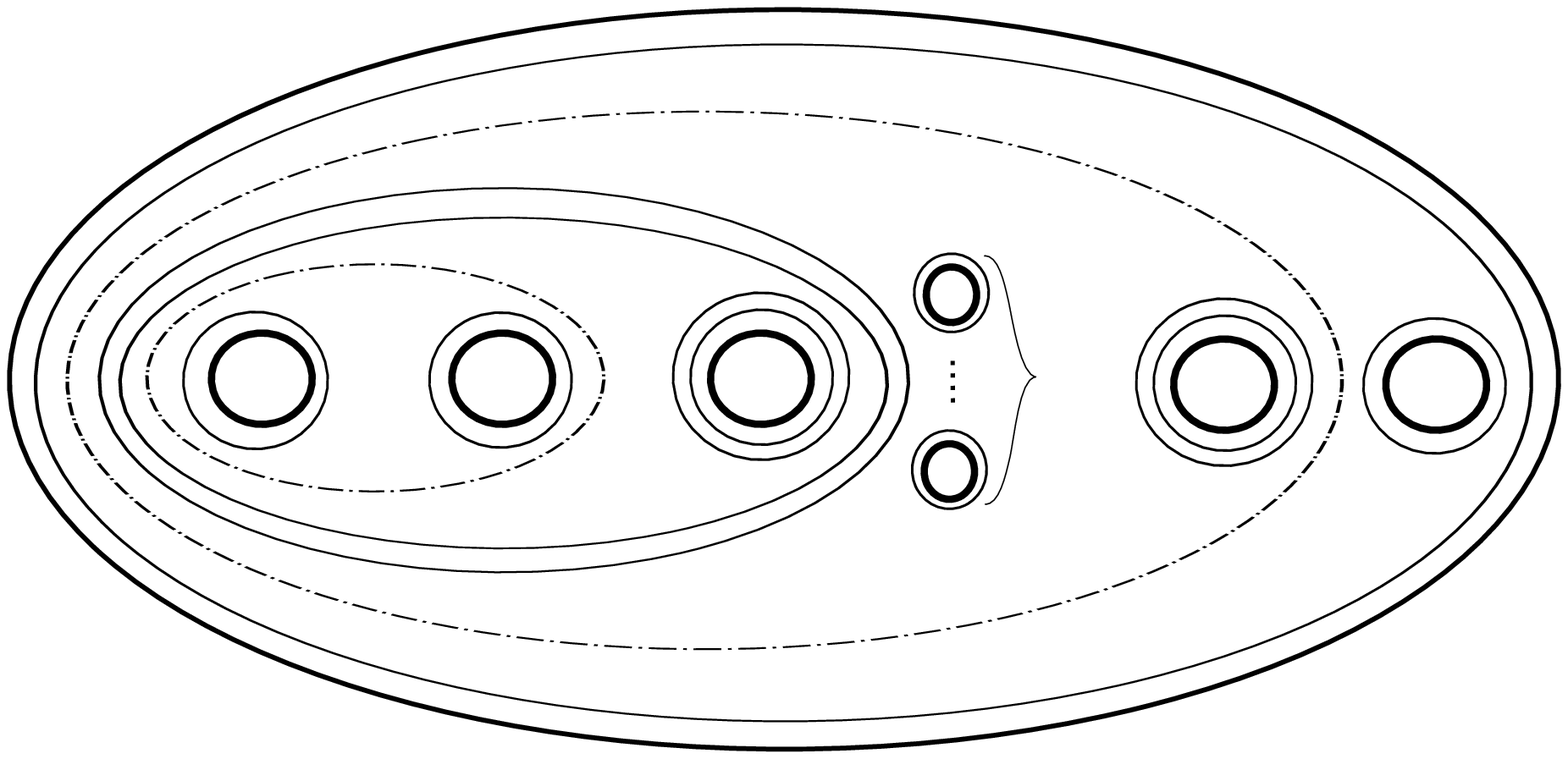}
\caption{{\bf The monodromy factorization for $L(1,l)$ with $l\geq 1$.}}
\label{f:mononext}
\end{figure}
We just note here that in these cases we need to apply the lantern
relation twice to get commuting Dehn twists as illustrated in
Figure~\ref{f:mononext}. For this reason, these monodromy
factorizations contain two left--handed Dehn twists.  In
Figure~\ref{f:monoredrawn} the monodromy for $L(1,l)$ with $l\geq 1$
is represented on 'half' of the Heegaard surface.
\begin{figure}[ht!]
\labellist
\tiny\hair 2pt
\pinlabel $L(1,l)$ at 39 74
\pinlabel $l-1$ at 185 151
\endlabellist
\centering
\includegraphics[scale=0.9]{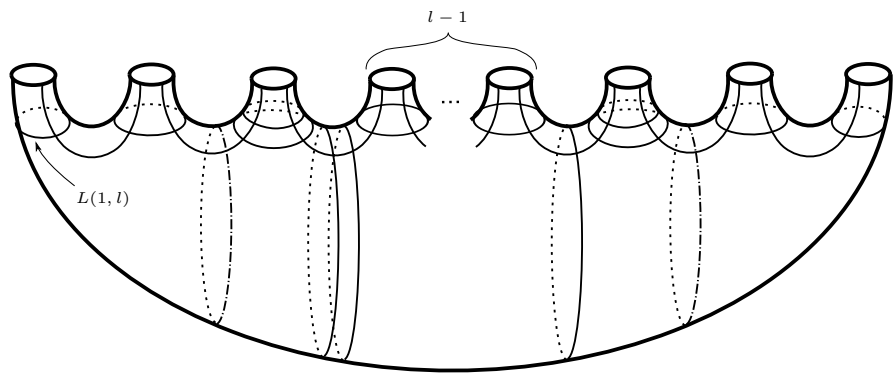}
\caption{{\bf The open book for $L(1,l)$ with $l\geq 1$.}}
\label{f:monoredrawn}
\end{figure}
The open book decomposition found above equips $S^3$ with a Heegaard
decomposition compatible with $L_{k,l}$; the $\alpha$-- and
$\beta$--curves of this decomposition for $k=0$ are given by
Figure~\ref{f:heg-a} when $l=0$ and by Figure~\ref{f:heg-b} when $l>0$.  
\begin{figure}[ht!]
\labellist
\small\hair 2pt
\pinlabel $L_{0,0}$ at 76 41
\pinlabel $w$ at 130 18
\pinlabel $z$ at 44 161
\endlabellist
\centering
\includegraphics[scale=0.8]{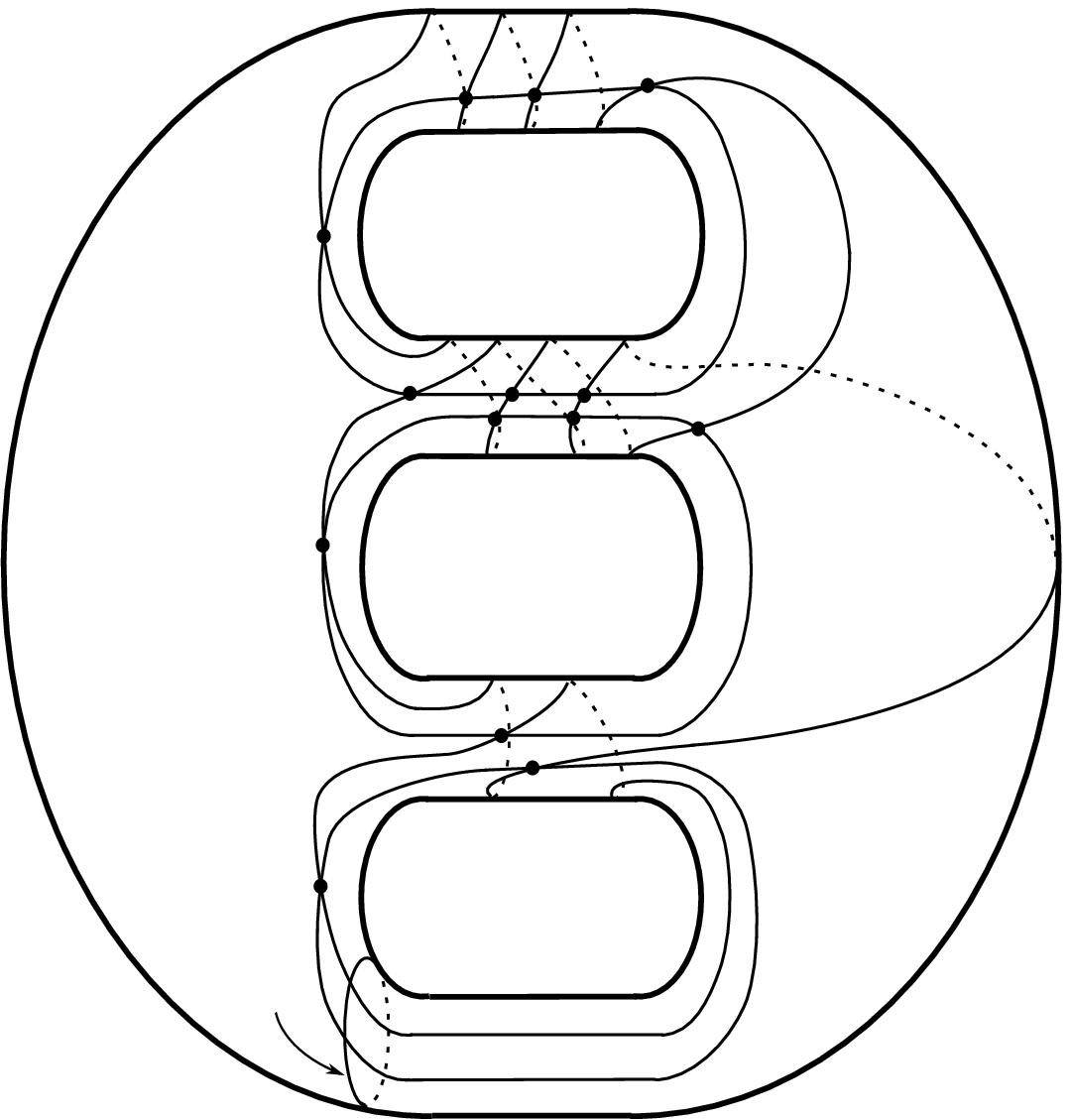}
\caption{{\bf The Heegaard decomposition for $L(0,0)$.}}
\label{f:heg-a}
\end{figure}
\begin{figure}[ht!]
\labellist
\small\hair 2pt
\pinlabel $z$ at 45 306
\pinlabel $w$ at 145 18
\tiny
\pinlabel $L_{0,l}$ at 81 43
\pinlabel $A$ at 78 661
\pinlabel $B$ at 153 710
\pinlabel $C$ at 153 622
\pinlabel $C'$ at 206 622
\endlabellist
\centering
\includegraphics[scale=0.68]{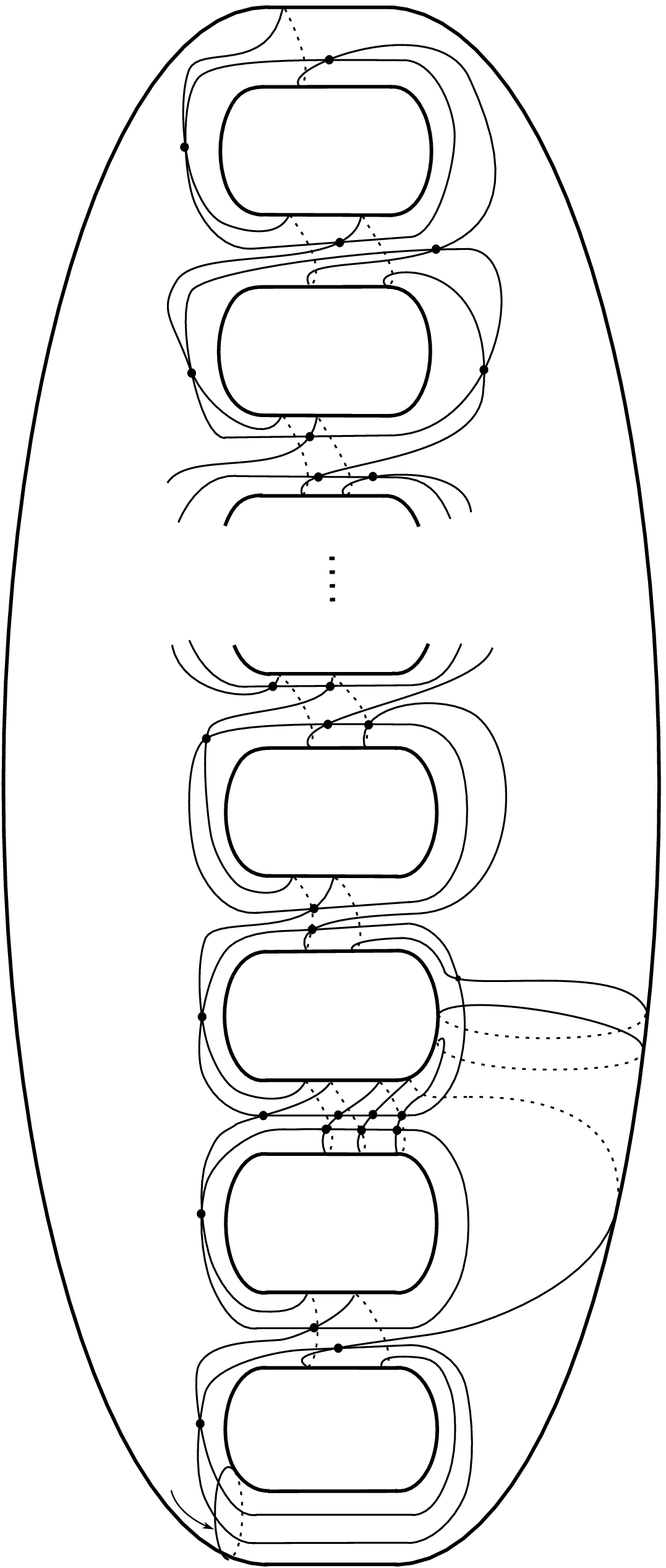}
\caption{{\bf The Heegaard decomposition for $L(0,l)$, $l>0$.}}
\label{f:heg-b}
\end{figure}
Recall that we get the $b_i$ curves by the usual perturbation of the
$a_i$'s and the action of the monodromy yields a Heegaard
decomposition for $S^3$ with the distinguished point $\x$ in ${\mathbb
{T}}_{\alpha }\cap {\mathbb {T}}_{\beta}$ determining the Legendrian
invariant.  As usual, the basepoint $z$ is
placed in the `large' region of the page $S_{+1}$, while the point $w$ 
giving rise to $\Laha$ is in the strip between $a_1$ and $b_1$, 
as indicated in the pictures. It is easy to see that moving $w$ to the other 
domain in the strip gives $\Laha=0$, because there is a holomorphic disk 
running into $\x$.  

Recall that $L_{k,l}$ is isotopic to the negative torus knot
$T_{(2, -(2n-1))}$. The group $\HFKa (-S^3, T_{(2,-(2n-1))})$ is isomorphic 
to the Floer homology $\HFKa(S^3, T_{(2,2n-1)})$ of the positive torus knot
$T_{(2,2n-1)}$.  In this case we have the following (compare 
Equation~\eqref{eq:NegTorusKnots}):
\begin{equation}
  \label{eq:PosTorusKnots}
  \HFKa_{s-n+1}(S^3, T_{(2,2n-1)},s)\cong
\left\{\begin{array}{ll}
    \Field & {\text{if $|s|\leq n-1$}} \\
    0 & {\text{otherwise.}}
    \end{array}
    \right.
    \end{equation}
After these preparations we are ready to determine the
invariants of some of the Legendrian knots discussed above. The computation
admits a relatively simple scheme when $k=0$, so we start with that
case.

\begin{thm}\label{t:L0l}
Let $L_{0,l}$ be given the orientation specified by the diagram of Figure~\ref{f:heg-a} or 
Figure~\ref{f:heg-b}. Then, $\Laha(L_{0,l})\neq 0$. 
\end{thm}

\begin{proof}
  We show that the intersection point $\x_0\in\CFKa$, which determines 
  $\Laha(L_{0,l})$, is alone in  its Alexander grading.

\bigskip
 
{\bf Claim:} Any other intersection point $\y $ in the Heegaard
decomposition given by Figure~\ref{f:heg-a} or~\ref{f:heg-b} has Alexander grading $A(\y
)$ strictly less than $A(\x_0)$.

\bigskip

The claim, in turn, will be proved by induction on $l$. For $l=0$
(which means that we are computing Legendrian
invariants of a Legendrian left--handed trefoil knot in an overtwisted
$S^3$) the explicit determination of the Alexander gradings of each of the
19 intersection points verifies the statement. To do this we only
need to find domains connecting any two intersection points, and their
relative Alexander grading is the multiplicity of the domain
containing the basepoint $w$. A straightforward linear algebra
argument similar to the one outlined in the proof of
Proposition~\ref{prop:NontrivialLn} shows that all top three degress
contain unique intersection points (giving rise to homology elements),
while the remaining 16 intersection points are of strictly smaller
Alexander grading.

The argument for the inductive step proceeds as follows: Notice that
by deleting the top--most $\alpha$-- and $\beta$--curves 
(containing the point $A$ of Figure~\ref{f:heg-b}) in the
Heegaard decomposition ${\mathfrak {H}}_1$ (and destabilizing the
Heegaard surface) we get a Heegaard decomposition ${\mathfrak {H}}_0$
adapted to $T_{(2,2n-3)}$. Repeating this procedure one more time we
get a Heegaard decomposition ${\mathfrak {H}}_{-1}$, which will be a
decomposition adapted to $T_{(2,2n-5)}$.

Note that the intersection matrix of the 
$\alpha$-- and $\beta$--curves has the following form
$$
\left(\begin{array}{rrrrrrrr}
-2 & 1  & 0        & 0      & 0 & 0 & 0 \\
1  & -2 & 1        & 0      & 0 & 0 & 0 \\
0  & 1  & \ddots   & \ddots & 0 & 0 & 0 \\
0  & 0  & \ddots   & -2     & 1 & 0 & 0 \\
0  & 0  &    0     & 1      & -4 & 2 & 0 \\
0  & 0  &    0     &   0    &  2 & -2 & 1  \\
0  & 0  &    0     &  0     & 0  & 1  &-1
\end{array}
\right)
$$

Now let us group the intersection points of the original diagram into
types $A,B$ and $C$ depending on their coordinate on the top--most
$\beta$--curve. It is easy to see that points in each $A$ and $B$ are
in 1-1 correspondence with intersections in ${\mathfrak {H}}_0$, while
points in $C$ are in 1-1 correspondence with points in ${\mathfrak
  {H}}_{-1}$.  From the form of this matrix it is easy to see that the
relative gradings within the groups coincide with the relative
gradings in their respective Heegaard diagrams ${\mathfrak {H}}_0$ and
${\mathfrak {H}}_{-1}$. As usual, let $\x_0$ denote the intersection
point representing the Legendrian invariant, $B_{\x_0}$ the point with
the same coordinates except on the top-most $\alpha$--circle, where
$A$ is substituted with $B$. Finally, $C_{\x_0}$ will denote the
intersection point where on the top--most $\beta$--circle we choose
$C$ (and so on the top--most $\alpha$--circle we should take $C'$) and
otherwise we take the same intersection points as in $\x_0$. Our
inductive assumption is that $\x_0$, $B_{\x_0}$ and $C_{\x_0}$ are of
highest Alexander gradings in the groups $A,B,C$, respectively.
Therefore to conclude the argument we only need to compare the
gradings of these intersection points.  This relative computation can
be performed locally near the top--most $\alpha$-- and $\beta$--curves
and we get that $A(\x)-A(B_{\x })=2$ and $A(\x )-A(C_{\x })=2$. By the
inductive hypothesis this shows that $\x_0$ is the unique cycle in the
top Alexander grading, concluding the proof.
\end{proof}

Note that the relation found in the proof of Theorem~\ref{t:L0l} 
amongst the Alexander gradings of the intersection points of the
diagrams ${\mathfrak {H}}_i$ ($i=\pm 1, 0$) is a manifestation of the
fact that the symmetrized Alexander polynomial 
$\De_n(t):=\Delta_{T_{(2,2n-1)}}(t)$ satisfies the identity
\[
\Delta _n (t) = (t-t^{-1}) \Delta _{n-1}(t)-\Delta _{n-2}(t).
\]

We consider two more cases, which can be handled by explicit methods:
$L_{1,1}$ and $L_{1,2}$. In these cases the Alexander grading alone is
not sufficient to calculate the homology class; we need a little
analysis of holomorphic disks. We begin with the first of these:

\begin{thm}
Let $L_{1,1}$ be given the orientation specified by the diagram of
Figure~\ref{fig:T27}.  Then, $\Laha(L_{1,1})\neq 0$.
\end{thm}

Recall that $L_{1,1}$ represents the negative torus knot $T_{(2,-7)}$.

\begin{proof}
  Consider the adapted Heegaard diagram for $L_{1,1}$ exhibited in
  Figure~\ref{fig:T27}. 
  \begin{figure}[ht!]
\labellist
\small\hair 2pt
\pinlabel $w$ at 141 18
\pinlabel $z$ at 31 258
\tiny
\pinlabel $L_{1,1}$ at 77 41
\pinlabel $E$ at 82 68
\pinlabel $R$ at 157 99
\pinlabel $S$ at 139 117
\pinlabel $D_1$ at 82 165
\pinlabel $V_1$ at 142 199
\pinlabel $D_2$ at 172 200
\pinlabel $V_2$ at 192 198
\pinlabel $C_1$ at 78 257
\pinlabel $C_2$ at 148 218
\pinlabel $M_1$ at 99 210
\pinlabel $M_2$ at 166 218
\pinlabel $C_3$ at 184 217
\pinlabel $C_4$ at 220 261
\pinlabel $C_5$ at 218 273
\pinlabel $C_6$ at 215 289
\pinlabel $P_2$ at 176 304
\pinlabel $C_7$ at 159 293
\pinlabel $P_1$ at 136 293
\pinlabel $B_3$ at 156 314
\pinlabel $B_2$ at 134 313
\pinlabel $Q_1$ at 113 312
\pinlabel $X$ at 169 399
\pinlabel $Y$ at 131 409
\pinlabel $B_1$ at 85 392
\pinlabel $A_1$ at 76 449
\pinlabel $A_2$ at 148 497
\endlabellist
\centering
\includegraphics[scale=0.8]{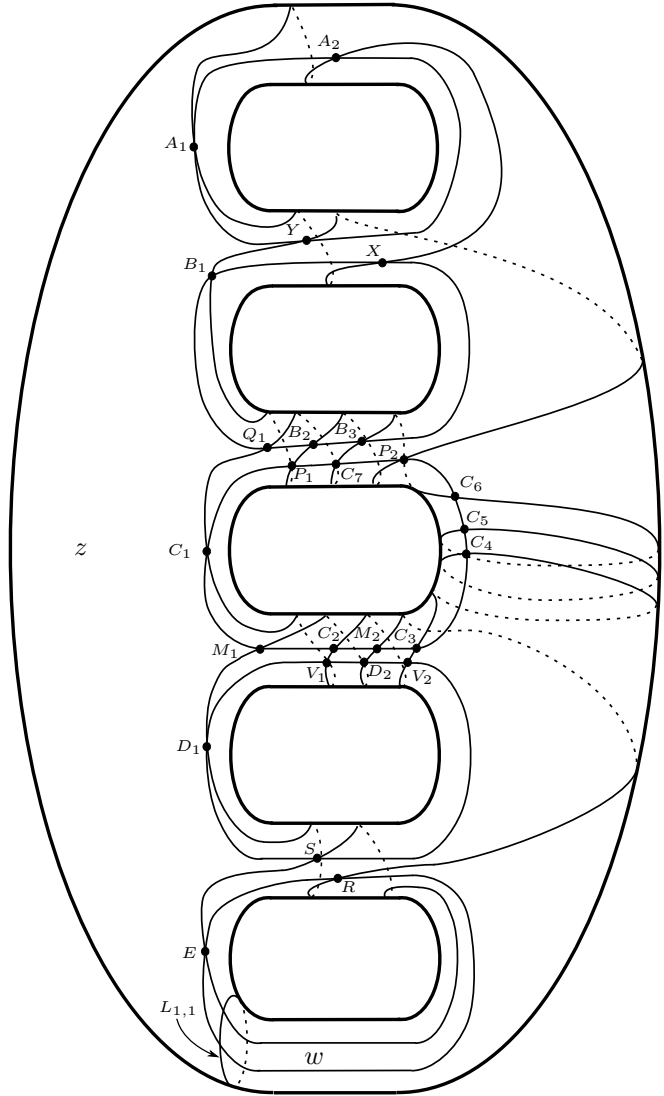}
\caption{{\bf Adapted Heegaard diagram for $L_{1,1}$.}}
\label{fig:T27}
\end{figure}
Analyzing this diagram as in the other cases,
  we see that there are $149$ elements in $\Ta\cap\Tb$, seven of which
  with the same Alexander grading as the point $\x_0$ determining
  $\Laha(L(1,1))$.  Indeed, in the notation suggested by that
  diagram, the seven intersection points in Alexander grading $A=1$
  are 
  \[
  \begin{array}{llll}
    A_{2}B_{2}C_{7} D_{1}E, & A_2 D_1 E P_2 Q_1,& A_2 D_1 E P_1 Q_1, &
    A_2 B_1 C_6 D_1 E, \\ X Y C_6 D_1 E, & A_1 B_2 C_5 D_1 E, & A_1
    B_1 C_1 D_1 E =\x _0 .&
  \end{array}
  \]
  Since $\x_0$ is a cycle, it can be thought of as determining a
  subcomplex of $\CFKa(L_{1,1},1)$.  The quotient complex then is
  generated by the six remaining generators in the given Alexander
  gradings. Indeed, it is straightforward to find the positive
  homotopy classes connecting these six generators with Maslov index
  one. There are six, and five of these are rectangles, and hence
  admit holomorphic representatives. The sixth (connecting $X Y C_6
  D_1 E$ to $A_1 B_2 C_5 D_1 E$) is also a planar surface, and it is
  not difficult to verify directly that it, too, always has a holomorphic
  representative.  We assemble this information of the chain complex
  in Figure~\ref{fig:ComplexT27}. 
\begin{figure}[ht!]
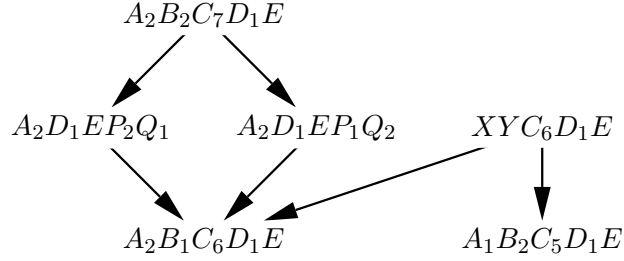

\begin{center}
\setlength{\unitlength}{1mm} \unitlength=1.5cm
\begin{graph}(4,2)(0,0)

\textnode{n1}(0,1){$A_2 D_1 E P_2 Q_1$}[\graphlinecolour{1}]
\textnode{n2}(1,0){$A_2 B_1 C_6 D_1 E$}[\graphlinecolour{1}]
\textnode{n3}(1,2){$A_2 B_2 C_7 D_1 E$}[\graphlinecolour{1}]
\textnode{n4}(2,1){$A_2 D_1 E P_1 Q_2$}[\graphlinecolour{1}]
\textnode{n5}(4,0){$A_1 B_2 C_5 D_1 E$}[\graphlinecolour{1}]
\textnode{n6}(4,1){$XYC_6 D_1 E$}[\graphlinecolour{1}]

\diredge{n1}{n2}
\diredge{n3}{n1}
\diredge{n3}{n4}
\diredge{n4}{n2}
\diredge{n6}{n2}
\diredge{n6}{n5}
 
\end{graph}
\end{center}
\caption{{\bf Part of the chain complex for $L_{1,1}$.}  This is the
  quotient of $\CFKa(L_{1,1},1)$ by the generator $A_1 B_1 C_1 D_1
  E$. Arrows indicate differentials.}
\label{fig:ComplexT27}
\end{figure}
  It is straightforward to verify that
  this complex is acyclic. It follows at once that $\x_0$ represents
  the nontrivial homology class in $\HFKa(L_{1,1},1)$.
\end{proof}

In a similar vein, we have the following:
\begin{thm}
Let $L_{1,2}$ be given the orientation specified by the diagram of Figure~\ref{f:T29}.
Then, $\Laha(L_{1,2})\neq 0$.
\end{thm}

Recall that $L_{1,2}$ corresponds to $T_{(2,-9)}$.

\begin{proof}
The adapted Heegaard diagram is pictured in Figure~\ref{f:T29}.  
\begin{figure}[ht!]
\labellist
\small\hair 2pt
\pinlabel $w$ at 141 18
\pinlabel $z$ at 31 258
\tiny
\pinlabel $L_{1,2}$ at 77 41
\pinlabel $F$ at 84 68
\pinlabel $T$ at 160 98
\pinlabel $S$ at 139 117
\pinlabel $E_1$ at 86 165
\pinlabel $R_1$ at 142 199
\pinlabel $E_2$ at 176 199
\pinlabel $R_2$ at 194 198
\pinlabel $D_5$ at 148 218
\pinlabel $N_1$ at 99 210
\pinlabel $N_2$ at 169 217
\pinlabel $D_4$ at 187 217
\pinlabel $D_3$ at 223 260
\pinlabel $D_2$ at 213 294
\pinlabel $V$ at 170 294
\pinlabel $M$ at 155 309
\pinlabel $C_4$ at 179 310
\pinlabel $D_1$ at 85 257
\pinlabel $C_1$ at 123 308
\pinlabel $P_1$ at 127 396
\pinlabel $C_2$ at 162 388
\pinlabel $P_2$ at 180 401
\pinlabel $C_3$ at 216 382
\pinlabel $B_1$ at 104 491
\pinlabel $Q_1$ at 107 396
\pinlabel $B_2$ at 138 409
\pinlabel $B_3$ at 160 409
\pinlabel $X$ at 176 495
\pinlabel $Y$ at 131 494
\pinlabel $A_1$ at 79 544
\pinlabel $A_2$ at 153 592
\pinlabel $1$ at 251 452
\pinlabel $2$ at 220 529
\pinlabel $3$ at 146 397
\pinlabel $4$ at 94 381
\pinlabel $5$ at 148 413
\pinlabel $6$ at 172 413
\pinlabel $7$ at 113 410
\pinlabel $8$ at 118 502
\pinlabel $9$ at 100 291
\pinlabel $10$ at 191 295
\endlabellist
\centering
\includegraphics[scale=0.75]{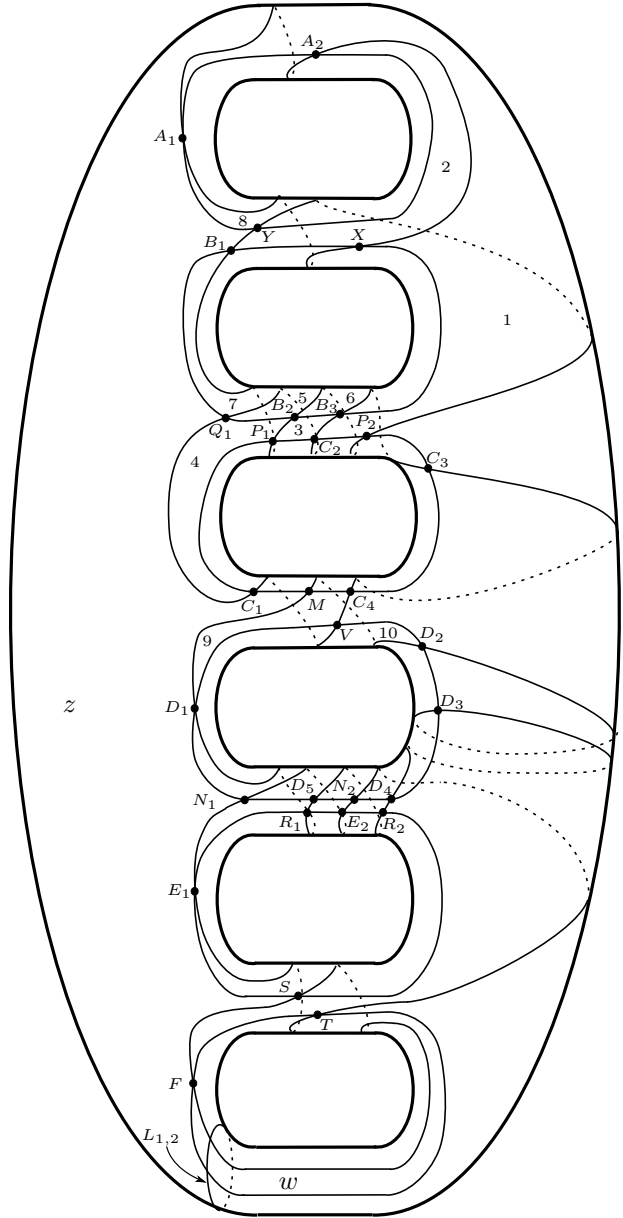}
\caption{{\bf Adapted Heegaard diagram for $L_{1,2}$.}
The domains (in the complement of the $\alpha$-- and $\beta$--circles)
labeled by $1,...,10$  play a role in Figure~\ref{f:T29Complex}.}
\label{f:T29}
\end{figure}
There are $347$ intersection points, 
$13$ of which are in the same Alexander grading
as $\x_0$. In terms of the numbering conventions from
Figure~\ref{f:T29}, these are 
\[\begin{array}{llll}
A_2 B_1 C_3 D_1 E_1 F & A_2 B_2 C_2 D_1 E_1 F & A_1 B_2 C_4 D_1 E_1 F
& A_1 B_2 C_1 D_2 E_1 F \\
A_1 B_1 C_2 D_2 E_1 F & A_2 B_2 C_3 D_2 E_1 F &
A_1 B_2 M V E_1 F & A_1 D_2 E_1 P_1 Q_1 F \\
A_2 D_1 E_1 P_2 E_1 F & A_2 D_1 E_1 P_1 Q_2 F &
A_2 D_2 E_1 P_2 Q_2 F & X Y C_3 D_1 E_1 F \\
A_1 B_1 C_1 D_1 E_1 F=\x _0. & & & 
\end{array}
\]

We can draw a graph whose vertices consist of these generators, and
its arrows denote positive Whitney disks with Maslov index one, as
pictured in Figure~\ref{f:T29Complex}. 
\begin{figure}[ht!]
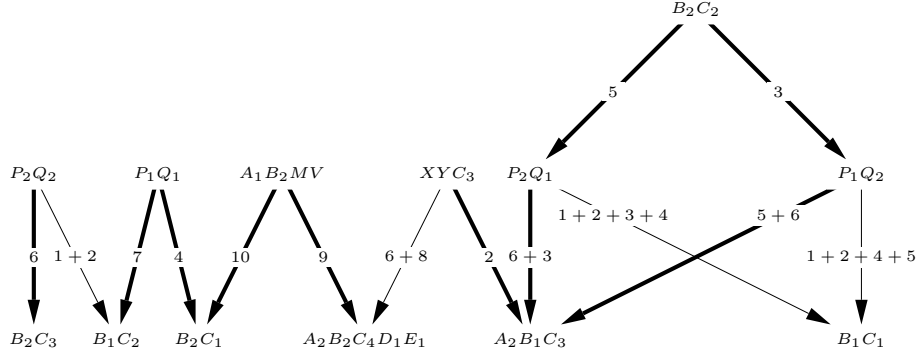

\begin{center}
\setlength{\unitlength}{1mm} 
\unitlength=1.1cm
\begin{graph}(10,4)(0,0)
\graphlinewidth{.05}
\textnode{n1}(0,0){\tiny $B_2C_3$}[\graphlinecolour{1}]
\textnode{n2}(0,2){\tiny $P_2Q_2$}[\graphlinecolour{1}]
\textnode{n3}(1,0){\tiny $B_1C_2$}[\graphlinecolour{1}]
\textnode{n4}(1.5,2){\tiny $P_1Q_1$}[\graphlinecolour{1}]
\textnode{n5}(2,0){\tiny $B_2C_1$}[\graphlinecolour{1}]
\textnode{n6}(3,2){\tiny $A_1B_2MV$}[\graphlinecolour{1}]
\textnode{n7}(4,0){\tiny $A_2B_2C_4D_1E_1$}[\graphlinecolour{1}]
\textnode{n8}(5,2){\tiny $XYC_3$}[\graphlinecolour{1}]
\textnode{n9}(6,0){\tiny $A_2B_1C_3$}[\graphlinecolour{1}]
\textnode{n10}(6,2){\tiny $P_2Q_1$}[\graphlinecolour{1}]
\textnode{n11}(8,4){\tiny $B_2C_2$}[\graphlinecolour{1}]
\textnode{n12}(10,2){\tiny $P_1Q_2$}[\graphlinecolour{1}]
\textnode{n13}(10,0){\tiny $B_1C_1$}[\graphlinecolour{1}]
\diredge{n2}{n1}
\edgetext{n2}{n1}{\tiny $6$}
\diredge{n2}{n3}[\graphlinewidth{.01}]
\edgetext{n2}{n3}{\tiny $1+2$}
\diredge{n4}{n3}
\edgetext{n4}{n3}{\tiny $7$}
\diredge{n4}{n5}
\edgetext{n4}{n5}{\tiny $4$}
\diredge{n6}{n5}
\edgetext{n6}{n5}{\tiny $10$}
\diredge{n6}{n7}
\edgetext{n6}{n7}{\tiny $9$}
\diredge{n8}{n7}[\graphlinewidth{.01}]
\edgetext{n8}{n7}{\tiny $6+8$}
\diredge{n8}{n9}
\edgetext{n8}{n9}{\tiny $2$}
\diredge{n10}{n9}
\edgetext{n10}{n9}{\tiny $6+3$}
\diredge{n11}{n10}
\edgetext{n11}{n10}{\tiny $5$}
\diredge{n11}{n12}
\edgetext{n11}{n12}{\tiny $3$}
\diredge{n10}{n13}[\graphlinewidth{.01}]
\freetext(7,1.5){\tiny $1+2+3+4$}
\diredge{n12}{n9}
\freetext(9,1.5){\tiny $5+6$}
\diredge{n12}{n13}[\graphlinewidth{.01}]
\edgetext{n12}{n13}{\tiny $1+2+4+5$}
\end{graph}
\end{center}
\caption{{\bf Complex for $L_{1,2}$.}  We have displayed here the
        $13$ generators in $A=-3$.  Arrows represent domains
        connecting generators. All the thicker arrows actually
        represent quadrilaterals; there are only four which are not:
        those containing $1$, and those containing $6+8$.}
\label{f:T29Complex}
\end{figure}
Those domains are written as sums of the $10$ domains labelled in
Figure~\ref{f:T29}.  Note that in that figure, we have abbreviated the
names of generators (for example, dropping $F$, which occurs in
each). In that picture, all but four of the arrows represent
rectangles (and hence, immediately, nontrivial differentials in the
chain complex).  Inspection of this graph immediately establishes that
there are no other non--negative Whitney disks among the given
generators with Maslov index one (for example, the domain from $P_2
Q_1=A_2 D_1 E_1 P_2 Q_1 F$ to $B_2 C_4=A_2 B_2 C_4 D_1 E_1$ is gotten
as $6+3-2+8+6$, which contains $2$ with multiplicity $-1$). To show
that $B_1 C_1 = A_1 B_1 C_1 D_1 F$ represents a homologically
nontrivial cycle, it suffices to show that the positive domain $6+8$
supports holomorphic representatives, which can be done by direct
means.  (Note that it is an annulus with an appropriate cut.)
\end{proof}

\begin{rem}
{\rm 
Recall that
\[
\HFKm (S^3, T_{(2,2n-1)})\cong 
\Field ^{n-1}\oplus \Field[U],
\]
where the top generator of the free
$\Field [U]$--module is at 
\[
(A=-(n-1), \quad M=-2(n-1)),
\]
while the $n-1$ generators of the $\Field^{n-1}$ summand are of bi--degrees
\[
(A=n-1-2i, \quad M=-2i),\quad i=0,\ldots, n-2. 
\]
(This follows readily from the calculation of $\HFKa$, stated in
Equation~\eqref{eq:PosTorusKnots}.)  Once again, the computation above
can be adapted to show that the homogeneous $U$--torsion elements
$\La(L_{0,l}), \La(L_{1,1})$ and $\La(L_{1,2})$ in $\HFKm (S^3,
T_{(2,2n-1)})$ are also nontrivial.}
\end{rem}

Notice that there is an obvious bijection between the homogeneous
$U$--torsion elements of $\HFKm(S^3, T_{(2,2n-1)})$ and the Legendrian
knots $L_{k,l}$ with $k+l=n-2$ constructed above.  Moreover, this
bijection can be chosen in such a way that if $L_{k,l}$ corresponds to
$x_{k,l}\in \HFKm(S^3, T_{(2,2n+1)})$ then
\[
d_3(\xi _{k,l})=2A(x_{k,l})-M(x_{k,l}).
\]

It is reasonable to expect that the Legendrian invariants
$\Laha(L_{k,l})$ and $\La(L_{k,l})$ are nonzero for all $k,l\geq 0$,
and hence that $\La(L_{k,l})$ is determined by $x_{k,l}$.  We will not
address this general computation in the present paper.

\section{Connected sums}
\label{s:conn}

Suppose that $L_i\subset (Y_i, \xi _i)$ are oriented Legendrian knots
in the contact 3--manifolds $(Y_i, \xi _i)$, $i=1,2$. We want to relate 
the invariants of $L_1$ and $L_2$ with the invariants of the 
connected sum $L_1\# L_2\subset (Y_1\# Y_2 , \xi _1\# \xi _2)$ (see~\cite{EHI} 
for the definition of the connected sum operation in the contact setting).

Fix open book decompositions $(B_i, \varphi _i)$ adapted to
$L_i\subset (Y_i, \xi _i)$ (with pages $S_i$) for $i=1,2$.  We may
assume that suitable portions of the open books and the adapted bases
appear as in the left--hand picture of Figure~\ref{f:connsum}, where,
as usual, the ordered pairs $(z_i,w_i)$ determine the orientations of
the knots.  We can now perform the Legendrian connected sum of
$(Y_1,\xi_1,L_1)$ and $(Y_2,\xi_2,L_2)$ in such a way that we glue the
open books as well as the contact structures and the knots.  
Specifically, we take the Murasugi sum of the two open books in the domains
containing $z_2$ and $w_1$, and then we drop these two basepoints.
We can make sure that the resulting oriented knot $L_1 \# L_2$ is
smoothly determined on a page of the resulting open book by the
ordered pair $(z_1,w_2)$, as illustrated in the central picture of
Figure~\ref{f:connsum}. The bases of arcs can also be arranged to be
the same as the ones illustrated.

There is a corresponding map of chain complexes
\begin{equation}\label{e:conn}
\Phi \colon \CFKm (-Y_1, L_1)\otimes _{\Field [U]} \CFKm(-Y_2, L_2)\to
\CFKm(-(Y_1\# Y_2), L_1 \# L_2)
\end{equation}
which (denoting the intersection point determining $\La(L_i)$ by
$\x(L_i)$) maps the intersection point $\x (L_1)\otimes\x (L_2)$ to
the generator $(\x (L_1), \x (L_2))$ of the chain complex of the
connected sum. According to \cite[Section 7]{OSzknot} the above map
induces an isomorphism on homology. The union of the two bases is not
adapted to the Legendrian knot $L_1\# L_2$, since each of the two
bases contains an arc intersecting it. Nevertheless, as it will be
shown in the proof of Theorem~\ref{t:connsum}, there is a sequence of
arc slides which carries the new basis of curves into an adapted basis
for $L_1\# L_2$, inducing handleslides on the underlying Heegaard
diagram for $(-(Y_1\# Y_2),L_1\# L_2)$.  These handleslides induce a
map of chain complexes by counting holomorphic triangles
(cf.~\cite[Section 9]{OSzF})
$$\Psi\colon \CFKm(-(Y_1\# Y_2), L_1 \# L_2)\longrightarrow
\CFKm(-(Y_1\# Y_2), L_1 \# L_2) , $$
where the second chain complex now denotes the chain complex 
with respect to the adapted Heegaard diagram.

\begin{thm}\label{t:connsum}
Suppose that $L_1\# L_2\subset (Y_1\# Y_2,\xi _1\# \xi _2 )$ is the
oriented connected sum of the oriented Legendrian knots 
$L_i \subset (Y_i,\xi_i)$, $(i=1,2)$. Then, there is a quasi--isomorphism 
\[
F\colon \CFKm(-Y_1, L_1)\otimes _{\Field[U]} 
\CFKm(-Y_2, L_2)\to \CFKm(-(Y_1\# Y_2), L_1\# L_2)
\]
which maps $\x(L_1)\otimes \x(L_2)$ to $\x(L_1\# L_2)$. 
\end{thm}

\begin{proof}
  There are three Heegaard diagrams for $(-(Y_1\# Y_2),L_1\# L_2)$
  coming into play: the connected sum diagram (whose $\alphas$-- and
  $\betas$--circles are gotten by doubling the initial bases), an
  intermediate diagram gotten by sliding the $\alpha$-circles as
  dictated by the arcslides in the middle diagram in
  Figure~\ref{f:connsum} (whose attaching circles we denote $\alpha '$
  and $\beta$), and the final one gotten by performing handleslides on
  the $\beta$-circles, as dictated by arcslides of the $b_i$ as in the
  rightmost diagram in Figure~\ref{f:connsum} (whose attaching cicles
  are $\alpha'$ and $\beta'$). Recall, that the $a_i$ arcs determine
  the $\alpha$--circles while the $b_i$'s give rise to the
  $\beta$--circles, and we are examining the Heegaard diagrams
  $(\Sigma , \beta , \alpha ), (\Sigma , \beta , \alpha ')$ and
  $(\Sigma , \beta ' ,\alpha ')$.
  \begin{figure}[ht!]
\labellist
\small\hair 2pt
\pinlabel $z_1$ at 20 40
\pinlabel $z_2$ at 20 114
\pinlabel $a_2$ at 35 15
\pinlabel $b_2$ at 61 14
\pinlabel $w_1$ at 47 51
\pinlabel $w_2$ at 43 105
\pinlabel $b_1$ at 32 142
\pinlabel $a_1$ at 56 143

\pinlabel $b_2$ at 184 12
\pinlabel $a_2$ at 135 10
\pinlabel $a'_2$ at 146 11
\pinlabel $b_1$ at 180 115
\pinlabel $a_1$ at 205 124
\pinlabel $a'_1$ at 192 119
\pinlabel $z_1$ at 121 34
\pinlabel $w_2$ at 208 77

\pinlabel $b_2$ at 319 9
\pinlabel $b'_2$ at 299 8
\pinlabel $a'_2$ at 282 8
\pinlabel $b_1$ at 315 110
\pinlabel $b'_1$ at 329 116
\pinlabel $a'_1$ at 340 120
\pinlabel $z_1$ at 260 28
\pinlabel $w_2$ at 341 71

\endlabellist
\centering
\includegraphics[scale=1]{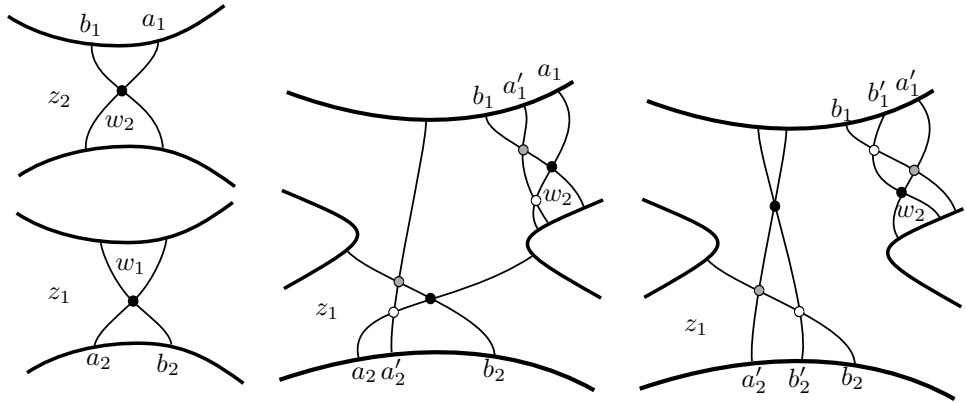}
\caption{{\bf Connected sums}.  Starting from two knots as pictured on
  the left, we form their connected sum to obtain a Heegaard diagram
  as in the middle, where we have three sets of arcs, two sets of
  which are isotopic to one another, and the third is obtained as an
  arcslide. In the middle picture, the dark dots represent the initial
  generator $\x$ (which represents the product of the two Legendrian
  invariants), the white dots represent the canonical generator for
  $\Ta\cap {\mathbb {T}}_{\alpha '}$, and the grey dots represent the
  intermediate generator $\x'$; whereas on the picture on the right,
  the dark dots represent $\x'$, the white dots represent the
  canonical generator of $\Tb\cap {\mathbb {T}}_{\beta '}$, and the
  grey dots represent the final intersection point $\x''$ (which
  represents the Legendrian invariant for the connected sum).}
\label{f:connsum}
\end{figure}
In all three of these diagrams there is a unique intersection point
of the $\alpha$-- and $\beta$--tori which is supported in $S_{+1}$: the
first, $\x$, represents the product of $\x(L_1)$ and $\x(L_2)$ under
the connected sum map (denoted above by $\Phi$), the second is denoted
$\x'$, and the third, $\x''$, clearly determines the Legendrian
invariant for $L_1\# L_2$.  We claim that the three generators are
mapped to one another under the maps induced by handleslides.
  
As a warm-up, we argue first that $\x'$ is, indeed, a cycle. As for
the contact class, we argue that for any $\y\in {\mathbb {T}}_{\beta
'}\cap \Ta$, if $\phi\in\pi_2(\x',\y)$ is any Whitney disk with all
non--negative local multiplicities, then $\x'=\y$ and $\phi$ is the
constant disk.  The argument is similar to the argument
from~\cite{HKM} recalled in the proof of Theorem~\ref{t:contel} with
one slight difference: now the arcs in the basis disconnect $S_{+1}$
into two regions, only one of which contains $z$. However, it is still
easy to see (by another look at Figure~\ref{f:connsum}) that any
postive domain flowing out of $\x'$ which has positive multiplicity on
the other region (not containing $z$) must also have positive
multiplicity at $z$.
  
We now turn to the verification of the claim about the triangle maps.
The Heegaard triple in this case is $(\Sigma , \beta , \alpha , \alpha
', w,z)$, and the diagram $(\Sigma,\alpha ,\alpha ',w,z)$ represents
an unknot in the $g$--fold connected sum of $S^2\times S^1$.
Moreover, in that diagram there is a unique intersection point
$\Theta\in\Ta\cap {\mathbb {T}}_{\alpha '}$ representing the
top--dimensional Floer homology class. The handleslide map
\[
\Psi_1\colon \CFKm(\Sigma,\beta , \alpha ,w,z)\longrightarrow
\CFKm(\Sigma,\beta , \alpha',w,z)
\]
is defined by
\[
\Psi_1({\mathbf u}) = \sum_{\y\in\Tb\cap {\mathbb {T}}_{\alpha '}}
\sum_{\{\psi\in\pi_2({\mathbf u},\Theta,\y)\big|\mu(\psi)=0,
  n_z(\psi)=0\}} \#{\mathfrak M}(\psi) \cdot U^{n_{w}(\psi)} \y,
\]
where, as usual, $\pi_2({\mathbf u},\Theta,\y)$ denotes the space of
homology classes of Whitney triangles at ${\mathbf u}$, $\Theta$, and
$\y$.
    
  The claim that $\Psi_1(\x)=\x'$ follows from the facts that
  \begin{itemize}
  \item there is a triangle $\psi_0\in\pi_2(\x,\Theta,\x')$, gotten as
    a disjoint union of the obvious small triangles in $S_{+1}$, and
  \item any other triangle
    $\psi\in\pi_2(\x,\Theta,\y)$ with $n_z(\psi)=0$ has a negative local
    multiplicity somewhere.
  \end{itemize}
  The first is gotten by glancing at Figure~\ref{f:connsum}.  To see
  the second, notice that any homology class
  $\psi\in\pi_2(\x',\Theta,\y)$ can be decomposed as $\psi_0*\phi$,
  where $\phi\in\pi_2(\x',\y)$. Moreover, if $\psi$ has only positive
  local multiplicities, then the same follows for $\phi$; also, since
  $n_z(\psi)=0=n_{z}(\psi_0)$, it follows that $n_{z}(\phi)=0$, as
  well.  But by our above argument that $\x'$ is a cycle, it now
  follows that $\phi$ is constant.

Consider next the handleslide map
\[  
\Psi_2\colon (\Sigma,\beta ,\alpha ',w,z)
\longrightarrow(\Sigma,\beta ',\alpha ',w,z). 
\]
This is defined by pairing with an intersection point ${\Theta '} \in
{\mathbb {T}}_{\beta '} \cap\Tb$ representing the top--dimensional
nontrivial homology. The argument that $\Psi_2(\x')=\x''$, where
$\x''$ determines the Legendrian invariant for the connected sum,
follows through a similar argument.
  
  The composition $\Psi_2\circ\Psi_1\circ\Phi$ now induces an
  isomorphism on homology, carrying $\x(L_1)\otimes \x(L_2)$
  to $\x(L_1\# L_2)$, which completes the proof of the theorem.
\end{proof}

The above result quickly leads to the following observation, 
providing

\begin{proof}[Proof of Theorem~\ref{t:anystab}]
Interpret stabilization as the connected sum with the stabilized
Legendrian unknot in the standard contact 3--sphere, use the connected
sum formula and the model calculation given in
Section~\ref{s:unknot}. This gives quasi--isomorphisms
\[
\Psi^-\co \CFKm (-Y, L)\to \CFKm (-Y, L^-)
\]
and
\[
\Psi^+\co \CFKm (-Y, L)\to \CFKm (-Y, L^+)
\]
such that 
\[
\Psi ^-(\x(L))=\x((L^-)\quad\text{and}\quad 
U\cdot \Psi^+(\x(L))=\x(L^+),
\]
concluding the proof of the theorem.
\end{proof}

Note that Proposition~\ref{p:stab} is a special case of
Theorem~\ref{t:anystab}; we gave a separate proof of it earlier, since
that proof is somewhat more direct. As an application of our previous
computations we can now prove the following result:

\begin{thm}\label{t:nonlooses3}
Let $\xi _i $ denote the overtwisted contact structure on $S^3$ with
Hopf invariant $d_3(\xi _i )=i\in \bfz$. Then, $\xi _i$ contains
non--loose Legendrian knots with non--vanishing Legendrian invariants
$\Laha$ for each $i\in \bfz$.
\end{thm}

\begin{proof} By Proposition~\ref{p:surgd3} and Lemma~\ref{l:surgd3}, 
for $i>0$ even the Legendrian knot $L_{0, i-1}$, while for $i<0$ odd
the knot $L(i)$ satisfies the claim. If $i=2j-1>0$ is odd then by the
connected sum formula $L_{0,j-1}\# L(1)$ is a good choice. For
$i=2j<0$ even take $L(j)\# L(1)$ and finally for $i=0$ the knot
$L_{0,0}\# L(1)\# L(1)$ will do.
\end{proof}

\begin{cor} 
Suppose that $(Y, \xi )$ is a contact 3--manifold with nonvanishing
contact invariant $c(Y, \xi )$ and $\zeta $ is an overtwisted contact
structure on $Y$ with $\t _{\zeta }=\t _{\xi }$.  Then $(Y, \zeta )$
contains Legendrian (transverse) knots with nonvanishing Legendrian
(resp. transverse) invariants.
\end{cor}
\begin{proof}
Consider a Legendrian unknot $L$ (in a Darboux chart) in $(Y, \xi )$
with $c(Y, \xi )\neq 0$. It is not difficult to find, possibly after
stabilization, an open book decomposition compatible with $(Y, \xi,L)$
together with an adapted basis, in such a way that the two basepoints
$z$ and $w$ determining $L$ are, in fact, in the same domain. This
observation immediately implies $\HFKa (-Y,L)=\HFa (-Y)$, and the
Legendrian invariant $\Laha$ of $L$ is easily seen to be nonzero in
$\HFKa(-Y, L)$, being equal to $c(Y, \xi )$. Any overtwisted contact
structure $\zeta$ on $Y$ with $\t _{\zeta}=\t _{\xi }$ can be given as
$\xi \# \xi _i$, where $\xi _i$ is the overtwisted contact structure
on $S^3$ with $d_3(\xi _i)=i$. This fact easily follows from
Eliashberg's classification of overtwisted contact structures,
together with the classification of homotopy types of 2--plane
fields. Now the connected sum of the Legendrian unknot with the
Legendrian knot having nontrivial Legendrian knot invariant in $\xi
_i$ (provided by Theorem~\ref{t:nonlooses3}) has nontrivial Legendrian
invariant by the connected sum formula.  This knot, and each of its
negative stabilizations are non--loose, and hence its transverse
push--off is non--loose and has nontrivial transverse invariant.
\end{proof}

\section{Transversely non--simple knots} 
\label{s:transverse}

A knot type $\mathcal K$ in $S^3$ is traditionally called
\emph{transversely simple} if two transverse knots $T_1, T_2$ in the
standard contact structure $\xi _{st}$ both of knot type $\mathcal K$
and equal self--linking number are transversely isotopic. Transversely
non--simple knot types were recently found by Etnyre-Honda,
Birman-Menasco \cite{EHAnn, BM}, cf. also \cite{NOT, VV}. 

The notion of transverse simplicity obviously generalizes to any
null--homologous knot type (for the self--linking number to be
well--defined) in any contact 3--manifold $(Y, \xi )$. In this form of
the definition, however, it is easy to find transversely non--simple
knot types provided $\xi $ is overtwisted: The binding of any open
book decomposition compatible with $\xi$ is a null--homologous,
non--loose transverse knot, while the same binding in $(Y\# S^3, \xi
\# \xi _0)$ (where $\xi _0$ is the overtwisted contact structure on
$S^3$ homotopic to $\xi _{st}$ and the connected sum is taken away
from the binding) is loose. Since $\xi $ and $\xi \# \xi _0$ are
homotopic as 2--plane fields, hence isotopic as contact structures,
the binding gives rise to two transverse knots in $\xi$. The
self--linking numbers of these two representatives are clearly equal,
while one has tight, the other overtwisted complement; consequently
the two knots cannot be transversely isotopic.  In conclusion, for
overtwisted contact structures in the definition of transverse
simplicity we should require that the two knots are either both
non--loose or both loose (besides being smoothly isotopic and having
identical self--linking). Two transverse loose knots with equal knot
type and self--linking number, and with a \emph{common} overtwisted
disk in their complement can be shown to be isotopic. Without the
existence of the common overtwisted disk, however, this question is
surprisingly subtle.

With the aid of the transverse invariant $\TransInv$ and the model
calculations for $L_{1,1}$ and $L_{1,2}$ discussed in the previous
section (together with a refinement of the gradings on Heegaard Floer
homology, deferred to Section~\ref{s:app}) we get

\begin{thm}\label{t:nonsimp}
Let ${\mathcal {K}}$ denote the knot type obtained by the connected
sum of the negative torus knots $T_{(2,-7)}$ and $T_{(2,-9)}$.  Then,
in the overtwisted contact 3--manifold $(S^3, \xi _{12})$ with Hopf
invariant $d_3(\xi _{12})=12$ there are two non--loose Legendrian
knots representing ${\mathcal {K}}$, with equal Thurston--Bennequin
and rotation numbers, which are not Legendrian isotopic and stay
nonisotopic after arbitrarily many negative stabilizations.
\end{thm}

The proof of Theorem~\ref{t:nonsimp} proceeds by considering a pair of
appropriate Legendrian representatives of ${\mathcal {K}}$, the computation of
their invariants, and finally the distinction of these elements. Our
candidate Legendrian knots are
\[
L_1= L_{0,2}\# L_{1,2} \quad {\mbox{and}}\quad  L_2=L_{1,1}\# L_{0, 3}.
\]
The knots $L_1$ and $L_2$ are obviously smoothly isotopic to
$T_{(2,-7)}\# T_{(2,-9)}$, and their Thurston--Bennequin and rotation
numbers can be easily deduced from their definition:

\begin{lem}
The Thurston--Bennequin and rotation numbers of $L_1$ and $L_2$ are
\[
\tb (L_1)=\tb (L_2)=-31 \quad {\mbox {and}} \quad \rot (L_1)=\rot (L_2)=-40,
\]
with the orientations specified by Figure~\ref{f:ot}.
\end{lem}
\begin{proof}
The formulae follow from the additivity of rotation numbers and the
identity
\[
\tb (L\# L')=\tb (L)+\tb (L')+1
\]
(cf. \cite[Lemma~3.3]{EHI}), together with the computations of
Lemma~\ref{l:tbrot}.
\end{proof}

Since $L_1$ and $L_2$ are knots in the connected sums of the
corresponding overtwisted contact structures, the claim about the
resulting contact structure follows. (Recall that on $S^3$ an
overtwisted contact structure is determined up to isotopy by its Hopf
invariant, which is additive under connected sum.)  The connected sum
formula of Theorem~\ref{t:connsum} implies that the Legendrian
invariant of both knots is determined by the tensor product of the
invariants of the summands. This argument shows that $\Laha(L_1)$ and
$\Laha(L_2)$ are non--vanishing, hence the knots $L_1, L_2$ are both
non--loose. (Notice that, in general, the connected sum of two
non--loose knots might become loose.) The connected sum formula for
knot Floer homology shows that in the Alexander and Maslov grading
$(A=5, M=-2)$ of the elements $\Laha(L_i)$ the knot Floer homology
group $\HFKa (S^3, T_{(2,7)}\# T_{(2,9)})$ is isomorphic to $\Field
\oplus \Field \oplus \Field$. (Only $\Field \oplus \Field$ comes from
the $\HFKm$--theory, hence capable of containing $\Laha$.) Therefore
to conclude our proof for Theorem~\ref{t:nonsimp} we need to show that
the knots $L_1$ and $L_2$ are not isotopic. This step requires an
auxiliary result from knot Floer homology, showing that elements of
the knot Floer group of a connected sum do remember the Alexander and
Maslov gradings of their components in the summands.  This leads to a
refinement of the invariant, which we explain now.  (The necessary
Heegaard Floer theoretic result will be given in Section~\ref{s:app}.)

Recall that the invariant $\La (L)$ was defined as an element of
$\HFKm (-Y,L)$, determined up to any $\Field [U]$--module isomorphism
of this module.  The ambiguity stems from the fact that for (oriented)
Legendrian isotopic knots $L_1,L_2$ the isotopy is not canonical,
therefore the isomorphism between the knot Floer homologies $\HFKm
(-Y,L_1)$ and $\HFKm (-Y, L_2)$ will depend on the chosen isotopy.  To
prove vanishing and non--vanishing results, it was sufficient to mod
out by all the module automorphisms. When we want to distinguish knots
based on their invariants, however, we might need to understand the
necessary equivalence relations a little better. Notice that an
isotopy between $L_1$ and $L_2$ induces a filtered chain homotopy
between the filtered chain complexes $\CFKm (-Y, L_i)$ (where the
filtration comes from the second basepoint), hence any algebraic
structure of the filtered chain complex provides further restrictions
on the automorphisms of $\HFKm (-Y, L)$ we need to take into account
for the invariance property to hold. For example, the \emph{length} of
a homogeneous element $x$ (that is, the minimal $n$ for which $d_n(x)$
vanishes, where $d_n$ is the $n^{th}$ higher derivative on the
spectral sequence associated to the filtered chain complex $\CFKm (-Y,
L)$) is such a further structure, as it is exploited in
\cite{NOT}. In a slightly different direction, in \cite{OzsSti} two
of the authors of the present paper have shown that the \emph{mapping
  class group}
\[
{\rm {MCG}}(Y, L)={\rm {Diff}}^+(Y,L)/{\rm {Diff}}^+_0(Y,L)
\]
of the complement of the knot admits a natural action on $\HFKm(-Y,
L)$, and it is not hard to see that the Legendrian invariant can be
defined as an orbit of the action of ${\rm {MCG}}(Y,L)$ rather than of
the full automorphism group. In many cases the mapping class group of
a knot is relatively small, hence this refinement provides a
significant sharpening of the invariant --- without even understanding
the explicit action of the mapping class group on the knot Floer
homologies. In some cases, however, the mapping class group can be
large (although might admit a fairly trivial representation on the
Floer homology), in which case the description of the action cannot be
avoided.  For example, connected sums of knots have infinite mapping
class groups.

For connected sums, therefore, we describe a slightly different
refinement of the invariant. For the sake of simplicity, we will be
content with formulating the relevant results for the $\HFKa$--theory only.

\begin{defn}
The $\Field$--vector space $M=\oplus _{s}M_*(s)$ is
\emph{Alexander bigraded} if it admits a splitting
\[
M=\oplus _{s}(\oplus _{\{ s_1,s_2\mid s_1+s_2=s\} } M_* (s_1, s_2)).
\]
The pair $(M,m)$ of an Alexander bigraded vector space $M$ and an
element $m\in M$ defines an equivalence class $[[M,m ]]$ of such
objects by saying that $(M,m)$ and $(N,n)$ are equivalent if there is
an isomorphism $f\colon M \to N$ with the property that $f(m)=n$ and
$f(M_* (s_1, s_2))=N(s_1, s_2)$. In other words, we require the
isomorphisms to respect the Alexander bigrading of the vector space.
\end{defn}

For a null--homologous Legendrian knot $L$ in the contact 
3--manifold $(Y, \xi )$ consider the homology class 
$\al_{\Laha}(L)\in \HFKa(-Y, L,\t_\xi)$ given by 
the cycle $\x (B, \varphi , A)$ as in Definition~\ref{d:Lambda}.
By slightly modifying the definition of Theorem~\ref{t:main}, consider 
\[
{\widetilde {\La }}(L):=
[[\HFKa (-Y, L, \t _{\xi }), \al _{\Laha }(L) ]].
\]

\begin{thm}\label{t:refinv}
Suppose that $L$ is a given null--homologous Legendrian knot in the
contact 3--manifold $(Y, \xi ) $ representing the knot type
${\mathcal {K}}$, which is the connected sum of two knot types
${\mathcal {K}}_1$ and ${\mathcal {K}}_2$. Then for any knot $K$ of
type ${\mathcal {K}}$ the knot Floer homology $\HFKa (-Y, K)$
naturally admits an Alexander bigrading and ${\widetilde {\La}} (L)$ is an
invariant of its oriented Legendrian isotopy class. In particular, if
$L, L'$ are two Legendrian representatives of ${\mathcal {K}}$ and
${\widetilde {\La }}(L)\neq {\widetilde {\La }}(L')$ then $L$ and $L'$
are not Legendrian isotopic.
\end{thm}
\begin{proof}
The K\"unneth formula \cite{OSzknot} provides an Alexander bigrading
for the knot Floer homology of $K=K_1\# K_2$ by using the Alexander
gradings of the knot Floer groups of $K_1$ and $K_2$.
Theorem~\ref{thm:CanonicalSplitting} now shows that the map $f_1$
appearing in the proof of Corollary~\ref{c:l-invariance} induces an
isomorphism on the knot Floer homology which respects this Alexander
bigrading. Therefore the proof of the corollary, in fact, shows that
for connected sums the refined equivalence class ${\widetilde
{\La}}(L)$ is an invariant of the Legendrian isotopy class of $L$.
\end{proof}

\begin{proof}[Proof of Theorem~\ref{t:nonsimp}]
As we saw above, the Legendrian knots $L_1$ and $L_2$ both represent
the knot type of ${\mathcal {K}}$, and have equal Thurston--Bennequin
invariants and rotation numbers.  Now
\[
\HFKa(S^3, m(K),5)= \HFKa(S^3, T_{(2,7)}, 3)\otimes
\HFKa(S^3, T_{(2,9)},2) \oplus 
\]
\[
\HFKa(S^3, T_{(2,7)}, 2)\otimes \HFKa(S^3, T_{(2,9)},3) \oplus
\HFKa(S^3, T_{(2,7)}, 1)\otimes \HFKa(S^3, T_{(2,9)},4),
\]
and the cycle which determines $\Laha(L_1)$ is in the first summand,
while the cycle which determines $\Laha(L_2)$ is in the last.  (The
image of the quotient map from $\HFKm$ to $\HFKa$, which contains
${\Laha}$, is disjoint from the middle summand.) In other words, the
Alexander bigrading of $\Laha (L_1)$ is $(3,2)$, while for $\Laha
(L_2)$ this quantity is $(1,4)$. This readily implies that the refined
invariants ${\widetilde {\La}}(L_1)$ and ${\widetilde {\La }}(L_2)$
are distinct, and so by Theorem~\ref{t:refinv} the knots $L_1$ and
$L_2$ are not Legendrian isotopic.
\end{proof}

\begin{proof}[Proof of Theorem~\ref{t:nonloosenonsimp}]
  Let $K$ denote the connected sum $T_{(2,-7)}\# T_{(2,-9)}$.
  Consider the positive transverse push--offs $T_1$ and $T_2$ of the
  Legendrian knots $L_1$ and $L_2$.  By defining the refined
  transverse invariant ${\tilde {\TransInv }}(T_i)$ as ${\tilde
    {\LegInv }}(L_i)$, the proof of Theorem~\ref{t:nonsimp} implies
  that $T_1$ and $T_2$ are transversely nonisotopic. Since their
  self--linking numbers $s\ell (T_i)$ can be easily computed from the
  Thurston--Bennequin and rotation numbers of the Legendrian knots
  $L_i$, we get that $s\ell (T_1)=s\ell (T_2)$.  Therefore $K$ as
  defined above is a transversely non--simple knot type in the
  overtwisted contact structure $\xi _{12}$, concluding the proof.
\end{proof}

Notice that if the nontriviality of the Legendrian invariants of
all Legendrian knots $L_{i,j}$ can be established, the argument 
used in the proof of Theorem~\ref{t:nonloosenonsimp} 
actually provides arbitrarily many transversely distinct transverse knots
with the same classical invariants: consider the connected sums
\[
L_{i,j}\# L_{k,l}
\]
with $i+j=n-2$ and $k+l=m-2\geq n-2$ fixed (hence fixing the knot type
and the Thurston--Bennequin numbers of the knots) and take only those
knots which satisfy $i+k=n-2$. It is not hard to see that these knots
will have the same rotation numbers, but the argument given in the
proof of Theorem~\ref{t:nonloosenonsimp}, assuming that the invariants
for the $L_{i,j}$ are nontrivial, shows that the transverse invariants
of their positive transverse push--offs do not agree. This would be a
way of constructing arbitrarily many distinct transverse knots with
the same classical invariants in some overtwisted contact $S^3$.

\begin{proof}[Proof of Corollary~\ref{c:nonlooseeverywhere}.]
  Consider the overtwisted contact structure $\zeta$ on $Y$ with $\t
  _{\zeta}=\t _{\xi }$ and write it as $\zeta '\# \xi _{12}$, where
  $\zeta '$ is an overtwisted contact structure on $Y$ with $\t
  _{\zeta '}=\t _{\xi }$.  (Again, by simple homotopy theoretic
  reasons and the classification of overtwisted contact structures,
  the above decomposition is possible.) The connected sum of the
  Legendrian knot $L$ in $\zeta '$ having nontrivial invariant ${\tilde
    {\LegInv}} (L)$ with $L_1$ and $L_2$ of
  Theorem~\ref{t:nonloosenonsimp} obviously gives a pair of Legendrian
  knots with equal classical invariants in $\zeta$, but with the
  property that their transverse push--offs have distinct transverse
  invariants. The distinction of these elements relies on a
  straightforward modification of
  Theorem~\ref{thm:CanonicalSplitting}, where the components in a
  fixed connected sum decomposition are not necessarily prime knots,
  but we fix the isotopy class of the embedded sphere separating the
  knot into two connected components.  This construction concludes the
  proof.
\end{proof}

\section{Appendix: On knot Floer homology of connected sums}
\label{s:app}
Recall from~\cite{OSzknot} that if $(Y_1,K_1)$ and $(Y_2,K_2)$ are two
three--manifolds equipped with knots, then we can form their connected
sum $(Y_1\# Y_2,K_1\# K_2)$. In this case, the knot Floer homology for
the connected sum can be determined by the knot Floer homology of its
summands by a K{\"u}nneth formula. Our aim here is to investigate
naturality properties of this decomposition. Specifically, we prove
the following

\begin{thm}
  \label{thm:CanonicalSplitting}
  Let $K$ be a knot obtained as a connected sum of two distinct prime
  knots $K_1$ and $K_2$. Then there is an additional intrinsic grading
  on the knot Floer homology of $K$ with the property that
  $$\HFKa(K,s)=\bigoplus_{\{s_1,s_2\big| s_1+s_2=s\}}
  \HFKa(K,s_1,s_2),$$ where $\HFKa(K,s_1,s_2)=\HFKa(K_1,s_1)\otimes
  \HFKa(K_2,s_2)$.  More precisely, if $D(K_1)$ and $D(K_2)$ are
  diagrams for $K_1$ and $K_2$ respectively, then there is an induced
  diagram $D(K_1)\# D(K_2)$ for $K_1\# K_2$, together with an
  isomorphism $$\Phi\colon H_*(\CFKa(D(K_1)))\otimes
  H_*(\CFKa(D(K_2)))\longrightarrow H_*(\CFKa(D(K_1)\# D(K_2))).$$ If
  $D'(K_1)$ and $D'(K_2)$ are a pair of different diagrams for $K_1$
  and $K_2$, and $D'(K_1)\#D'(K_2)$ denotes the induced sum of
  diagrams for $K_1\# K_2$, then there is an isomorphism $\Psi\colon
  H_*(\CFKa(D(K_1)\#D(K_2)))\longrightarrow
  H_*(\CFKa(D'(K_1)\#D'(K_2)))$ such that the diagram
\begin{equation}
  \label{eq:CommDiag}
  \begin{CD}
        H_*(\CFKa(D(K_1)))\otimes H_*(\CFKa(D(K_2)))
        @>{\Phi}>>
        H_*(\CFKa(D(K_1)\#D(K_2))) \\
        @V{\psi_1\otimes \psi_2}VV 
        @V{\Psi}VV \\
        H_*(\CFKa(D'(K_1)))\otimes H_*(\CFKa(D'(K_2))) 
        @>{\Phi'}>>
        H_*(\CFKa(D'(K_1)\#D'(K_2)))
        \end{CD}
\end{equation}
commutes.
\end{thm}

Let $Y$ be a closed, oriented three--manifold, $K\subset Y$ be an
oriented knot, and $S\subset Y$ meeting $K$ transversely in exactly
two points.  Call $S$ a {\em splitting sphere} for $K$. A splitting
sphere expresses $K$ as a connected sum of two knots $K_1$ and $K_2$.
In the case where one of the two summands is unknotted, we call the
splitting sphere {\em trivial}. In this language, a knot $K$ is prime
if every splitting sphere for $K$ is trivial.  Recall the following
classical result \cite{Liko}.

\begin{thm}
  If $K$ is the connected sum of two prime knots $K_1$ and $K_2$,
  then there is, up to isotopy, a unique nontrivial splitting sphere
  $S$ for $K$. \qed
\end{thm}

Let $(\Sigma,\alphas,\betas,w,z,\gamma)$ be a doubly-pointed Heegaard
diagram equipped with a curve $\gamma\subset \Sigma$ which is disjoint
from all $\alpha_i$, $\beta_j$, $w$, and $z$.  Suppose moreover that
$\gamma$ is a separating curve, dividing $\Sigma$ into two components
$F_1$ and $F_2$, so that $w\in F_1$ and $z\in F_2$. This {\em decorated
Heegaard diagram} determines the following data:
\begin{itemize}
  \item a three--manifold $Y$ (gotten from the Heegaard diagram)
  \item an oriented knot $K\subset Y$ (determined by $w$ and $z$)
  \item an embedded two--sphere $S$ meeting the Heegaard surface along
    $\gamma$ (consisting of all Morse flows between index zero and
    index three critical points passing through $\gamma$);
\end{itemize}
i.e. a decorated Heegaard diagram determines a knot $K$ in $Y$
together with a splitting sphere $S$. We call
$(\Sigma,\alphas,\betas,w,z,\gamma)$ a {\em decorated Heegaard diagram
compatible with $(Y,K,S)$}.

\begin{defn}
  A {\em decorated Heegaard move} on
  $(\Sigma,\alphas,\betas,w,z,\gamma)$ is a move of one of the
  following types:
  \begin{itemize}
  \item {\em Isotopies.} 
    Isotopies of $\alphas$ and $\betas$, preserving the
    conditions that
    all curves among the $\alphas$ and $\betas$ are disjoint,
    and disjoint from $w$, $z$, and $\gamma$;
    isotopy of $\gamma$, preserving the condition that it
    is disjoint from $\alphas$, $\betas$, $w$, and $z$.
  \item {\em Handleslides.} Handleslides among the $\alphas$ or
    $\betas$, supported in the complement of $w$, $z$, and away from
    $\gamma$.
  \item {\em Stabilizations/destabilizations.} Stabilization is
    obtained by forming the connected sum of
    $(\Sigma,\alphas,\betas,w,z,\gamma)$ with a genus one surface
    equipped with a pair of curves $\alpha_{g+1}$ and $\beta_{g+1}$
    meeting transversely in a single point; destabilization is the
    inverse operation.
  \end{itemize}
\end{defn}

\begin{prop}
  Suppose that $(\Sigma,\alphas,\betas,w,z,\gamma)$ and
  $(\Sigma',\alphas',\betas',w',z',\gamma')$ are two 
  decorated Heegaard diagrams compatible with $(Y,K,S)$ and $(Y',K',S')$
  respectively. If $(Y,K,S)$ is diffeomorphic to $(Y',K',S')$, then
  the decorated Heegaard diagram $(\Sigma',\alphas',\betas',w',z',\gamma')$
  is diffeomorphic to one obtained from $(\Sigma,\alphas,\betas,w,z,\gamma)$
  after a finite sequence of decorated Heegaard moves. 
\end{prop}

\begin{proof}
  Fix a Morse function $f_0$ compatible with
  $(\Sigma,\alphas,\betas,w,z,\gamma)$. Let $f_1$ be a different Morse
  function compatible with $(\Sigma,\alphas,\betas,w,z,\gamma)$, which
  agrees with $f_0$ in a neighborhood of $S\cup K$.  Then we can
  connect them by a generic one-parameter family $f_t$, wherein they
  undergo isotopies, handleslides and
  stabilizations/destabilizations. Since all functions have the
  prescribed behavior at $w$, $z$, and $\gamma$, the Heegaard
  moves will be supported away from $w$, $z$, and $\gamma$.
  Changing $f$ near $S\cup K$ has the effect of isotopies of $\Sigma$
  supported near $w$, $z$, and $\gamma$. 
\end{proof}

\begin{proof}[Proof of Theorem~\ref{thm:CanonicalSplitting}]
  Let $(\Sigma,\alphas,\betas,w,z,\gamma)$ be a Heegaard diagram for
  $K=K_1\# K_2$. Then, our additional grading is defined (up to an
  additive constant) as follows.  Fix a point $m$ near $\gamma$. We
  define $\spinc_1'(\x)-\spinc_1'(\y)=n_{m}(\phi)-n_{w}(\phi)$ where
  $\phi\in\pi_2(\x,\y)$ is any homotopy class; similarly,
  $\spinc_2'(\x)-\spinc_2'(\y)=n_{z}(\phi)-n_{m}(\phi)$. Since the
  handleslides between compatible decorated projections cannot cross
  $\gamma$ (or $w$ or $z$), clearly, the triangle maps induced by
  handleslides preserve these gradings (at least in the relative
  sense).
  
  In fact, if $D(K_1)$ and $D(K_2)$ are diagrams for $K_1$ and $K_2$,
  we can form the connected sum diagram by connecting $z_1\in D(K_1)$
  with $w_2\in D(K_2)$, dropping these two basepoints, and using only
  $w_1$ and $z_2$ in $D(K_1)\# D(K_2)$. This is the doubly-pointed
  Heegaard diagram for the connected sum, and if we draw $\gamma$
  around the connected sum annulus, then it is decorated so as to be
  compatible with the nontrivial sphere $S$ splitting $K$.  Indeed,
  the usual proof of the K\"unneth principle shows that the map
  $$\CFKa(D(K_1))\otimes \CFKa(D(K_2))\longrightarrow \CFKa(D(K_1)\#
  D(K_2))$$ is a quasi-isomorphism. It is also easy to see
  that 
  $$(\spinc_1(\x),\spinc_2(\x))=(\spinc_1'(\x),\spinc_2'(\x)),$$ at
  least up to an overall additive constant.  Indeed, we can now define
  $\spinc_1'$ and $\spinc_2'$ to be normalized so that they agree with
  $\spinc_1$ and $\spinc_2$.  Finally, if we have two diagrams which
  differ by decorated Heegaard moves, then the triangle maps clearly
  induce isomorphisms required in Equation~\eqref{eq:CommDiag}. The
  statement follows at once since $\Psi$ preserves the bigrading (in
  an absolute sense).
\end{proof}

\end{document}